\documentclass{amsart}

\usepackage{amsmath}
\usepackage{amssymb}
\usepackage{amsthm}
\usepackage{amscd}

\newtheorem{thm}{Theorem}[section]
\newtheorem{prop}[thm]{Proposition}
\newtheorem{lem}[thm]{Lemma}
\newtheorem{cor}[thm]{Corollary}

\newtheorem{conj}[thm]{Conjecture}

\theoremstyle{definition}
\newtheorem{dfn}[thm]{Definition}

\theoremstyle{remark}
\newtheorem{rem}{Remark}
\newtheorem*{acknowledgment}{Acknowledgements}

\newcommand{\C}{\mathbb{C}}
\newcommand{\R}{\mathbb{R}}
\newcommand{\Z}{\mathbb{Z}}

\newcommand{\RR}{\mathcal{R}}
\renewcommand{\H}{\mathcal{H}}

\newcommand{\E}{\mathbf{E}}

\newcommand{\Twist}{\mathfrak{Twist}}
\newcommand{\U}{\mathfrak{U}}

\newcommand{\Th}{\mathrm{Th}}
\newcommand{\Map}{\mathrm{Map}}
\newcommand{\Pin}{\mathrm{Pin}}

\newcommand{\pt}{\mathrm{pt}}

\newcommand{\KO}{K\!O}
\newcommand{\KR}{K\!R}

\title[A variant of $K$-theory and topological T-duality]
{A variant of $K$-theory and topological T-duality 
for Real circle bundles}

\author[K. Gomi]{Kiyonori Gomi}

\address{
Department of Mathematical Sciences,
Shinshu University, 
3--1--1 Asahi, Matsumoto, Nagano 390-8621,
Japan.}

\email{kgomi@math.shinshu-u.ac.jp}

\date{}

\begin{document}

\begin{abstract}
For a space with involutive action, there is a variant of $K$-theory. Motivated by T-duality in type II orbifold string theory, we establish that a twisted version of the variant enjoys a topological T-duality for Real circle bundles, i.e.\ circle bundles with real structure. 
\end{abstract}

\maketitle

\tableofcontents


\section{Introduction}


\subsection{A variant of $K$-theory}

Topological $K$-theory has various variants. Among them, the variant $K_\pm(X)$ concerned to this paper is that introduced in Witten's paper \cite{W}. While $K_\pm(X)$ is originally introduced in a context of string theory, it also appears in Rosenberg's K\"{u}nneth formula for equivariant $K$-theory \cite{R2} in a different notation. As is pointed out in \cite{A-H}, we can regard $K_\pm(X)$ as a twisted equivariant $K$-theory \cite{FHT1}. We can also regard $K_\pm$ as a part of an $RO(\Z/2)$-graded equivariant cohomology theory, as in \cite{May1}.

\medskip

The definition of $K_\pm(X)$ itself is very simple: For a space $X$ with an involutive action $\tau : X \to X$, i.e.\ a $\Z/2$-action, $K^n_\pm(X)$ is defined through the $\Z/2$-equivariant $K$-theory \cite{Seg} as follows:
$$
K^n_\pm(X) = K^{n+1}_{\Z/2}(X \times \tilde{I}, X \times \partial \tilde{I}),
$$ 
where $\tilde{I} = [-1, 1]$, and $\Z/2$ acts on $X \times \tilde{I}$ by $(x, t) \mapsto (\tau(x), -t)$. The variant $K^n_\pm(X)$ obeys the Bott periodicity clearly, and fits into the exact sequence:
$$
\begin{CD}
K^1_\pm(X) @<<< K^1(X) @<f<< K^1_{\Z/2}(X) \\
@VVV @. @AAA \\
K^0_{\Z/2}(X) @>f>> K^0(X) @>>> K^0_\pm(X),
\end{CD}
$$
where $f$ is to forget the $\Z/2$-actions. From this, if the $\Z/2$-action is trivial, then
$$
K^n_\pm(X) \cong K^{n+1}(X).
$$
Other basic properties are summarized in the body of this paper.


\subsection{Main theorem}

Our main theorem is a version of `topological T-duality'. This duality originates from T-duality in string theory, and is first discussed by Bouwkneght, Evslin and Mathai \cite{BEM}. Nowadays the duality is generalized to various cases by many authors. Its simplest form following Bunke and Schick \cite{B-S} is as follows: Let $(E, h)$ be a pair consisting of a (principal) circle bundle $\pi : E \to X$ and a cohomology class $h \in H^3(E; \Z)$. Then there exists another pair $(\hat{E}, \hat{h})$ consisting of a circle bundle $\hat{\pi} : \hat{E} \to X$ and $\hat{h} \in H^3(\hat{E}; \Z)$ such that:
\begin{align*}
c_1(\hat{E}) &= \pi_*h, &
c_1(E) &= \hat{\pi}_*\hat{h}, &
K^{h + n}(E) &\cong K^{\hat{h} + n - 1}(\hat{E}),
\end{align*}
where $\pi_* : H^3(E, \Z) \to H^2(X; \Z)$ and $\hat{\pi}_* : H^3(\hat{E}; \Z) \to H^2(X; \Z)$ are the push-forward along the projections, $K^{h + *}(E)$ and $K^{h + *}(\hat{E})$ twisted $K$-theory \cite{D-K,R1}.

\medskip

In our main theorem, circle bundles are generalized to \textit{Real circle bundles}, and twisted $K$-theory to $\Z/2$-equivariant twisted $K$-theory and its variant. 

By a \textit{Real circle bundle}, we mean a principal $S^1$-bundle $\pi : E \to X$ on a space $X$ with $\Z/2$-action such that a lift $\tau : E \to E$ of the $\Z/2$-action on $X$ satisfies $\tau^2(\xi) = \xi$ and $\tau(\xi u) = \tau(\xi) u^{-1}$ for all $\xi \in E$ and $u \in S^1$. This notion is essentially equivalent to that of \textit{Real line bundles}, i.e.\ complex line bundles with real structure in the sense of \cite{A}. If we construct an invariant Hermitian metric on a Real line bundle by average, then its unit sphere bundle gives rise to a Real circle bundle. Conversely, there is the Real line bundle associated to a Real circle bundle. As usual, the choice of a Hermitian metric does not matter in 
classifications.

As is well-known, circle bundles (or complex line bundles) on a space $X$ are classified by the second integral cohomology $H^2(X) = H^2(X; \Z)$. There is an equivariant analogue of this fact (\cite{H-Y}), and $\Z/2$-equivariant circle bundles are classified by the equivariant cohomology $H^2_{\Z/2}(X) = H^2_{\Z/2}(X; \Z)$ in the sense of Borel. There also exists a Real analogue: Real (principal) circle bundles, or equivalently Real (complex) line bundles, are classified by a variant of ordinary cohomology
$$
H^2_\pm(X) 
= H^3_{\Z/2}(X \times \tilde{I}, X \times \partial \tilde{I}; \Z).
$$
This classification is due to Kahn \cite{K}, in view of the fact that the Thom isomorphism identifies $H^2_\pm(X)$ with the equivariant cohomology with local coefficients $H^2_{\Z/2}(X; \Z(1))$, where $\Z(1)$ is the group $\Z$ with the $\Z/2$-action $n \mapsto -n$. We write $c_1^R(E) \in H^2_\pm(X)$ for the cohomology class classifying a Real circle bundle $\pi : E \to X$. Then we get the following Gysin sequence:
$$
\cdots \to
H^{n-2}_{\pm}(X) \overset{c_1^R(E) \cup }{\to}
H^{n}_{\Z/2}(X) \overset{\pi^*}{\to}
H^{n}_{\Z/2}(E) \overset{\pi_*}{\to}
H^{n-1}_{\pm}(X) \to
\cdots.
$$

Now, one of our main theorems concerns a \textit{pair} $(E, h)$ consisting of a Real circle bundle $\pi : E \to X$ and a cohomology class $h \in H^3_{\Z/2}(E)$. By an isomorphism of pairs $\phi : (E, h) \to (E', h')$, we mean an isomorphism of Real circle bundle $ \phi : E \to E'$ covering the identity of $X$ such that $\phi^*h' = h$.

\begin{thm} \label{thm:main_pair}
Let $X$ be a space with $\Z/2$-action.

\begin{itemize}
\item[(a)]
For any pair $(E, h)$ on $X$, there exists another pair $(\hat{E}, \hat{h})$ such that:
\begin{align*}
c_1^R(\hat{E}) &= \pi_*h, &
c_1^R(E) &= \hat{\pi}_*\hat{h}, &
c_1^R(E) \cup c_1^R(\hat{E}) &= 0, &
p^*h &= \hat{p}^*\hat{h},
\end{align*}
where $p$ and $\hat{p}$ are the projections from $E \times_X \hat{E} = \pi^*\hat{E} = \hat{\pi}^*E$:
$$
\begin{array}{c@{}c@{}c@{}c@{}c}
 &  & E \times_{X} \hat{E} & & \\
 & {}^p \swarrow &  & \searrow {}^{\hat{p}} & \\
E & &  & & \hat{E} \\
 & {}_\pi \searrow & & \swarrow {}_{\hat{\pi}} & \\
 &  & X. & & 
\end{array}
$$
We call $(\hat{E}, \hat{h})$ a pair T-dual to $(E, h)$.

\item[(b)]
The isomorphism class of a pair $(\hat{E}, \hat{h})$ T-dual to $(E, h)$ is unique.

\item[(c)]
The isomorphism class of a pair $(\hat{E}, \hat{h})$ T-dual to $(E, h)$ depends only on the isomorphism class of $(E, h)$.

\item[(d)]
Let $(E, h)$ be a pair on $X$, and $(\hat{E}, \hat{h})$ its T-dual pair. Then $(\hat{E}, \hat{h} + \hat{\pi}^*\eta)$ is a pair T-dual to $(E, h + \pi^*\eta)$ for any $\eta \in H^3_{\Z/2}(X; \Z)$.

\item[(e)]
Let $\mathfrak{Pair}(X)$ denote the isomorphism classes of pairs on $X$. The assignment of a T-dual pair $(E, h) \mapsto (\hat{E}, \hat{h})$ induces a natural $\Z/2$-action on $\mathfrak{Pair}(X)$ compatible with the action of $H^3_{\Z/2}(X; \Z)$ in (d).
\end{itemize}
\end{thm}

Given $h \in H^3_{\Z/2}(X) = H^3_{\Z/2}(X; \Z)$, we write $K^{h + n}_{\Z/2}(X)$ for the twisted equivariant $K$-theory. According to its definition, the variant also gets twisted: $K^{h + n}_\pm(X)$. 

Generally, we can also consider a twist by $H^1_{\Z/2}(X; \Z/2)$. Such twists will be omitted for simplicity in the present paper, except for a special case: Let $\underline{\R}_1$ be the real line bundle $X \times \R$ with the $\Z/2$-action $\tau(x, \xi) = (\tau(x), - \xi)$, and $w_1^{\Z/2}(\underline{\R}_1) \in H^1_{\Z/2}(X; \Z/2)$ its equivariant first Stiefel-Whitney class. Then, by the Thom isomorphism theorem, we have
\begin{align*}
K^{w_1^{\Z/2}(\underline{\R}_1) + h + n}_{\Z/2}(X) 
&\cong K^{h + n}_{\pm}(X), &
K^{w_1^{\Z/2}(\underline{\R}_1) + h + n}_{\pm}(X) 
&\cong K^{h + n}_{\Z/2}(X),
\end{align*}
which allows us to handle the twist by $w_1^{\Z/2}(\underline{\R}_1)$ implicitly. 

We now state the other main theorem:

\begin{thm} \label{thm:main_T_transformation}
Let $(E, h)$ be a pair on a finite $\Z/2$-CW complex $X$, and $(\hat{E}, \hat{h})$ a pair T-dual to $(E, h)$. For any $n \in \Z$, there exist $K^*_{\Z/2}(X)$-module isomorphisms
\begin{align*}
T &: \ K^{h + n}_{\Z/2}(E) \to K^{\hat{h} + n - 1}_{\pm}(\hat{E}), &
T &: \ K^{h + n}_{\pm}(E) \to K^{\hat{h} + n - 1}_{\Z/2}(\hat{E}).
\end{align*}
\end{thm}

\medskip

The notion of a $\Z/2$-CW complex is a generalization of that of a CW complex \cite{May1}. In general, for a compact Lie group $G$, a $G$-CW complex is a space with $G$-action constructed from $G$-cells and attaching $G$-maps as in the case of a usual CW complex. A $G$-cell is a $G$-space of the form $G/H \times e^d$, where $H \subset G$ is a closed subgroup, $e^d$ is a usual $d$-dimensional cell, and $G$ acts on $G/H$ by the left multiplication but on $e^d$ trivially. A compact smooth manifold with smooth $G$-action provides an example of a finite $G$-CW complex.

\medskip

If $\Z/2$ acts on $X$ freely, then a Real (principal) circle bundle on $X$ yields a (general) circle bundle on $X/\Z_2$ in the sense of Baraglia \cite{Bar2}, and our main theorem recovers his topological T-duality for such circle bundles.

\medskip

The proof of our main theorem is topological: To prove Theorem \ref{thm:main_pair}, we employ ideas of Baraglia \cite{Bar2} as well as Bunke and Schick \cite{B-S}. We are not using the idea of a classifying space in \cite{B-S}, but such an approach is possible (see Subsection \ref{subsec:universal_pair}). The proof of Theorem \ref{thm:main_T_transformation} is parallel to that in \cite{B-S}, except for a modification of the T-dual transformation due to the fact that a Real circle bundle is not necessarily equivariantly $K$-orientable.


\subsection{Background from string}

We here try to explain the background of the main theorem: orbifolding of string theory. 

It is widely recognized \cite{W} that the Ramond-Ramond charges of D-branes in the superstring theory of type IIA, IIB and type I live in the even $K$-theory $K^0(X)$, the odd $K$-theory $K^1(X)$, and $\KO$-theory $\KO(X)$, respectively. T-duality is one of dualities relating superstring theories (see e.g.\ \cite{V}). For instance, IIA and IIB theories are T-dual on the toroidal compactifications along $1$-dimension. This means that IIA and IIB theory, which are formulated on $\R^{10}$ originally, give rise to `equivalent' theories on $\R^9 \times S^1$. This equivalence is compatible with the $K$-theoretic description of D-brane charges, since there is an isomorphism \cite{Hori}:
$$
K^n(X \times S^1) \cong K^{n-1}(X \times S^1).
$$
The work of Bouwkneght, Evslin and Mathai \cite{BEM} came from an attempt to generalize the above relation of $K$-theory on the trivial circle bundle $X \times S^1 \to X$ to non-trivial bundles in a way consistent with other background of string theory.

\medskip

Aside from a compactification, there is another recipe to produce a theory from string theory, called \textit{orbifolding} (see \cite{Dab,Sen}). Roughly, this recipe is to take into account a symmetry of a superstring theory. An action of a group on the spacetime $\R^{10}$ gives rise to a symmetry. Apart from such an `external' symmetry, there also exist `internal' symmetries. For example, we can switch the orientation of a string. This $\Z/2$-symmetry, which makes sense in IIB theory, is often denoted by $\Omega$. Another internal symmetry is the $\Z/2$-symmetry $(-1)^{F_L}$ in type II theory, which acts according to the left moving spacetime fermion number. In general, we use a mixture of these symmetries to orbifold. In the case where $\Omega$ is included, an orbifolding is called an orientifolding.

Upon orbifolding, the home of D-brane charges, i.e.\ $K$-theory, is to be modified accordingly \cite{W}: For the external symmetry given by a group action, the modification of $K$-theory is equivariant $K$-theory. For $\Omega$, the modification is $\KR$-theory \cite{A}. Then, for $(-1)^{F_L}$, the modification is the variant $K_\pm$.

Note that orbifolding happens to relate string theories also. For example, the orientifolding of IIB theory by $\Omega$ gives rise to type I theory, and the orbifolding of IIB theory by $(-1)^{F_L}$ to IIA theory. These relations of theories are compatible with those of $K$-theories: If $\Z/2$ acts on a space $X$ trivially, then there are isomorphisms $\KR^n(X) \cong \KO^n(X)$ and $K^n_\pm(X) \cong K^{n+1}(X)$.

\medskip

Now, we come to the point. The compatibility of dualities and orbifolding is an issue of physicists, and is tested in various cases. As a case where the compatibility is valid \cite{Dab,G-S,Sen}, we consider the $2k$-dimensional torus $T^{2k} = \R^{2k}/\Z^{2k}$ with the $\Z/2$-action $I_{2k} : \vec{x} \to - \vec{x}$. Then, by T-duality, the orbifolding of IIA theory compactified on $T^{2k}$ by $I_{2k}$ is related to the orbifolding of IIB theory compactified on $T^{2k}$ by $(-1)^{F_L}I_{2k}$, the mixture of the spatial symmetry $I_{2k}$ and the internal symmetry $(-1)^{F_L}$. In this case, the compatibility of $K$-theories is also verified in \cite{G-S}. Actually, for any $\Z/2$-space $X$ and the torus $T^\ell$ with the $\Z/2$-action $\vec{x} \mapsto - \vec{x}$, a repeated use of the Gysin sequence provides us the isomorphisms:
\begin{align*}
K^n_{\Z/2}(X \times T^\ell) &\cong
(K^n_{\Z/2}(X) \oplus K^{n-1}_\pm(X))^{\oplus 2^{\ell-1}}, \\
K^n_{\pm}(X \times T^\ell) &\cong
(K^n_\pm(X) \oplus K^{n-1}_{\Z/2}(X))^{\oplus 2^{\ell-1}}.
\end{align*}

In view of these isomorphisms and the original topological T-duality, it is now natural to anticipate our topological T-duality.


\subsection{Outline of the paper}

This paper is roughly divided into two parts: The former part reviews $H_\pm$ and $K_\pm$, and the latter proves the main theorems.

In Section \ref{sec:H_pm}, we begin with Borel's equivariant cohomology (with local coefficients), and then review basic properties of $H_\pm$ including the Thom isomorphism theorem. The Chern classes of Real vector bundles are also reviewed here. In Section \ref{sec:K_pm}, we recall equivariant twisted $K$-theory briefly and then introduce the twisted version of $K_\pm$. Though is parallel to $H_\pm$, some basic properties of $K_\pm$ are summarized for later convenience. After these preliminary, in Section \ref{sec:torus}, we determine the ring structure of $K^*_{\Z/2}(T^2) \oplus K^*_\pm(T^2)$, whose corollary is necessary for the proof of Theorem \ref{thm:main_T_transformation}. Then, in Section \ref{sec:topological_T_duality}, we prove our main theorems. An example illustrating the theorems, a construction of the classifying space for pairs, and a proposal of a possible topological T-duality are also given in this section. Finally, in Appendix \ref{sec:appendix}, we prove Kahn's classification of Real line bundles \cite{K}, supplying some details omitted in his original French paper.


\subsection{Convention}

Throughout the paper, we denote by $\Z/2 = \{ 0, 1 \}$ the cyclic group of order $2$. We also use the notation $\Z_2 = \{ \pm 1 \}$ to mean $\Z/2$, in the context where the group product is treated multiplicative, or in order to suppress notations. The trivial $1$-dimensional real representation of $\Z/2$ will be denoted by $\R_0$, while the non-trivial one by $\R_1$ or $\tilde{\R}$. The interval $[-1, 1] \subset \R_1$ will be written as $\tilde{I}$. 

Similarly, the trivial $1$-dimensional complex representation of $\Z/2$ will be denoted by $\C_0$ and the non-trivial one by $\C_1$. The representation ring $R(\Z/2)$, generated by complex representations of $\Z/2$, will be identified with $R = \Z[t]/(t^2 - 1)$. The ideals $I$ and $J$ in $R$ are defined by $I = (1 - t)$ and $J = (1 + t)$. Some details of these ideals can be found in the appendix of \cite{MD-R}.

All spaces in this paper are assumed to be locally contractible, paracompact and completely regular, as in \cite{FHT1}. For a $\Z/2$-space $X$, i.e.\ a space equipped with a $\Z/2$-action, we write $\tau : X \to X$ for the action of the  non-trivial element in $\Z/2$. According to the context, we also use the notation $g x$ to mean the point $x \in X$ acted by $g \in \Z_2$. Given another space $Y$ with $\Z/2$-action, the direct product $X \times Y$ is always given the $\Z/2$-action $(x, y) \mapsto (\tau(x), \tau(y))$. We will write $\mathrm{pt}$ for the space consisting of a single point, on which $\Z/2$ acts trivially. For a (real or complex) representation $V$ of $\Z/2$, we write $\underline{V}$ for the $\Z/2$-equivariant vector bundle on $X$ such that its underlying space is $X \times V$ and $\Z/2$ acts by $\tau(x, v) = (\tau(x), \tau(v))$. Basically, $S^1$ stands for the circle with the trivial $\Z/2$-action, and $\tilde{S}^1$ for the circle $S^1 \subset \C$ with the $\Z/2$-action $u \mapsto \bar{u} = u^{-1}$.

Finally, we use `Real' to mean the real structure in the sense of \cite{A}.


\section{A variant of cohomology: $H_\pm$}
\label{sec:H_pm}

This section reviews $H_\pm$: After recalling Borel's equivariant cohomology (with local coefficients), we review the definition of $H_\pm$, its basic properties, its interpretation as a cohomology with local coefficients, the Thom isomorphism for Real line bundles, and the Chern classes of Real vector bundles.

\subsection{Review of Borel's equivariant cohomology}

To begin with, we recall Borel's equivariant cohomology: Let $E\Z_2 \to B\Z_2$ denote the universal principal $\Z/2$-bundle. For a space $X$ with $\Z/2$-action, we write $E\Z_2 \times_{\Z_2} X$ for the quotient of $E\Z_2 \times X$ by the $\Z/2$-action $\tau(\xi, x) = (\tau(\xi), \tau(x))$. Then, for $n \in \Z$, we define the $\Z/2$-equivariant cohomology group to be the singular cohomology group of $E\Z_2 \times_{\Z_2} X$:
$$
H^n_{\Z/2}(X) = H^n_{\Z/2}(X; \Z) = H^n(E\Z_2 \times_{\Z_2} X; \Z).
$$
The cup product in singular cohomology induces that in the equivariant cohomology $H^*_{\Z/2}(X)$, making it into a (graded-commutative) ring. The collapsing map $X \to \mathrm{pt}$ induces the projection $E\Z_2 \times_{\Z/2} X \to B\Z_2$, and hence a ring homomorphism $H^*_{\Z/2}(\mathrm{pt}) \to H^*_{\Z/2}(X)$. With this ring homomorphism and the cup product, $H^*_{\Z/2}(X)$ gives rise to a module over the ring $H^n_{\Z/2}(\mathrm{pt}) = H^*(B\Z_2) \cong \Z[t]/(2t)$, where $t \in H^2(B\Z_2)$.

\medskip

Since the singular cohomology constitutes a (generalized) cohomology theory, the equivariant cohomology also constitutes a (generalized) $\Z/2$-equivariant cohomology theory: The relative cohomology $H^*_{\Z/2}(X, Y)$ is defined for a pair $(X, Y)$ consisting of a $\Z/2$-space $X$ and its invariant closed subspace $Y$. Then the homotopy axiom, the excision axiom, the exactness axiom as well as the additivity axiom hold true. 

Associated to this $\Z/2$-equivariant cohomology, the reduced cohomology theory also makes sense: In the case where a $\Z/2$-space $X$ has a fixed point $\pt \in X$, the reduced theory $\tilde{H}^n_{\Z/2}(X)$ is defined as the kernel of the homomorphism $H^n_{\Z/2}(X) \to H^n_{\Z/2}(\mathrm{pt})$ induced from the inclusion $\pt \to X$. Then we have the natural decomposition $H^n_{\Z/2}(X) \cong \tilde{H}^n_{\Z/2}(X) \oplus H^n_{\Z/2}(\mathrm{pt})$ as usual. We also have an isomorphism $\tilde{H}^n_{\Z/2}(\Sigma X) \cong \tilde{H}^{n-1}_{\Z/2}(X)$, where we let $\Z/2$ act on an interval $I$ trivially and define the reduced suspension $\Sigma X$ to be $\Sigma X = X \wedge S^1 = (X \times I)/(\mathrm{pt} \times I \cup X \times \partial I)$.

\medskip

By the similar argument as above, $H^*_{\Z/2}(X)$ enjoys the Thom isomorphism theorem for $\Z/2$-equivariant real vector bundles admitting $\Z/2$-equivariant orientation, such as $\Z/2$-equivariant complex vector bundles. From the Thom isomorphism theorem, the Gysin exact sequence is derived.

Notice that, for a $\Z/2$-equivariant real vector bundle $V \to X$, the obstruction to the equivariant orientability is the equivariant first Stiefel-Whitney class $w^{\Z/2}_1(V) \in H^1_{\Z/2}(X; \Z/2) = H^1_{\Z/2}(E\Z_2 \times_{\Z_2} X; \Z/2)$. In the case where $V$ is equivariantly non-orientable, the Thom isomorphism theorem and the Gysin sequence involve equivariant cohomology with local coefficients.

\subsection{Equivariant cohomology with local coefficients}
\label{subsec:cohomology_with_local_coefficients}

As the coefficients of the usual singular cohomology are generalized to local coefficients (see \cite{FHT4,Spa} and \cite{Bar2} for example), the coefficients of Borel's equivariant cohomology are also generalized to local coefficients. 

Let $X$ be a space with $\Z/2$-action. In general, a module $\mathcal{M}$ over the group $\pi_1(E\Z_2 \times_{\Z_2} X)$ defines a local system on $E\Z_2 \times_{\Z_2} X$. Using this local system, we define the equivariant cohomology with local coefficients in $\mathcal{M}$:
$$
H^n_{\Z/2}(X; \mathcal{M}) = H^n(E\Z_2 \times_{\Z_2} X; \mathcal{M}).
$$

We are particularly interested in modules $\mathcal{Z}$ whose underlying groups are $\Z$. In this case, the module structure on $\mathcal{Z}$ defines a homomorphism of groups $\pi_1(E\Z_2 \times_{\Z_2} X) \to \Z/2$, and vice verse. Such homomorphisms are in one to one correspondence with elements in $H^1_{\Z/2}(X; \Z/2)$. Thus, identifying the local system $\mathcal{Z}$ with an element $\omega \in H^1_{\Z/2}(X; \Z/2)$, we put
$$
H^{\omega + n}_{\Z/2}(X) = H^n_{\Z/2}(X; \mathcal{Z}).
$$
This notation, which regards a local system as a kind of grading, is justified by the existence of the cup product
$$
\cup : \ H^{\omega + n}_{\Z/2}(X) \times H^{\omega' + n'}_{\Z/2}(X)
\longrightarrow H^{\omega + \omega' + n + n'}_{\Z/2}(X).
$$
The cohomology $H^1_{\Z/2}(X; \Z/2)$ classifies $\Z/2$-equivariant real line bundles $L \to X$ through their equivariant first Stiefel-Whitney class. Thus, such a line bundle $L$ also plays a role of a local system. In this case, we write
$$
H^{L + n}_{\Z/2}(X) = H^{w^{\Z/2}_1(L) + n}_{\Z/2}(X).
$$

With the notation above, the Thom isomorphism theorem is generalized as follows: Let $\pi : V \to X$ be a $\Z/2$-equivariant real vector bundle of rank $r$. We write $\det V$ for the equivariant real line bundle given by the determinant line bundle of $V$. Then, for any $n \in \Z$, there is an isomorphism of $H^*_{\Z/2}(X)$-modules
$$
H^{L + n}_{\Z/2}(X) 
\cong H^{\pi^*(L \otimes \det V) + n + r}_{\Z/2}(D(V), S(V)),
$$
where $D(V)$ and $S(V)$ are the disk bundle and the sphere bundle associated to a $\Z/2$-invariant Riemannian metric on $V$. As usual, the Thom class $\Phi_{\Z/2}(V) \in H^{\pi^*\det V + r}_{\Z/2}(D(V), S(V))$ of $V$ is the image of $1 \in H^0_{\Z/2}(X)$ under the Thom isomorphism. Then the Thom isomorphism $H^{L + n}_{\Z/2}(X) \to H^{\pi^*(L \otimes \det V) + n + r}_{\Z/2}(D(V), S(V))$ is realized as $x \mapsto \pi^*x \cup \Phi_{\Z/2}(V)$.

\bigskip

On a space $X$ with $\Z/2$-action, there always exists a particular local system $\Z(m)$ for each $m \in \Z$. This appears in a number of works such as \cite{K,P-S}, and is defined to be the group $\Z$ on that $\pi_1(E\Z_2 \times_{\Z_2} X)$ acts through the natural homomorphism in the homotopy exact sequence for the fibration $X \to E\Z_2 \times_{\Z_2} X \to B\Z_2$:
$$
\pi_1(X) \to
\pi_1(E\Z_2 \times_{\Z_2} X) \to
\pi_1(B\Z_2) = \Z/2.
$$
The local system $\Z(m)$ corresponds to the $\Z/2$-equivariant real line bundle $\underline{\R}_m$ such that its underlying real line bundle is $X \times \R$ and the $\Z/2$-action is $\tau(x, r) = (\tau(x), (-1)^m r)$. Note that $\underline{\R}_m$ on $X$ is the pull-back of $\underline{\R}_m$ on $\pt$.

Taking $\Z(m)$ or equivalently $\underline{\R}_m$ as a local system, we get
$$
H^{\underline{\R}_m + n}_{\Z/2}(X) = H^n_{\Z/2}(X; \Z(m)).
$$
These cohomology groups assemble to give the ring $\bigoplus_{m, n} H^{\underline{\R}_m + n}_{\Z/2}(X)$ graded by the real representation ring $RO(\Z/2) = \Z \oplus \Z$. We can also regard that this ring is graded by $\Z \oplus \Z/2$ due to the periodicity $\Z(m + 2) = \Z(m)$ or $\underline{\R}_{m+2} = \underline{\R}_m$.

\medskip

It should be noticed that $\underline{\R}_1$ may become trivial depending on the base space $X$. For example, if $X = Y \sqcup Y$ is the disjoint union of two copies of a space $Y$ and $\Z/2$ acts on $X$ by exchanging the copies, then $w^{\Z/2}_1(\underline{\R}_1) = 0$. On the other hand, if $X$ is connected, then $w^{\Z/2}_1(\underline{\R}_1) \neq 0$.

\bigskip

For computations of $H^*_{\Z/2}(X; \Z(m))$, useful is the Serre spectral sequence for the fibration $X \to E\Z_2 \times_{\Z_2} X \to B\Z_2$, which reads
$$
E_2^{p, q} = H^p_{\mathrm{group}}(\Z_2; H^q(X; \Z) \otimes \Z(m))
\Longrightarrow H^*_{\Z/2}(X; \Z(m)).
$$
In the above, $H^*_{\mathrm{group}}(\Z_2; \mathcal{M})$ stands for the \textit{group cohomology} \cite{Bro} of $\Z_2$ with its coefficients in a $\Z/2$-module $\mathcal{M}$, and the cohomology $H^q(X; \Z)$ is regarded as a $\Z/2$-module by the $\Z/2$-action induced from that on $X$. It is know that:
\begin{align*}
H^n_{\mathrm{group}}(\Z_2; \Z(0))
&\cong
\left\{
\begin{array}{cl}
\Z, & (n = 0) \\
\Z/2, & (\mbox{$n > 0$ even}) \\
0, & (\mbox{otherwise})
\end{array}
\right. \\
H^n_{\mathrm{group}}(\Z_2; \Z(1))
&\cong
\left\{
\begin{array}{cl}
\Z/2, & (\mbox{$n > 0$ odd}) \\
0. & (\mbox{otherwise})
\end{array}
\right.
\end{align*}

\subsection{Definition of $H_\pm$}
\label{subsec:definition_variant}

\begin{dfn}
For a space $X$ with $\Z/2$-action and $n \in \Z$, we put
$$
H^n_\pm(X) =
H^n_\pm(X; \Z) = 
H^{n+1}_{\Z/2}
(X \times \tilde{I}, X \times \partial{\tilde{I}}; \Z),
$$
where $\tilde{I} = [-1, 1]$ is endowed with the $\Z/2$-action $t \mapsto -t$.
\end{dfn}

By the definition based on the $\Z/2$-equivariant cohomology, $H^n_\pm(X)$ is a (graded) module over the ring $H^*_{\Z/2}(\mathrm{pt})$. Moreover, $H^n_\pm(X)$ is a (graded) module over the ring $H^*_{\Z/2}(X)$. We can further introduce a graded commutative multiplication
$$
H^m_\pm(X) \times H^n_\pm(X) \to H^{m+n}_{\Z/2}(X).
$$
This is the composition of 
\begin{align*}
H^m_\pm(X) \times H^n_\pm(X) 
&=
H^{m+1}_{\Z/2}(X \times \tilde{I}, X \times \partial \tilde{I}) \times
H^{n+1}_{\Z/2}(X \times \tilde{I}, X \times \partial \tilde{I}) \\
&\to
H^{m + n + 2}_{\Z/2}(X \times \tilde{I}^2, X \times \partial \tilde{I}^2) \\
&\cong 
H^{m + n}_{\Z/2}(X),
\end{align*}
in which the Thom isomorphism for the equivariant real vector bundle $\underline{\C}_1 \cong \underline{\R}_1 \oplus \underline{\R}_1$ is applied. By construction, the multiplication is compatible with the structure of the module over $H^*_{\Z/2}(X)$. Hence the ring $H^*_{\Z/2}(X)$ and the module $H^*_\pm(X)$ combine to form a $\Z \oplus \Z/2$-graded ring
$$
\mathbb{H}^*(X) = H^*_{\Z/2}(X) \oplus H^*_\pm(X).
$$

\medskip

In the case where $Y \subset X$ is an invariant closed subspace, we define the relative group $H^n_{\pm}(X, Y)$ as follows:
$$
H^n_\pm(X, Y) = H^n_\pm(X, Y; \Z)
= H^{n + 1}_{\Z/2}(X \times \tilde{I}, 
Y \times \tilde{I} \cup X \times \partial \tilde{I}; \Z).
$$
Using this relative version, we can recover $H^*_{\Z/2}$ from $H^*_\pm$:

\begin{lem} \label{lem:recover_equivariant}
For any space $X$ with $\Z/2$-action and $n \in \Z$, we have a natural $H^*_{\Z/2}(X)$-module isomorphism
$$
H^n_{\Z/2}(X) \cong 
H^{n+1}_\pm(X \times \tilde{I}, X \times \partial \tilde{I}).
$$
\end{lem}

\begin{proof}
The Thom isomorphism for $\underline{\C}_1$ and the definition of $H^*_\pm$ give:
\begin{align*}
H^n_{\Z/2}(X) 
&\cong
H^{n+2}_{\Z/2}(D(\underline{\C}_1), S(\underline{\C}_1)) \\
&\cong H^{n+2}_{\Z/2}(X \times \tilde{I}^2, X \times \partial \tilde{I}^2) 
\cong H^{n+1}_\pm(X \times \tilde{I}, X \times \partial \tilde{I}),
\end{align*}
which respects the $H^*_{\Z/2}(X)$-module structures.
\end{proof}

\subsection{Basic property of $H_\pm$}

Because of the definition via $H^*_{\Z/2}$, the cohomology groups $H^*_\pm(X)$ constitute a $\Z/2$-equivariant cohomology theory as well: The homotopy, excision, exactness and additivity axioms are satisfied. There also exists the corresponding reduced theory $\tilde{H}^*_\pm(X) =  H^*_\pm(X, \pt)$ for $\Z/2$-spaces with fixed point. As usual, we have $H^n_\pm(\Sigma X) \cong H^{n-1}_\pm(X)$. It also holds that $\tilde{H}^n_\pm(\tilde{\Sigma} X) \cong \tilde{H}^{n-1}_{\Z/2}(X)$ and $\tilde{H}^n_{\Z/2}(\tilde{\Sigma}X) \cong \tilde{H}^{n-1}_\pm(X)$, where $\tilde{\Sigma} X = X \wedge \tilde{S}^1 = (X \times \tilde{I})/(\pt \times \tilde{I} \cup X \times \partial \tilde{I})$.

\medskip

To compute $H^*_\pm$, the following is useful:

\begin{prop} \label{prop:basic_exact_sequence}
For any space $X$ with $\Z/2$-action, there are natural exact sequences of modules over $H^*_{\Z/2}(X)$:
\begin{gather*}
\cdots \to
H^{n-1}_{\pm}(X) \overset{\delta}{\to}
H^{n}_{\Z/2}(X) \overset{f}{\to}
H^{n}(X) \to
H^{n}_{\pm}(X) \overset{\delta}{\to}
\cdots, \\
\cdots \to
H^{n-1}_{\Z/2}(X) \overset{\delta'}{\to}
H^{n}_{\pm}(X) \overset{f'}{\to}
H^{n}(X) \to
H^{n}_{\Z/2}(X) \overset{\delta'}{\to}
\cdots,
\end{gather*}
where $f$ and $f'$ are the map forgetting the $\Z/2$-action. 
\end{prop}

\begin{proof}
The first exact sequence directly follows from the exact sequences for the pair $(X \times \tilde{I}, X \times \partial \tilde{I})$ in $H^*_{\Z/2}$. For the second one, we similarly consider the exact sequence for the pair in $H^*_\pm$:
$$
\cdots \to
H^{n}_{\pm}(X \times \tilde{I}, X \times \partial \tilde{I}) \to
H^{n}_{\pm}(X \times \tilde{I}) \to
H^{n}_{\pm}(X \times \partial \tilde{I}) \to
\cdots.
$$
We have $H^{n}_{\pm}(X \times \tilde{I}, X \times \partial \tilde{I}) \cong H^{n-1}_{\Z/2}(X)$ by Lemma \ref{lem:recover_equivariant}, and $H^{n}_{\pm}(X \times \tilde{I}) \cong H^n_\pm(X)$ since $\tilde{I}$ is equivariantly contractible. We also have a module isomorphism
$$
H^n_{\pm}(X \times \partial \tilde{I}) 
= H^{n + 1}_{\Z/2}(X \times \partial \tilde{I} \times \tilde{I},
X \times \partial \tilde{I} \times \partial \tilde{I})
\cong H^{n+1}(X \times \tilde{I}, X \times \partial \tilde{I})
\cong H^n(X),
$$
since $\Z/2$ acts on $\partial \tilde{I}$ freely.
\end{proof}

Using the proposition above, we can determine the cohomology ring of $\pt$:

\begin{prop}
We have a ring isomorphism 
$$
\mathbb{H}^*(\pt)
= H^*_{\Z/2}(\pt) \oplus H^*_{\pm}(\pt) 
\cong \Z[t^{\frac{1}{2}}]/(2t^{\frac{1}{2}}),
$$
where $t^{\frac{1}{2}}$ corresponds to the additive generator of $H^1_{\pm}(\mathrm{pt}) = \Z/2$. In low degree:
$$
\begin{array}{|c|c|c|c|c|c|c|}
\hline
 & 0 & 1 & 2 & 3 & 4 & 5 \\
\hline
H^n_{\Z/2}(\mathrm{pt}) &
\Z & 0 & \Z/2 & 0 & \Z/2 & 0 \\
\hline
H^n(\mathrm{pt}) &
\Z & 0 & 0 & 0 & 0 & 0 \\
\hline
H_\pm^n(\mathrm{pt}) &
0 & \Z/2 & 0 & \Z/2 & 0 & \Z/2 \\
\hline
\end{array}
$$
\end{prop}

\begin{proof}
Recall the ring isomorphism $H^*_{\Z/2}(\mathrm{pt}) \cong \Z[t]/(2t)$. The first exact sequence in Proposition \ref{prop:basic_exact_sequence} immediately gives
$$
H^n_{\pm}(\mathrm{pt}) =
\left\{
\begin{array}{cl}
\Z/2, & (n = 1, 3, 5, \ldots) \\
0. & (\mbox{otherwise})
\end{array}
\right.
$$
Further, $\delta$ in the exact sequence is a map of modules over $H^*_{\Z/2}(\mathrm{pt})$, which determines the module structure on $H^*_\pm(\mathrm{pt})$. Explicitly, if we write $s \in H^1_\pm(\mathrm{pt})$ for the non-trivial element, then $t^{2i}s \in H^{2i+1}_\pm(\mathrm{pt})$ is the additive generator. Now, to determine the ring structure in question, it suffices to determine whether $s^2 \in H^2_{\Z/2}(\mathrm{pt})$ is trivial or not. For this aim, we consider the following diagram:
$$
\begin{CD}
H^1_{\pm}(\mathrm{pt}) \times H^1_{\pm}(\mathrm{pt}) @>{1 \times \delta}>>
H^1_{\pm}(\mathrm{pt}) \times H^2_{\Z/2}(\mathrm{pt}) \\
@VVV @VVV \\
H^2_{\Z/2}(\mathrm{pt}) @>{\delta'}>> H^3_\pm(\mathrm{pt}),
\end{CD}
$$
where the vertical maps are the multiplications. Recalling the definitions of $\delta$, $\delta'$ and $H^*_\pm$, we can show the commutativity of the diagram above. Hence we get $\delta'(s^2) = st \neq 0$. Because $\delta'$ is an isomorphism in the present case, we conclude $s^2 \neq 0$, so that $s^2 = t$. Setting $s = t^{1/2}$, we complete the proof.
\end{proof}

\bigskip

Because of its definition by means of equivariant cohomology, the Thom isomorphism for $\Z/2$-equivariantly orientable real vector bundle in $H^*_{\Z/2}$ is directly imported to $H^*_{\pm}$. Accordingly, there exists the Gysin sequence for sphere bundles.

\subsection{$H_\pm$ as a cohomology with local coefficients}

We now state the interpretation of $H^*_\pm$ as an equivariant cohomology with local coefficients.

\begin{prop} \label{prop:interpretation}
Let $X$ be a space with $\Z/2$-action. For any $n \in \Z$, we have an isomorphism of $H^*_{\Z/2}(X)$-modules:
$$
H^n_\pm(X) \cong H^{\underline{\R}_1 + n}_{\Z/2}(X) = H^n_{\Z/2}(X; \Z(1)).
$$
\end{prop}

\begin{proof}
The Thom isomorphism for $\underline{\R}_1$ gives
$$
H^{\R_1 + n}_{\Z/2}(X) 
\cong 
H^{n + 1}_{\Z/2}(D(\underline{\R}_1), 
S(\underline{\R}_1)) \\
\cong
H^{n + 1}_{\Z/2}(X \times \tilde{I}, 
X \times \partial \tilde{I}).
$$
The definition of $H^*_\pm$ completes the proof.
\end{proof}

The cup product in equivariant cohomology with local coefficients is compatible with the multiplication defined in Subsection \ref{subsec:definition_variant} through the isomorphism in Proposition \ref{prop:interpretation}, because $\Phi_{\Z/2}(\underline{\R}_1) \Phi_{\Z/2}(\underline{\R}_1) = \Phi_{\Z/2}(\underline{\C}_1)$. Therefore the graded ring $\mathbb{H}^*(X)$ is isomorphic to $\oplus_{m, n} H^n_{\Z/2}(X; \Z(m))$.

\medskip

As a simple application of Proposition \ref{prop:interpretation}, we have:

\begin{prop}
For any space $X$ with $\Z/2$-action and $n \in \Z$, the maps $\delta$ and $\delta'$ in Proposition \ref{prop:basic_exact_sequence} have the following realizations:
\begin{align*}
\delta &: \ H^n_\pm(X) \to H^{n+1}_{\Z/2}(X), &
&x \mapsto x \cup t^{1/2}, \\
\delta' &: H^n_{\Z/2}(X) \to H^{n+1}_\pm(X), &
&x \mapsto x \cup t^{1/2},
\end{align*}
where $t^{1/2} \in H^1_{\pm}(\mathrm{pt}) = \Z/2$ is the generator.
\end{prop}

\begin{proof}
By Proposition \ref{prop:interpretation}, we can identify the exact sequences in Proposition \ref{prop:basic_exact_sequence} with the Gysin sequences for the sphere bundle $S(\underline{\R}_1) \to X$, and the generator $t^{1/2}$ with the Euler class $\Phi_{\Z/2}(\underline{\R}_1)$. 
\end{proof}

\begin{cor}
The compositions
\begin{gather*}
H^n_\pm(X) \overset{\delta}{\to} 
H^{n+1}_{\Z/2}(X) \overset{\delta'}{\to} 
H^{n + 2}_\pm(X), \\
H^n_{\Z/2}(X) \overset{\delta'}{\to} 
H^{n+1}_{\pm}(X) \overset{\delta}{\to} 
H^{n + 2}_{\Z/2}(X)
\end{gather*}
are the multiplication of $t \in H^2_{\Z/2}(\mathrm{pt})$.
\end{cor}

It is well-known that the equivariant cohomology of a space $X$ with \textit{free} $\Z/2$-action is $H^n_{\Z/2}(X) \cong H^n(X/\Z_2) = H^n(X/\Z_2; \Z)$. Another simple application of Proposition \ref{prop:interpretation} gives the corresponding result for $H^*_\pm$:

\begin{lem} 
For a space $X$ with free $\Z/2$-action, there exists a natural isomorphism of groups
$$
H^n_\pm(X) \cong H^{\underline{\R}_1/\Z_2 + n}(X/\Z_2),
$$
where $\underline{\R}_1/\Z_2 \to X/\Z_2$ is the real line bundle given by taking the quotient of the equivariant real line bundle $\underline{\R}_1 \to X$.
\end{lem}

We remark that the isomorphism above makes $H^{\underline{\R}_1/\Z_2 + n}(X/\Z_2)$ into a module over $H^*_{\Z/2}(X)$. We remark also that $\underline{\R}_1/\Z_2 \cong X \times_{\Z_2} \R_1$.

\begin{proof}
By Proposition \ref{prop:interpretation}, we have $H^n_\pm(X) \cong H^{\underline{\R}_1 + n}_{\Z/2}(X)$. Since the $\Z/2$-action is free, everything descend to the quotient.
\end{proof}

\subsection{Thom isomorphism for Real line bundle}

Let $X$ be a space with $\Z/2$-action, and $\pi : R \to X$ a Real line bundle, that is, a complex line bundle with real structure. Forgetting its complex structure, we get a real vector bundle of rank $2$, which we denote by $R_{\R}$. Though $R$ is not equivariant as a complex vector bundle, the real vector bundle $R_{\R}$ is equivariant. We thus get the determinant bundle $\det_{\R}R = \mathrm{det}R_{\R}$, which is a $\Z/2$-equivariant real vector bundle of rank $1$.

\begin{lem} \label{lem:determinant}
For any Real line bundle $\pi : R \to X$ over a space $X$ with $\Z/2$-action, there is a natural isomorphism of $\Z/2$-equivariant real vector bundles:
$$
\mathrm{det}_{\R} R \cong \underline{\R}_1.
$$
In particular, $w_1^{\Z/2}(R) = w_1^{\Z/2}(\underline{\R}_1)$ in $H^1_{\Z/2}(X; \Z/2)$.
\end{lem}

\begin{proof}
If we forget about the $\Z/2$-action, then the determinant bundle $\det_{\R}R$ admits a canonical section. Explicitly, the value of the section at $x \in X$ is $v \wedge Jv \in \mathrm{det}_{\R}R$, where $v \in R$ is any unit norm element (with respect to a $\Z/2$-invariant Hermitian metric) and $J : R \to R$ is the multiplication of $i \in \C$. This section defines an isomorphism $\det_{\R}R \cong \underline{\R}_1$ of the underlying bundles. Recalling that $R$ is a Real vector bundle, we compute the $\Z/2$-action on the section to have
$$
\tau(v \wedge Jv) 
= \tau(v) \wedge \tau(Jv) 
= \tau(v) \wedge (- J \tau(v))
= - \tau(v) \wedge J \tau(v).
$$
This proves that the isomorphism $\det_{\R}R \cong \underline{\R}_1$ is $\Z/2$-equivariant.
\end{proof}

\begin{thm}[Thom isomorphism]
Let $X$ be a space with $\Z/2$-action. For any Real line bundle $\pi : R \to X$ and $n \in \Z$, we have natural isomorphisms of modules over $H^*_{\Z/2}(X)$:
\begin{align*}
H^n_{\Z/2}(X) &\cong H^{n + 2}_{\pm}(D(R), S(R)), &
H^n_{\pm}(X) &\cong H^{n + 2}_{\Z/2}(D(R), S(R)).
\end{align*}
If we write $\Phi_{\pm}(R) \in H^2_{\pm}(D(R), S(R))$ for the image of $1 \in H^0_{\Z/2}(X)$, then the isomorphisms are given by $x \mapsto \pi^*x \cup \Phi_{\pm}(R)$.
\end{thm}

\begin{proof}
This is a direct consequence of the Thom isomorphism theorem in equivariant cohomology with local coefficients, Lemma \ref{lem:determinant} and Proposition \ref{prop:interpretation}.
\end{proof}

In the same way as usual, we define the Euler class $\chi_R(R) \in H^2_\pm(X)$ of a Real line bundle $\pi : R \to X$ to be the image of $1 \in H^0_{\Z/2}(X)$ under the composition of
$$
H^0_{\Z/2}(X) \overset{\mathrm{Thom}}{\to} 
H^2_\pm(D(R), S(R)) \to
H^2_\pm(D(R)) \cong 
H^2_\pm(X).
$$
Similarly, we define the push-forward $\pi_* : H^n_\pm(S(R)) \to H^{n-1}_{\Z/2}(X)$ along $\pi: S(R) \to X$ to be the composition of
$$
H^n_\pm(S(R)) \to
H^{n+1}_\pm(D(R), S(R)) \overset{\mathrm{Thom}}{\to}
H^{n -1}_{\Z/2}(X).
$$
The push-forward $\pi_* : H^n_{\Z/2}(S(R)) \to H^{n-1}_{\pm}(X)$ is defined in the same way. Now, considering the exact sequence for the pair $(D(R), S(R))$, we get:

\begin{cor}[Gysin sequence]
For any Real line bundle $\pi : R \to X$, we have natural exact sequences of $H^*_{\Z/2}(X)$-modules:
\begin{gather*}
\cdots \to
H^{n-2}_{\Z/2}(X) \overset{\chi_R(R) \cup }{\to}
H^{n}_{\pm}(X) \overset{\pi^*}{\to}
H^{n}_{\pm}(S(R)) \overset{\pi_*}{\to}
H^{n-1}_{\Z/2}(X) \to
\cdots, \\
\cdots \to
H^{n-2}_{\pm}(X) \overset{\chi_R(R) \cup }{\to}
H^{n}_{\Z/2}(X) \overset{\pi^*}{\to}
H^{n}_{\Z/2}(S(R)) \overset{\pi_*}{\to}
H^{n-1}_{\pm}(X) \to
\cdots.
\end{gather*}
\end{cor}

\medskip

In the statements above, $H^*_{\Z/2}$ and $H^*_\pm$ are dealt with separately, so that two Thom isomorphisms and two Gysin sequences are. If we combine them, then we get a single Thom isomorphism theorem and a single Gysin exact sequence in $\mathbb{H}^*$ respecting the $\mathbb{H}^*(X)$-module structures.

\medskip

Let $\tilde{S}^1$ be the unit circle $S^1 \subset \C$ whose $\Z/2$-action is $u \mapsto \bar{u} = u^{-1}$. For any space $X$ with $\Z/2$-action, the direct product $X \times \tilde{S}^1$ is the total space of the trivial Real circle bundle $\pi : X \times \tilde{S}^1 \to X$. Because of the equivariant section $i : X \to X \times \tilde{S}^1$, ($i(x) = (x, \pt)$), the Gysin sequences reduce to the split exact sequences:
\begin{gather*}
0 \to 
H^n_{\Z/2}(X) \overset{\pi^*}{\to} 
H^n_{\Z/2}(X \times \tilde{S}^1) \overset{\pi_*}{\to}
H^{n-1}_{\pm}(X) \to
0, \\
0 \to 
H^n_{\pm}(X) \overset{\pi^*}{\to}
H^n_{\pm}(X \times \tilde{S}^1) \overset{\pi_*}{\to}
H^{n-1}_{\Z/2}(X) \to
0,
\end{gather*}
so that there are natural $H^*_{\Z/2}(X)$-module isomorphisms:
\begin{align*}
H^n_{\Z/2}(X \times \tilde{S}^1) &\cong 
H^n_{\Z/2}(X) \oplus H^{n-1}_\pm(X), \\
H^n_{\pm}(X \times \tilde{S}^1) &\cong 
H^n_{\pm}(X) \oplus H^{n-1}_{\Z/2}(X),
\end{align*}
which are also compatible with the $\mathbb{H}^*(X)$-module structures.

The next lemma, which we use later, is an application of the above splitting.

\begin{lem} \label{lem:cohomology_circle_with_flip}
The following holds true:
\begin{itemize}
\item[(a)]
We have a ring isomorphism
$$
\mathbb{H}^*(\tilde{S}^1)
=
H^*_{\Z/2}(\tilde{S}^1) \oplus H^*_{\pm}(\tilde{S}^1)
\cong 
\Z[t^{\frac{1}{2}}, \chi]/(2t^{\frac{1}{2}}, \chi^2 - t^{\frac{1}{2}}\chi),
$$
where $t^{\frac{1}{2}} \in H^1_\pm(\mathrm{pt}) \cong \Z/2$, and $\chi \in \tilde{H}^1_{\pm}(S^1_{\mathrm{flip}}) \cong \Z$ is the unique element such that $\pi_* \chi = 1$. In low degree, we have:
$$
\begin{array}{|c|c|c|c|c|c|}
\hline
 & n = 0 & n = 1 & n = 2 & n = 3 & n = 4 \\
\hline
H^n_{\Z/2}(\tilde{S}^1) &
\Z & 0 & \Z_2 t^{\frac{1}{2}} \chi \oplus \Z_2 t & 0 & 
\Z_2 t^{\frac{3}{2}} \chi \oplus \Z_2 t^2 \\
\hline
H^n(\tilde{S}^1) &
\Z & \Z f'(\chi) & 0 & 0 & 0 \\
\hline
H_\pm^n(\tilde{S}^1) &
0 & \Z \chi \oplus \Z_2 t^{\frac{1}{2}} & 0 & 
\Z_2 t\chi \oplus \Z_2 t^{\frac{3}{2}} & 0 \\
\hline
\end{array}
$$

\item[(b)]
The equivariant map $\nu : \tilde{S}^1 \to \tilde{S}^1$, ($\nu (u) = - u$) acts on $\mathbb{H}^*(\tilde{S}^1)$ by 
\begin{align*}
\nu^* \chi &= \chi + t^{1/2}, &
\nu^* t^{1/2} &= t^{1/2}.
\end{align*}

\end{itemize}

\end{lem}

\begin{proof}
We choose a fixed point $p \in \tilde{S}^1$, which is $+1$ or $-1$. Setting $i(\pt) = p$, we define a section $i : \pt \to \tilde{S}^1$ of the projection $\pi : \tilde{S}^1 \to \pt$. We use this section to split the Gysin sequence. If we choose $\chi \in \tilde{H}^1_\pm(\tilde{S}^1)$ such that $\pi_* \chi = 1$, then we have the following basis of $\mathbb{H}^*(\tilde{S}^1)$ respecting the $\mathbb{H}^*(\pt)$-module structure:
\begin{align*}
H^0_{\Z/2}(\tilde{S}^1) &\cong \Z, &
H^{2k}_{\Z/2}(\tilde{S}^1) &\cong 
\Z_2 t^{k} \oplus \Z_2 t^{\frac{2k - 1}{2}}\chi, \\
H^1_\pm(\tilde{S}^1) &\cong \Z_2 t^{\frac{1}{2}} \oplus \Z \chi, &
H^{2k+1}_\pm(\tilde{S}^1) &\cong 
\Z_2 t^{\frac{2k + 1}{2}} \oplus \Z_2 t^k \chi,
\end{align*}
where $k \ge 1$. To determine $\chi^2$, we cover $\tilde{S}^1$ by two intervals $U$ and $V$ such that $U \cong V \cong \pt$ and $U \cap V \cong \partial \tilde{I}$ equivariantly. In the Mayer-Vietoris sequence for $\{ U, V \}$, we find the following surjections:
\begin{align*}
 & H^1_{\pm}(\tilde{S}^1) \to H^1_\pm(U) \oplus H^1_\pm(V) 
= \Z_2 t^{1/2} \oplus \Z_2 t^{1/2}, \\
 & H^2_{\Z/2}(\tilde{S}^1) \to H^2_{\Z/2}(U) \oplus H^2_{\Z/2}(V)
= \Z_2 t \oplus \Z_2 t.
\end{align*}
The image of $t^{1/2} \in H^1_\pm(\tilde{S}^1)$ is $(t^{1/2}, t^{1/2}) \in H^1_\pm(U \sqcup V)$, whereas the image of $\chi \in H^1_\pm(\tilde{S}^1)$ is $(0, t^{1/2}) \in H^1_\pm(U \sqcup V)$, provided that $U$ contains the fixed point $p$. From these expressions, we conclude $\chi^2 = t^{1/2}\chi$, completing the proof of (a). The expressions above also prove (b).
\end{proof}


The Thom isomorphism theorem for Real line bundles is readily generalized to Real vector bundles $\pi : V \to X$ of rank $r > 1$. In the case where $r$ is odd, the Thom isomorphism is similar to the case of $r = 1$, since $\mathrm{det}_\R V = \underline{\R}_1$. However, in the case where $r$ is even, we have $\mathrm{det}_\R V = \underline{\R}_0$. Therefore the Thom isomorphisms have the following form:
\begin{align*}
H^n_{\Z/2}(X) &\cong H^{n + 2r}_{\Z/2}(D(V), S(V)), &
H^n_{\pm}(X) &\cong H^{n + 2r}_{\pm}(D(R), S(R)).
\end{align*}
Accordingly, the Gysin exact sequences are
\begin{gather*}
\cdots \to
H^{n-2r}_{\Z/2}(X) \overset{\chi_R(V) \cup }{\to}
H^{n}_{\Z/2}(X) \overset{\pi^*}{\to}
H^{n}_{\Z/2}(S(V)) \overset{\pi_*}{\to}
H^{n-2r+1}_{\Z/2}(X) \to
\cdots, \\
\cdots \to
H^{n-2r}_{\pm}(X) \overset{\chi_R(V) \cup }{\to}
H^{n}_{\pm}(X) \overset{\pi^*}{\to}
H^{n}_{\pm}(S(V)) \overset{\pi_*}{\to}
H^{n-2r+1}_{\pm}(X) \to
\cdots,
\end{gather*}
where the Euler class $\chi_R(V)$ lives in $H^{2r}_{\Z/2}(X)$.

\subsection{Chern class of Real vector bundle}
\label{subsec:chern_class_of_Real_vector_bundle}

For a Real vector bundle $\pi : V \to X$ over a $\Z/2$-space $X$, the notion of Chern classes is formulated in \cite{K,P-S}. In this paper, we will write the $i$th Chern classes of $V$ as
$$
c_i^R(V) \in H^{2i}_{\Z/2}(X; \Z(i)).
$$
We will regard $c_{2j+1}^R(V) \in H^{2j + 1}_\pm(X)$ by using Proposition \ref{prop:interpretation} freely. These Chern classes and the usual Chern classes are compatible under the operation of forgetting $\Z/2$-actions, that is, $f$ and $f'$ carry $c_i^R(V)$ to $c_i(V) \in H^{2i}(X)$.

Most of the properties of these Chern classes $c_i^R(V)$ are parallel to those of the usual Chern classes. For example, the total Chern class $c^R(V) = 1 + c_1^R(V) + c_2^R(V) + \cdots$ obeys $c^R(V \oplus V) = c^R(V)c^R(V')$ for another Real vector bundle $V'$. Also, if $V$ is of rank $r$, then $\chi_R(V) = c_{2r}^R(V)$. Further, the following classification of Real line bundles is possible:

\begin{prop}[\cite{K}] \label{prop:classify_Real_circle_bundle}
Let $X$ be a space with $\Z/2$-action. The first Chern class $c_1^R$ induces an isomorphism from the group of isomorphism classes of Real line bundles on $X$ to $H^2_{\Z/2}(X; \Z(1)) \cong H^2_\pm(X)$.
\end{prop}

We give a proof of this classification in Appendix \ref{sec:appendix}. 

As is mentioned in Introduction, the notion of Real line bundles and that of Real circle bundles are essentially the same. Hence the first Chern class $c_1^R$ will also be adapted to Real circle bundles.

\smallskip

As in the case of line bundles without real structure, there exists the classifying space for Real line bundles (see for example \cite{HHP} for a proof of this fact): Concretely, the classifying space for Real line bundles is $\C P^\infty = \varinjlim_n \C P^n$ with the $\Z/2$-action induced from the complex conjugation on $\C P^n = (\C^{n+1} \backslash \{ 0 \})/\C^*$, and the universal line bundle on $\C P^\infty$ gives rise to the universal Real line bundle. Any Real line bundle $R$ on a $\Z/2$-space $X$ is isomorphic to the pull-back of the universal Real line bundle under a $\Z/2$-equivariant map $\phi : X \to \C P^\infty$, and hence $c_1^R(R)$ is the pull-back of the first Chern class of the universal Real line bundle.

\bigskip

Given a Real line bundle $\pi : R \to X$ on a space $X$ with $\Z/2$-action, its underlying $\Z/2$-equivariant real vector bundle $R_{\R}$ of rank $2$ has the third integral Stiefel-Whitney class $W_3^{\Z/2}(R) = W_3^{\Z/2}(R_\R) \in H^3_{\Z/2}(X; \Z)$. This characteristic class is the obstruction to the existence of $\Z/2$-equivariant $\Pin^c(2)$-structure on $R_{\R}$. (See for example \cite{L-M} about $\Pin^c$-groups and $\Pin^c$-structures.)

\begin{prop} \label{prop:Pin_c_structure}
For any Real line bundle $\pi : R \to X$, we have:
$$
W_3^{\Z/2}(R) = \delta (c_1^R(R)).
$$
\end{prop}

As a result, we see that $W_3^{\Z/2}(R)$ does not generally agrees with $W_3^{\Z/2}(\underline{\R}_1) = 0$. ($W_3^{\Z/2}(\underline{\R}_1) = 0$ follows from $H^3_{\Z/2}(\pt; \Z) = 0$.)

\smallskip

We show a few lemmas to prove Proposition \ref{prop:Pin_c_structure}.

\begin{lem} \label{lem:cohomology_circle_with_trivial_action}
Let $S^1$ be the circle with the trivial $\Z/2$-action. Then we have the following ring isomorphism:
$$
\mathbb{H}^*(S^1) 
= H^*_{\Z/2}(S^1) \oplus H^*_\pm(S^1)
\cong \Z[t^{1/2}, e]/(2t^{1/2}, e^2),
$$
where $t^{1/2} \in H^1_\pm(\pt; \Z) \cong \Z/2$ and $e \in H^1_{\Z/2}(S^1) \cong \Z$. In low degree:
$$
\begin{array}{|c|c|c|c|c|c|}
\hline
 & n = 0 & n = 1 & n = 2 & n= 3 & n = 4 \\
\hline
H^n_{\Z/2}(S^1) &
\Z & \Z e & \Z_2 t & \Z_2 te & \Z_2 t^2 \\
\hline
H^n(S^1) &
\Z & \Z f(e) & 0 & 0 & 0 \\
\hline
H^n_\pm(S^1) &
0 & \Z_2 t^{\frac{1}{2}} & \Z_2 t^{\frac{1}{2}}e & \Z_2 t^{\frac{3}{2}} & 
\Z_2 t^{\frac{3}{2}}e \\
\hline
\end{array}
$$
\end{lem}

\begin{proof}
Since $H^*(S^1; \Z)$ is of torsion free, an application of the K\"{u}nneth formula to the definition of the equivariant cohomology gives $H^*_{\Z/2}(S^1) \cong H^*_{\Z/2}(\pt) \otimes_{\Z} H^*(S^1)$. In the same way, we have $H^*_\pm(S^1) \cong H^*_{\pm}(\pt) \otimes_{\Z} H^*(S^1)$.
\end{proof}

Let $R \to S^1$ be the Real line bundle such that its underlying line bundle is $R = S^1_{\mathrm{triv}} \times \C$ and $\Z/2$ acts by $\tau(u, z) \mapsto (u, u \bar{z})$, where we regard $u \in S^1 \subset \C$.

\begin{lem} 
$W_3^{\Z/2}(R) \neq 0$ in $H^3_{\Z/2}(S^1_{\mathrm{triv}}; \Z) = \Z/2$.
\end{lem}

\begin{proof}
The spectral sequence
$$
E_2^{p, q} = H^p_{\mathrm{group}}(\Z_2; H^q(S^1; \Z))
\Longrightarrow H^*_{\Z/2}(S^1)
$$
degenerates at $E_2$, as is seen by comparing the computation of $H^*_{\Z/2}(S^1)$ so far. Hence we have
$$
H^3_{\Z/2}(X; \Z) \cong E_2^{2, 1}
= H^2_{\mathrm{group}}(\Z_2; H^1(S^1; \Z)).
$$
With this isomorphism, $W^{\Z/2}_3(R)$ is computed as follows: We choose an invariant Hermitian metric on $R$. If we ignore the group action, then the unoriented frame bundle of $R_\R$ is $F = S^1 \times O(2)$, and there is the unique $\Pin^c(2)$-structure $\tilde{F} = S^1 \times \Pin^c(2)$. We here try to lift the $\Z/2$-action on $F$ to that on $\tilde{F}$. A candidate of a lift is expressed as $g \cdot (u, \tilde{\eta}) = (u, \phi(g; u) \tilde{\eta})$ for $g \in \Z_2$ and $(u, \tilde{\eta}) \in \tilde{F}$ by using a map $\phi : \Z_2 \times S^1 \to \Pin^c(2)$. Then we get a group $2$-cocycle $\zeta$ of $\Z_2$ with values in the group $\Map(S^1, U(1))$ of $U(1)$-valued functions by the formula
$$
\phi(g_1; u) \phi(g_2; u)
= \zeta(g_1, g_2; u) \phi(g_1g_2; u).
$$
The natural map $\Map(S^1, U(1)) \to H^1(S^1; \Z)$ induces a homomorphism in the group cohomology $H^2_{\mathrm{group}}(\Z_2; \Map(S^1, U(1))) \to H^2_{\mathrm{group}}(\Z_2; H^1(S^1; \Z))$. The image of $[\zeta] \in H^2_{\mathrm{group}}(\Z_2; \Map(S^1, U(1)))$ under the homomorphism is $W_3^{\Z/2}(R_{\R})$.

Now, we carry out the computation above concretely. The $\Z/2$-action on $F$ induced from that on $R$ is identified with $\tau(e^{i\theta}, g) = (e^{i\theta}, r(\theta)\epsilon g)$, where
\begin{align*}
r(\theta) 
&=
\left(
\begin{array}{rr}
\cos \theta & - \sin \theta \\
\sin \theta & \cos \theta
\end{array}
\right), &
\epsilon
&=
\left(
\begin{array}{rr}
-1 & 0 \\
0 & 1
\end{array}
\right).
\end{align*}
To construct a candidate of a lift, we recall the group $\Pin(2)$:
\begin{align*}
\Pin(2) &=
\{ \cos \theta + \sin \theta e_1e_2 \in Cl(2) |\ 
\theta \in \R/2\pi \Z \} \\
& \quad
\cup
\{ \cos \theta e_1 + \sin \theta e_2 \in Cl(2) |\ \theta \in \R/2\pi \Z \},
\end{align*}
where the Clifford algebra $Cl(2)$ is the algebra over $\R$ generated by $e_1$ and $e_2$ subject to the relations $e_1^2 = e^2 = -1$ and $e_1e_2 + e_2e_1 = 0$. The $\Z/2$-covering $\varpi : \Pin(2) \to O(2)$ is given by
\begin{align*}
\varpi(\cos \theta + \sin \theta e_1e_2)
&=
r(2\theta), &
\varpi(\cos \theta e_1 + \sin \theta e_2)
&=
r(2\theta) \epsilon
\end{align*}
The group $\Pin^c(2)$ is then given by
$$
\Pin^c(2) = (\Pin(2) \times U(1))/\Z_2
$$
where $\Z_2$ acts on $\Pin(2) \times U(1)$ by $(g, z) \mapsto (-g, -z)$. We write $[g, z] \in \Pin^c(2)$ for the element represented by $(g, z)$. Now, we define $\phi : \Z_2 \times S^1 \to \Pin^c(2)$ by
\begin{align*}
\phi(1; e^{i\theta}) &= 1, &
\phi(-1; e^{i\theta}) &=
[\cos \frac{\theta}{2} e_1 + \sin\frac{\theta}{2} e_2, e^{i\theta/2}] 
[e_1e_2, 1].
\end{align*}
We easily see $\Z_2 \times \tilde{F} \to \tilde{F}$, $(g, u, \tilde{\eta}) \mapsto (u, \phi(g; u)\tilde{\eta})$ coves the $\Z/2$-action $\Z_2 \times F \to F$. The group $2$-cocycle $\zeta : \Z_2 \times \Z_2 \times S^1 \to \R$ computed from $\phi$ is then
\begin{align*}
\zeta(1, 1; u) &= \zeta(1, -1; u) = \zeta(-1, 1; u) = 1, &
\zeta(-1, -1; u) &= -u,
\end{align*}
which is mapped to a non-trivial class in $H^2_{\mathrm{group}}(\Z_2; H^1(S^1; \Z)) = \Z/2$.
\end{proof}

\begin{lem} \label{lem:cohomology_CP_infty}
Let $\C P^\infty$ be the classifying space of Real line bundles. Then we have the following ring isomorphism:
$$
\mathbb{H}^*(\C P^\infty) 
= H^*_{\Z/2}(\C P^\infty) \oplus H^*_\pm(\C P^\infty)
\cong \Z[t^{1/2}, c]/(2t^{1/2}),
$$
where $t^{1/2} \in H^1_\pm(\pt; \Z) \cong \Z/2$, and $c \in H^2_{\Z/2}(\C P^\infty) \cong \Z$ is the first Chern class of the universal Real line bundle. In low degree, we have:
$$
\begin{array}{|c|c|c|c|c|c|}
\hline
 & n = 0 & n = 1 & n = 2 & n = 3 & n = 4  \\
\hline
H^n_{\Z/2}(\C P^\infty) &
\Z & 0 & \Z_2 t & \Z_2 t^{\frac{1}{2}}c & \Z c^2 \oplus \Z_2 t^2  \\
\hline
H^n(\C P^\infty) &
\Z & 0 & \Z f'(c) & 0 & \Z f(c^2)  \\
\hline
H^n_\pm(\C P^\infty) &
0 & \Z_2 t^{\frac{1}{2}} & \Z c & \Z_2 t^{\frac{3}{2}} & \Z_2 tc  \\
\hline
\end{array}
$$
\end{lem}

\begin{proof}
It is well-known that $H^*(\C P^\infty; \Z) \cong \Z[c]$, where $c \in H^2(\C P^\infty; \Z) \cong \Z$ is the first Chern class of the universal line bundle. The $\Z/2$-space $\C P^\infty$ is an example of a spherical conjugation complex \cite{HHP}. For such a space, the spectral sequence reviewed in Subsection \ref{subsec:cohomology_with_local_coefficients} degenerates at the $E_2$-term \cite{P-S}. Hence $\mathbb{H}^*(\C P^\infty) \cong \Z[t^{1/2}, c]/(2t^{1/2})$ at least as an abelian group. Clearly, $H^1_\pm(\C P^\infty) \cong \Z/2$ has $t^{1/2} \in H^1_\pm(\pt)$ as its basis. The generator $c \in H^2_\pm(\C P^\infty; \Z) \cong \Z$ agrees with the first Chern class of the universal Real line bundle, because $f : H^2_\pm(\C P^\infty) \to H^2(\C P^\infty)$ is the identity. Using Proposition \ref{prop:basic_exact_sequence}, we can verify that $\mathbb{H}^*(\C P^\infty) \cong \Z[t^{1/2}, c]/(2t^{1/2})$ is actually a ring isomorphism.
\end{proof}

\begin{proof}[The proof of Proposition \ref{prop:Pin_c_structure}]
For a space $X$ with $\Z/2$-action, the assignment of $W_3^{\Z/2}(R) \in H^3_{\Z/2}(X; \Z)$ to a Real line bundle $R \to X$ defines a natural map $W_3^{\Z/2} : H^2_\pm(X) \to H^3_{\Z/2}(X; \Z)$. Proposition \ref{prop:Pin_c_structure} is equivalent to $W_3^{\Z/2} = \delta$. By the naturality, this will be established by showing that $W_3^{\Z/2} : H^2_{\pm}(\C P^\infty) \to H^3_{\Z/2}(\C P^\infty; \Z)$ is a non-trivial map $\Z \to \Z/2$, i.e.\ the reduction mod $2$. For its proof, we use the commutative diagram from the naturality
$$
\begin{CD}
H^2_{\pm}(\C P^\infty) @>{W_3^{\Z/2}}>> H^3_{\Z/2}(\C P^\infty; \Z) \\
@V{\phi_\pm^*}VV @VV{\phi_{\Z/2}^*}V \\
H^2_{\pm}(S^1) @>{W_3^{\Z/2}}>> 
H^3_{\Z/2}(S^1; \Z),
\end{CD}
$$
where $\phi^*_\pm$ and $\phi_{\Z/2}^*$ are the pull-back under a map $\phi : S^1 \to \C P^\infty$ classifying the Real line bundle $R \to S^1$. Because $W_3^{\Z/2}(R) \neq 0$, this Real line bundle is non-trivial, hence so is $\phi^*_\pm : \Z \to \Z/2$. Clearly, $\phi^*_{\Z/2} : \Z/2 \to \Z/2$ is the identity. By Lemma \ref{lem:cohomology_circle_with_trivial_action}, the map $W_3^{\Z/2} : H^2_{\pm}(S^1; \Z) \to H^3_{\Z/2}(S^1; \Z)$ turns out to be the identity $\Z/2 \to \Z/2$ as well, so that $W_3^{\Z/2} : H^2_{\pm}(\C P^\infty) \to H^3_{\Z/2}(\C P^\infty; \Z)$ is non-trivial.
\end{proof}


\section{A variant of $K$-theory: $K_\pm$}
\label{sec:K_pm}

This section reviews $K_\pm$ and its twisted version. The properties of $K_\pm$ are basically parallel to those of $H_\pm$, but the differences are to be noticed.

\subsection{Review of equivariant twisted $K$-theory}

Twisted $K$-theory is originally invented in \cite{D-K,R1}. Equivariant twisted $K$-theory is in a sense equivariant $K$-theory with local coefficients. We call the datum playing the role of a local system a \textit{twist}. Applying the formulation in \cite{FHT1} to the groupoid associated to a space $X$ with $\Z/2$-action, we have the monoidal category of (equivariant, graded) twists $\Twist_{\Z/2}(X)$. The group $\pi_0(\Twist_{\Z/2}(X))$ of isomorphism classes of twists on $X$ fits into the exact sequence:
$$
1 \to 
H^3_{\Z/2}(X; \Z) \to
\pi_0(\Twist_{\Z/2}(X)) \to
H^1_{\Z/2}(X; \Z/2) \to
1.
$$
This exact sequence splits as a set, but not as a group: its extension class 
$$
H^1_{\Z/2}(X; \Z/2) \times H^1_{\Z/2}(X; \Z/2) \to H^3_{\Z/2}(X; \Z)
$$ 
is given by $(\alpha, \alpha') \mapsto \beta(\alpha \cup \alpha')$, where $\beta : H^2_{\Z/2}(X; \Z/2) \to H^3_{\Z/2}(X; \Z)$ is the Bockstein homomorphism associated to the exact sequence of coefficients $\Z \to \Z \to \Z/2$. We write $\Twist^0_{\Z/2}(X)$ for the subcategory of twists consisting of `ungraded twists', i.e.\ twists classified by $H^3_{\Z/2}(X; \Z)$. On any twist $\tau \in \Twist_{\Z/2}(X)$, a ${\Z/2}$-equivariant line bundle acts as an automorphism, preserving the subcategory $\Twist^0_{\Z/2}(X)$. Actually, in this subcategory, the automorphism group of a twist $\tau \in \Twist^0_{\Z/2}(\Z/2)$ is isomorphic to the equivariant cohomology $H^2_{\Z/2}(X; \Z)$, which classifies  $\Z/2$-equivariant line bundles \cite{H-Y}.

\medskip

Given a twist $\tau \in \Twist_{\Z/2}(X)$ and $n \in \Z$, we have the $\Z/2$-equivariant twisted $K$-group $K^{\tau + n}_{\Z/2}(X)$. If twists are isomorphic, then the corresponding twisted $K$-groups are isomorphic. Accordingly, we often use the notation $K^{h + n}_{\Z/2}(X)$ for $h \in \pi_0(\Twist_{\Z/2}(X)) \cong H^1_{\Z/2}(X; \Z/2) \times H^3_{\Z/2}(X; \Z)$, keeping track of the isomorphism class of twists only. If $h = 0$, then the twisted equivariant $K$-group recovers the usual equivariant $K$-group. By means of a multiplication
$$
K^{h + n}_{\Z/2}(X) \times K^{h' + n'}_{\Z/2}(X) \longrightarrow
K^{h + h' + n + n'}_{\Z/2}(X),
$$
the group $\bigoplus_n K^{h + n}_{\Z/2}(X)$ gives rise to a graded module over the equivariant $K$-ring $K^*_{\Z/2}(X)$. In particular, $K^{h + n}_{\Z/2}(X)$ is always a module over the representation ring $K^0_{\Z/2}(\pt) \cong R = \Z[t]/(t^2 - 1)$ of $\Z/2$.

\smallskip

The equivariant twisted $K$-groups form a ${\Z/2}$-equivariant cohomology theory on a suitable category: We can formulate the relative group $K^{\tau + n}_{\Z/2}(X, Y)$ for a pair $(X, Y)$ and a twist $\tau \in \Twist_{\Z/2}(X)$. Then the homotopy axiom, the excision axiom, the exactness axiom and the additivity axiom hold true. Further, the Bott periodicity is satisfied: $K^{\tau+ n}_{\Z/2}(X, Y) \cong K^{\tau + n + 2}_{\Z/2}(X, Y)$. 

\smallskip

For a $\Z/2$-equivariant real vector bundle $\pi : V \to X$ of rank $r$ and $\tau \in \Twist_{\Z/2}(X)$, the Thom isomorphism for equivariant twisted $K$-theory is the following isomorphism of $K^*_{\Z/2}(X)$-modules \cite{FHT1}:
$$
K^{\tau + n}_{\Z/2}(X) \cong 
K^{\pi^*(\tau + \tau(V)) + n + r}_{\Z/2}(D(V), S(V)),
$$
where $\tau(V) \in \Twist_{\Z/2}(X)$ is a twist classified by 
$$
(w_1^{\Z/2}(V), W_3^{\Z/2}(V)) 
\in H^1_{\Z/2}(X; \Z/2) \times H^3_{\Z/2}(X; \Z).
$$
The Thom isomorphism is expressed as $x \mapsto \pi^*x \cup \Phi_{\Z/2}(V)$ by means of the Thom class $\Phi_{\Z/2}(V) \in K^{\pi^*\tau(V) + r}_{\Z/2}(D(V), S(V))$, which is the image of $1 \in K^0_{\Z/2}(X)$ under the Thom isomorphism. Once the Thom isomorphism is established, the Euler class $\chi_{\Z/2}(V) \in K^{\tau(V) + r}_{\Z/2}(X)$, the push-forward along $\pi : S(V) \to X$ and the Gysin exact sequence follow as usual.

\smallskip

In the above, we only concerned $\Z/2$-equivariant case. Forgetting the group action, we have the same story of twisted $K$-theory $K^{\tau + n}(X)$ on a space $X$.

\subsection{The variant $K_\pm$ and its twisted version}

\begin{dfn}
Let $X$ be a space with $\Z/2$-action. For $\tau \in \Twist^0_{\Z/2}(X)$ and $n \in \Z$, we define 
$$
K^{\tau + n}_{\pm}(X)
= K^{\tau + n + 1}_{\Z/2}(X \times \tilde{I}, X \times \partial \tilde{I}).
$$
When $\tau$ is trivial, we just write $K^n_\pm(X) = K^{\tau + n}_{\pm}(X)$.
\end{dfn}

As in the case of $H_\pm$, the group $K^{\tau + *}_{\pm}(X)$ is a module over $K^*_{\Z/2}(X)$ and hence $K^0_{\Z/2}(\pt) \cong R$. We can further define the multiplication
$$
K^{\tau + n}_{\pm}(X) \times K^{\tau' + n'}_{\pm}(X) \longrightarrow
K^{\tau + \tau' + n + n'}_{\Z/2}(X)
$$
in the same way as in Subsection \ref{subsec:definition_variant}. We put
$$
\mathbb{K}^{\tau + *}(X) = K^{\tau + *}_{\Z/2}(X) \oplus K^{\tau + *}_{\pm}(X).
$$
which is graded by $\Z \oplus \Z$ with double periodicity, i.e.\ $\Z/2 \oplus \Z/2$. Then $\mathbb{K}^{\tau + *}(X)$ is a module over the ring $\mathbb{K}^*(X)$ as well as $\mathbb{K}^*(\pt)$.

\medskip

In the case where $X$ contains an invariant closed subspace $Y$, we define
$$
K^{\tau + n}_\pm(X, Y)
= K^{\tau + n + 1}_{\pm}(X \times \tilde{I}, 
Y \times \tilde{I} \cup X \times \partial \tilde{I}).
$$
Using this relative version, we can again recover $K^{\tau + *}_{\Z/2}(X)$ from $K^{\tau + *}_{\pm}(X)$:
$$
K^{\tau + n}_{\Z/2}(X)
\cong K^{\tau + n + 1}_{\pm}(X \times \tilde{I}, X \times \partial \tilde{I}).
$$

\medskip

As in the case of $H_\pm$, we can interpret $K^{\tau + n}_\pm(X)$ as a twisted $K$-theory:

\begin{lem} \label{prop:interpretation_K}
For any $\tau \in \Twist^0_{\Z/2}(X)$ and $n \in \Z$, there is a natural isomorphism of $K^*(X)$-modules
$$
K^{\tau(\underline{\R}) + \tau + n}_{\Z/2}(X)
\cong K^{\tau + n}_\pm(X).
$$
\end{lem}

\begin{proof}
Thom isomorphism for $\underline{\R}_1$ establishes the lemma.
\end{proof}

Because $\Phi_{\Z/2}(\underline{\R}_1) \Phi_{\Z/2}(\underline{\R}_1) = \Phi_{\Z/2}(\underline{\C}_1)$, the multiplication
$$
K^{\tau + n}_\pm(X) \times K^{\tau + n'}_\pm(X) \longrightarrow
K^{\tau + \tau' + n + n'}_{\Z/2}(X)
$$
defined by using the product in twisted $K$-groups via the lemma above agrees with that defined in the same way as in Subsection \ref{subsec:definition_variant}.

\subsection{Axiom of $K_\pm$}

The groups $K^{\tau + n}_\pm(X)$ again constitute a $\Z/2$-equivariant cohomology theory, so that the homotopy axiom, the excision axiom, the exactness axiom, and the additivity axiom hold true. 

For later convenience, we here state the axioms precisely: We write $(X, \tau_X)$ for a pair consisting of a space with $\Z/2$-action and a twist $\tau \in \Twist^0_{\Z/2}(X)$. A map $(f, F) : (X, \tau_X) \to (Y, \tau_Y)$ consists of a $\Z/2$-equivariant maps $f : X \to Y$ and an isomorphism of twists $F : \tau_X \to f^*\tau_Y$. A homotopy from $(f_0, F_0) : (X, \tau_X) \to (Y, \tau_Y)$ to $(f_1, F_1) : (X, \tau_X) \to (Y, \tau_Y)$ means a map $(\tilde{f}, \tilde{F}) : (X \times [0, 1], \pi^*\tau_X) \to (Y, \tau_Y)$ such that $(\tilde{f}, \tilde{F})|_{X \times \{ i \}} = (f_i, F_i)$ for $i = 0, 1$, where $\Z/2$ acts on the interval $[0, 1]$ trivially, and $\pi : X \times [0, 1] \to X$ is the projection.

\begin{prop}[Homotopy axiom]
For a homotopy $(\tilde{f}, \tilde{F}) : (X \times [0, 1], \pi^*\tau_X) \to (Y, \tau_Y)$, we have $(f_0, F_0)^* = (f_1, F_1)^*$, so that the following diagram commutes:
$$
\begin{CD}
K^{\tau_Y + n}_\pm(Y) @>{f_0^*}>> K^{f_0^*\tau_Y + n}_\pm(X) \\
@V{f_1^*}VV @VV{F_0^*}V \\
K^{f_1^*\tau_Y + n}_\pm(X) @>{F_1^*}>> K^{\tau_X + n}_\pm(X).
\end{CD}
$$
\end{prop}

Let $\tilde{f} : X \times [0, 1] \to Y$ be a $\Z/2$-equivariant map, and $\tau_Y \in \Twist^0_{\Z/2}(Y)$ a twist. We put $f_j = \tilde{f} \circ \iota_j$ for $j = 0, 1$, where $\iota_j : X \to X \times [0, 1]$ is the inclusion $\iota_j(x) = (x, j)$. Then we have two twists $f_j^*\tau_Y \in \Twist^0_{\Z/2}(X)$.

\begin{cor} \label{cor:homotopy_construct_isomorphism}
In the notation above, there uniquely exists an isomorphism of twists $u(\tilde{f}) : f_0^*\tau_Y \to f_1^*\tau_Y$ such that the following diagram commutes:
$$
\begin{CD}
K^{\tau_Y + n}_\pm(Y) @>{f_0^*}>> K^{f_0^*\tau_Y + n}_\pm(X) \\
@V{f_1^*}VV @| \\
K^{f_1^*\tau_Y + n}_\pm(X) @>{u(\tilde{f})^*}>> K^{f_0^*\tau_Y + n}_\pm(X).
\end{CD}
$$
\end{cor}

\begin{proof}
By the homotopy equivalence $X \times [0, 1] \simeq X$, the inclusion $\iota_0$ induces an isomorphism $\pi_0(\Twist^0_{\Z/2}(X)) \cong \pi_0(\Twist^0_{\Z/2}(X \times [0, 1]))$. In view of this fact, the twists $\pi^*f_0^*\tau_Y, \tilde{f}^*\tau_Y \in \Twist^0_{\Z/2}(X \times [0, 1])$ are isomorphic. Hence there is an isomorphism $\tilde{F} : \pi^*f_0^*\tau_Y \to \tilde{f}^*\tau_Y$. Two such isomorphisms differ by an automorphism of $\pi_*f_0^*\tau_Y$. Since $\mathrm{Aut}(\pi_*f_0^*\tau_Y) \cong \mathrm{Aut}(f_0^*\tau_Y) \cong H^2_{\Z/2}(X)$ canonically, the requirement that $\tilde{F}|_{X \times \{ 0 \}}$ is the identity rules out the ambiguity of $\tilde{F}$. Now, an application of the homotopy axiom to $(\tilde{f}, \tilde{F}) : (X \times [0, 1], \tilde{f}^*\pi_Y) \to (Y, \tau_Y)$ leads to the present corollary, where $u(\tilde{f}) = \tilde{F}|_{X \times \{ 1 \}}$.
\end{proof}

The remaining axioms are as follows:

\begin{prop}
The following holds true.
\begin{itemize}
\item
(Excision axiom)
Let $A$ and $B$ are invariant closed subspaces in a space $X$ with $\Z/2$-action. For any twist $\tau \in \Twist^0_{\Z/2}(X)$ and $n \in \Z$, the inclusion $A \to A \cup B$ induces a natural isomorphism of $K^*_{\Z/2}(X)$-modules:
$$
K^{\tau + n}_\pm(A \cup B, A) \cong K^{\tau + n}_\pm(A, A \cap B),
$$
where these $K$-groups are regarded as $K^*_{\Z/2}(X)$-modules through the inclusions $A \cup B \to X$ and $A \to X$.

\item
(Exactness axiom)
Let $X$ be a space with $\Z/2$-action and $Y \subset X$ an invariant closed subspace. For any twist $\tau \in \Twist^0_{\Z/2}(X)$ and $n \in \Z$, we have a natural exact sequence of $K^*_{\Z/2}(X)$-modules
$$
\cdots \to
K^{n-1}_\pm(Y) \to
K^n_\pm(X, Y) \to
K^n_\pm(X) \to
K^n_\pm(Y) \to
\cdots
$$

\item
(Additivity axiom)
For a family of spaces $X_\lambda$ and twists $\tau_\lambda \in \Twist^0_{\Z/2}(X_\lambda)$, the inclusion induces a natural isomorphism
$$
K^{\sqcup_\lambda \tau_\lambda + n}_\pm(\sqcup_\lambda X_\lambda)
\cong \bigoplus_\lambda K^{\tau_\lambda + n}_\pm(X_\lambda).
$$
\end{itemize}
\end{prop}

\subsection{Basic property of $K_\pm$}

We summarize here the other basic properties.

\begin{prop} \label{prop:basic_exact_sequence_K}
For any space $X$ with $\Z/2$-action and $\tau \in \Twist^0_{\Z/2}(X)$, there are natural exact sequences of modules over $K^*_{\Z/2}(X)$:
\begin{gather*}
\cdots \to
K^{\tau + n-1}_{\pm}(X) \overset{\delta}{\to}
K^{\tau + n}_{\Z/2}(X) \overset{f}{\to}
K^{\tau + n}(X) \to
K^{\tau + n}_{\pm}(X) \overset{\delta}{\to}
\cdots, \\
\cdots \to
K^{\tau + n-1}_{\Z/2}(X) \overset{\delta'}{\to}
K^{\tau + n}_{\pm}(X) \overset{f'}{\to}
K^{\tau + n}(X) \to
K^{\tau + n}_{\Z/2}(X) \overset{\delta'}{\to}
\cdots,
\end{gather*}
where $f$ and $f'$ are the map forgetting the $\Z/2$-action. 
\end{prop}

\begin{proof}
The argument in the proof of Proposition \ref{prop:basic_exact_sequence_K} applies without change. 
\end{proof}

\begin{prop}
We have $R$-module isomorphisms
\begin{align*}
K^0_{\Z/2}(\pt) &\cong R, &
K^1_{\Z/2}(\pt) &= 0, \\
K^0_{\pm}(\pt) &= 0, &
K^1_{\pm}(\pt) &\cong R/J.
\end{align*}
Further, we have a ring isomorphism
\begin{align*}
\mathbb{K}^*(\pt)
&= K^*_{\Z/2}(\pt) \oplus K^*_{\pm}(\pt)
\cong \Z[t, \sigma]/(t^2 - 1, \sigma^2 - 1 + t, (1 + t)\sigma) \\
&\cong \Z[\sigma]/(\sigma^3 - 2\sigma),
\end{align*}
where $\sigma \in K^1_\pm(\pt)$ corresponds to the Euler class $\chi_{\Z/2}(\underline{\R}_1) \in K^{\tau(\underline{\R}_1) + 1}_{\Z/2}(\pt)$ through the isomorphism in Lemma \ref{prop:interpretation_K}.
\end{prop}

\begin{proof}
The $R$-modules structures are straightforward by Proposition \ref{prop:basic_exact_sequence_K}. To determine the ring structure, it suffices to compute $\sigma^2$. Since $\chi_{\Z/2}(\underline{\R}_1)\chi_{\Z/2}(\underline{\R}_1) = \chi_{\Z/2}(\underline{\C}_1) = 1 - t$, we see $\sigma^2 = 1 - t$.
\end{proof}

\begin{prop}
For any space $X$ with $\Z/2$-action, $\tau \in \Twist^0_{\Z/2}(X)$ and $n \in \Z$, the maps $\delta$ and $\delta'$ in Proposition \ref{prop:basic_exact_sequence_K} have the following realizations:
\begin{align*}
\delta &: \ K^{\tau + n}_\pm(X) \to K^{\tau + n+1}_{\Z/2}(X), &
&x \mapsto x \cup \sigma, \\
\delta' &: K^{\tau + n}_{\Z/2}(X) \to K^{\tau + n+1}_\pm(X), &
&x \mapsto x \cup \sigma.
\end{align*}
Therefore the compositions
\begin{gather*}
K^{\tau + n}_\pm(X) \overset{\delta}{\to} 
K^{\tau + n+1}_{\Z/2}(X) \overset{\delta'}{\to} 
K^{\tau + n}_\pm(X), \\
K^{\tau +n}_{\Z/2}(X) \overset{\delta'}{\to} 
K^{\tau +n+1}_{\pm}(X) \overset{\delta}{\to} 
K^{\tau + n}_{\Z/2}(X)
\end{gather*}
are the multiplication of $1 - t \in K^0_{\Z/2}(\mathrm{pt}) = R$.
\end{prop}

\begin{proof}
The same argument as in the case of $H_\pm$ can be adapted.
\end{proof}

\medskip

\begin{lem} 
Let $X$ be a space with free $\Z/2$-action. For any $\tau \in \Twist^0_{\Z/2}(X)$ and $n \in \Z$, there exists a natural isomorphism of groups
$$
K^{\tau + n}_\pm(X) \cong 
K^{\bar{\tau} + \tau(\underline{\R}_1/\Z_2) + n}(X/\Z_2),
$$
where $\bar{\tau} \in \Twist^0(X/\Z_2)$ corresponds to $\tau$ through $\Twist^0_{\Z/2}(X) \cong \Twist^0(X/\Z_2)$.
\end{lem}

\begin{proof}
The twisted $K$-theory of a local quotient groupoid is invariant under local equivalences \cite{FHT1}. The groupoid associated to the free $\Z/2$-space $X$ and the quotient space $X/\Z_2$ are in local equivalence. Now Proposition \ref{prop:interpretation_K} completes the proof.
\end{proof}

\subsection{Thom isomorphism for Real line bundle}

\begin{thm}[Thom isomorphism]
Let $X$ be a space with $\Z/2$-action, and $\pi : R \to X$ a Real line bundle. For $h \in H^3_{\Z/2}(X; \Z)$ and $n \in \Z$, we have natural isomorphisms of modules over $K^*_{\Z/2}(X)$:
\begin{align*}
K^{h + n}_{\Z/2}(X) 
&\cong K^{\pi^*(h + W_3^{\Z_2}(R)) + n}_{\pm}(D(R), S(R)), \\
K^{h + n}_{\pm}(X) 
&\cong K^{\pi^*(h + W_3^{\Z_2}(R)) + n}_{\Z/2}(D(R), S(R)).
\end{align*}
If we write $\Phi_{\pm}(R) \in K^{W_3^{\Z_2}(R)}_{\pm}(D(R), S(R))$ for the image of $1 \in K^0_{\Z/2}(X)$, then the isomorphisms are given by $x \mapsto \pi^*x \cup \Phi_{\pm}(R)$.
\end{thm}

\begin{proof}
Noting Lemma \ref{lem:determinant}, we have
\begin{align*}
[\tau(\underline{\R}_1)] + [\tau(R)]
&= (w_1^{\Z/2}(\underline{\R}_1), 0) + (w_1^{\Z/2}(R), W_3^{\Z/2}(R)) \\
&= (0, W_3^{\Z/2}(R) 
+ \beta(w_1^{\Z/2}(\underline{\R}_1) \cup w_1^{\Z/2}(\underline{\R}_1)))
= (0, W_3^{\Z/2}(R))
\end{align*}
in $\pi_0(\Twist^0_{\Z/2}(X)) \cong H^1_{\Z/2}(X; \Z/2) \times H^3_{\Z/2}(X; \Z)$. Then the Thom isomorphism for equivariant twisted $K$-theory and Proposition \ref{prop:interpretation} complete the proof. 
\end{proof}

As before, the Euler class $\chi_R(R) \in K^{W_3^{\Z/2}(R) + 0}_\pm(X)$ of a Real line bundle $\pi : R \to X$ is defined to be the image of $1 \in K_{\Z/2}(X)$ under the composition of
$$
K_{\Z/2}(X) \overset{\mathrm{Thom}}{\to} 
K^{\pi^*W_3^{\Z_2}(R)}_\pm(D(R), S(R)) \to
K^{\pi^*W_3^{\Z_2}(R)}_\pm(D(R)) \cong 
K^{W_3^{\Z_2}(R)}_\pm(X).
$$
The push-forward $\pi_* : K^{\pi^*(h + W_3^{\Z/2}(R)) + n}_\pm(S(R)) \to K^{h + n-1}_{\Z/2}(X)$ along $\pi: S(R) \to X$ is defined to be the composition of
$$
K^{\pi^*(h + W_3^{\Z_2}(R)) + n}_\pm(S(R)) \to
K^{\pi^*(h + W_3^{\Z_2}(R)) +n+1}_\pm(D(R), S(R)) 
\overset{\mathrm{Thom}}{\to}
K^{h + n -1}_{\Z_2}(X),
$$
and $\pi_* : K^{\pi^*(h + W_3^{\Z/2}(R)) + n}_{\Z/2}(S(R)) \to K^{h + n-1}_{\pm}(X)$ similarly.

\begin{cor}[Gysin sequence]
Let $X$ be a space with $\Z/2$-action, and $\pi : R \to X$ a Real line bundle. For any $h \in H^3_{\Z/2}(X; \Z)$, we have natural exact sequences of $K^*_{\Z/2}(X)$-modules:
\begin{gather*}
\cdot\cdot \to
K^{h + n}_{\Z/2}(X) \overset{\chi_R(R)}{\to}
K^{h + W_3 + n}_{\pm}(X) \overset{\pi^*}{\to}
K^{\pi^*(h +  W_3) + n}_{\pm}(S(R)) \overset{\pi_*}{\to}
K^{h + n-1}_{\Z/2}(X) \to
\cdot\cdot, \\
\cdot\cdot \to
K^{h + n}_{\pm}(X) \overset{\chi_R(R)}{\to}
K^{h + W_3 +  n}_{\Z/2}(X) \overset{\pi^*}{\to}
K^{\pi^*(h + W_3) + n}_{\Z/2}(S(R)) \overset{\pi_*}{\to}
K^{h + n-1}_{\pm}(X) \to
\cdot\cdot,
\end{gather*}
where $W_3 = W_3^{\Z/2}(R)$.
\end{cor}

\smallskip

As in the case of ordinary cohomology, $K_{\Z/2}$ and $K_\pm$ combine to give a single Gysin sequence respecting the $\mathbb{K}^*(X)$-module structures.

\medskip

We will later use the following property of the push-forward:

\begin{prop}
Let $X$ be a space with $\Z/2$-action, $Y \subset X$ an invariant closed subspace, and $h \in H^3_{\Z/2}(X; \Z)$. For any Real line bundle $\pi : R \to X$, the push-forward is compatible with the exact sequences for pairs, i.e.\ the following diagram is commutative (up to sing): 
$$
\begin{array}{cc@{}c@{}c@{}c@{}c@{}c@{}cc}
\cdots & \to &
K^{\pi^*(h + W_3) + n}_{\Z/2}(S, S_Y) & \to &
K^{\pi^*(h + W_3) + n}_{\Z/2}(S) & \to &
K^{\pi^*(h + W_3) + n}_{\Z/2}(S_Y) & \to &
\cdots \\
 & & 
\downarrow \pi_* & &
\downarrow \pi_* & &
\downarrow \pi_* & & \\
\cdots & \to &
K^{h + n-1}_{\pm}(X, Y) &\to &
K^{h + n-1}_{\pm}(X) & \to &
K^{h + n-1}_{\pm}(Y) & \to &
\cdots
\end{array}
$$
where $W_3 = W_3^{\Z/2}(R)$, $S = S(R)$ and $S_Y = S(R)|_Y$. There also exists the commutative diagram with $\Z/2$ and $\pm$ exchanged:
$$
\begin{array}{cc@{}c@{}c@{}c@{}c@{}c@{}cc}
\cdots & \to &
K^{\pi^*(h + W_3) + n}_{\pm}(S, S_Y) & \to &
K^{\pi^*(h + W_3) + n}_{\pm}(S) & \to &
K^{\pi^*(h + W_3) + n}_{\pm}(S_Y) & \to &
\cdots \\
 & & 
\downarrow \pi_* & &
\downarrow \pi_* & &
\downarrow \pi_* & & \\
\cdots & \to &
K^{h + n-1}_{\Z/2}(X, Y) &\to &
K^{h + n-1}_{\Z/2}(X) & \to &
K^{h + n-1}_{\Z/2}(Y) & \to &
\cdots.
\end{array}
$$
\end{prop}

\begin{proof}
Since the roles of $K_{\Z/2}$ and $K_\pm$ are interchangeable, we only consider the former case. Then the only non-trivial part is the commutativity of the diagram
$$
\begin{CD}
K^{\pi^*(h + W_3) + n-1}_{\Z/2}(S_Y) @>>> 
K^{\pi^*(h + W_3) + n}_{\Z/2}(S, S_Y) \\
@V{\pi_*}VV @VV{\pi_*}V \\
K^{h + n-2}_\pm(Y) @>>>
K^{h + n-1}_\pm(X, Y).
\end{CD}
$$
Because of the naturality of the Thom isomorphism, the commutativity of the above diagram follows from that of the following diagram by substituting $\mathcal{X} = D(R)$, $\mathcal{A} = D_Y$, $\mathcal{B} = S$ and $\sigma = \pi^*(h + W_3)$:
$$
\begin{CD}
K^{\sigma + n - 1}_{\Z/2}(\mathcal{A} \cap \mathcal{B}) @>>>
K^{\sigma + n}_{\Z/2}(\mathcal{B}, \mathcal{A} \cap \mathcal{B}) \\
@VVV @VVV \\
K^{\sigma + n}_{\Z/2}(\mathcal{A}, \mathcal{A} \cap \mathcal{B}) @>>> 
K^{\sigma + n + 1}_{\Z/2}(\mathcal{X}, \mathcal{A} \cup \mathcal{B}),
\end{CD}
$$
where the upper row comes from the exact sequence for the pair $(\mathcal{B}, \mathcal{A} \cap \mathcal{B})$, the lower row from the triad $(\mathcal{X}; \mathcal{A}, \mathcal{B})$, the left column from the pair $(\mathcal{A}, \mathcal{A} \cap \mathcal{B})$, and the right column from the triad $(\mathcal{X}; \mathcal{B}, \mathcal{A})$. The commutativity of the diagram further reduces to that of
$$
\begin{CD}
K^{\sigma + n-1}_{\Z/2}(\mathcal{A} \cap \mathcal{B}) @>>>
K^{\sigma + n}_{\Z/2}(\mathcal{B}, \mathcal{A} \cap \mathcal{B}) @>{\cong}>>
K^{\sigma + n}_{\Z/2}(\mathcal{A} \cup \mathcal{B}, \mathcal{A}) \\
@VVV @. @VVV \\
K^{\sigma + n}_{\Z/2}(\mathcal{A}, \mathcal{A} \cap \mathcal{B}) @>{\cong}>>
K^{\sigma + n}_{\Z/2}(\mathcal{A} \cup \mathcal{B}, \mathcal{B}) @>>>
K^{\sigma + n-1}_{\Z/2}(\mathcal{A} \cup \mathcal{B}),
\end{CD}
$$
where the homomorphisms in the above are the connecting homomorphisms in the exact sequences for pairs, or the excision isomorphism. In a direct manner, we can verify the commutativity of the last diagram (up to sign).
\end{proof}


\section{Computation on the torus}
\label{sec:torus}

We compute the equivariant $K$-theory and its variant for the $n$-dimensional torus. In particular, the case of $n = 1, 2, 3$ is studied in some detail.

\medskip

We write $\tilde{S}^1$ for the circle $S^1 \subset \C$ with its $\Z/2$-action given by complex conjugation. In view of the identification $\tilde{I}/\partial{I} \cong \tilde{S}^1$ induced by $\tilde{I} \to \tilde{S}^1$, ($t \mapsto \exp \pi i t$), we choose $-1 \in \tilde{S}^1$ as the base point $\pt \in \tilde{S}^1$. For $n \ge 1$, we put $\tilde{T}^n = (\tilde{S}^1)^n$.

\subsection{General case}

Let $X$ be any space with $\Z/2$-action, and $\pi : X \times \tilde{S}^1 \to X$ the trivial Real circle bundle. Because of the equivariant section $i : X \to X \times \tilde{S}^1$, ($i(x) = (x, \pt)$), the Gysin sequences reduce to the split exact sequences:
\begin{gather*}
0 \to 
K^n_{\Z/2}(X) \overset{\pi^*}{\to} 
K^n_{\Z/2}(X \times \tilde{S}^1) \overset{\pi_*}{\to}
K^{n-1}_{\pm}(X) \to
0, \\
0 \to 
K^n_{\pm}(X) \overset{\pi^*}{\to}
K^n_{\pm}(X \times \tilde{S}^1) \overset{\pi_*}{\to}
K^{n-1}_{\Z/2}(X) \to
0,
\end{gather*}
Thus, for any $n \in \Z$, there are natural $K^*_{\Z/2}(X)$-module isomorphisms:
\begin{align*}
K^n_{\Z/2}(X \times \tilde{S}^1) &\cong 
K^n_{\Z/2}(X) \oplus K^{n-1}_\pm(X), \\
K^n_{\pm}(X \times \tilde{S}^1) &\cong 
K^n_{\pm}(X) \oplus K^{n-1}_{\Z/2}(X).
\end{align*}
Further, these isomorphisms are compatible with the $\mathbb{K}^*(X)$-module structures.

\begin{prop} \label{prop:torus_general_case}
For $n \ge 1$, we have the following $R$-module isomorphisms
\begin{align*}
K^0_{\Z/2}(\tilde{T}^n) &\cong (R \oplus R/J)^{2^{n-1}}, &
K^1_{\Z/2}(\tilde{T}^n) &= 0, \\
K^0_\pm(\tilde{T}^n) &= 0, &
K^1_\pm(\tilde{T}^n) &\cong (R \oplus R/J)^{2^{n-1}}.
\end{align*}
\end{prop}

Notice that the result of $K^*_{\Z/2}(\tilde{T}^n)$ is already known in \cite{MD-R}. (cf.\ \cite{G-S})

\begin{proof}
Regard $\tilde{T}^n = \tilde{T}^{n-1} \times \tilde{S}^1$ as the trivial Real circle bundle. Then the proposition follows from an induction based on the splitting of the Gysin sequence.
\end{proof}

\medskip

We here construct the section of the projection $\pi_* : K^n_{\Z/2}(X \times \tilde{S}^1) \to K^{n-1}_\pm(X)$ in a natural way: As before, let $i : X \to X \times \tilde{S}^1$ be the inclusion $i(x) = (x, \pt)$. Thanks to the projection $\pi : X \times \tilde{S}^1 \to X$, the exact sequence for the pair $(X \times \tilde{S}^1, X \times \pt)$ also reduces to a split exact sequence
$$
0 \to 
K^n_{\Z/2}(X \times \tilde{S}^1, X \times \pt) \overset{j^*}{\to}
K^n_{\Z/2}(X \times \tilde{S}^1) \overset{i^*}{\to}
K^n_{\Z/2}(X \times \pt) \to
0.
$$
Recall $\tilde{I}/\partial \tilde{I} \cong \tilde{S}^1$. We write $j^* : K^{n-1}_\pm(X) \to K^n_{\Z/2}(X \times \tilde{S}^1)$ for the composition of 
$$
K^{n-1}_\pm(X) = K^n_{\Z/2}(X \times \tilde{I}, X \times \partial \tilde{I})
\cong K^n_{\Z/2}(X \times \tilde{S}^1, X \times \pt)
\overset{j^*}{\to} K^n_{\Z/2}(X \times \tilde{S}^1)
$$
by an abuse of notations.

\begin{lem} 
The composition $\pi_* \circ j^*$ is the identity of $K^{n-1}_\pm(X)$:
$$
K^{n-1}_\pm(X) \overset{j^*}{\to}
K^n_{\Z/2}(X \times \tilde{S}^1) \overset{\pi_*}{\to}
K^{n-1}_\pm(X).
$$
\end{lem}

\begin{proof}
Replacing $X$ by $X \sqcup \pt$ if necessary, we can assume $X$ admits a fixed point $p \in X$. The decompositions of $K_\pm$ and $K_{\Z/2}$ by means of the reduced theories $\tilde{K}_\pm$ and $\tilde{K}_{\Z/2}$ are compatible with $j^*$ and $\pi_*$. Hence it suffices to prove that the compositions of the following maps are the identity maps:
\begin{gather*}
\tilde{K}^{n-1}_\pm(X) \overset{j^*}{\to}
K^n_{\Z/2}(X \times \tilde{S}^1, p \times \tilde{S}^1) \overset{\pi_*}{\to}
\tilde{K}^{n-1}_\pm(X), \\
K^{n-1}_\pm(\pt) \overset{j^*}{\to}
K^n_{\Z/2}(\tilde{S}^1) \overset{\pi_*}{\to}
K^{n-1}_\pm(\pt).
\end{gather*}
Recalling the definition of the push-forward and the construction of the Thom isomorphism for Real line bundles, $\pi_* j^*$ in the upper row decomposes as follows:
\begin{align*}
\tilde{K}^{n-1}_\pm(X)
&\cong 
K^n_{\Z/2}(X \times \tilde{S}^1, X \times \pt \cup p \times \tilde{S}^1) \\
&\overset{j^*}{\to}
K^n_{\Z/2}(X \times \tilde{S}^1, p \times \tilde{S}^1) \quad
(\mbox{inclusion}) \\
&\to
K^{n+1}_{\Z/2}(X \times \tilde{D}^2, 
X \times \tilde{S}^1 \cup p \times \tilde{S}^1) \quad
(\mbox{connecting map}) \\
&\cong
K^{n}_{\Z/2}(X \times \tilde{I}, 
X \times \partial \tilde{I} \cup p \times \tilde{I}), \quad
(\mbox{Thom isomorphism})
\end{align*}
where $\tilde{D}^2 \subset \C_1$ is the unit disk, and the last Thom isomorphism is that for $\underline{\R}_0 \to X \times \tilde{I}$. Using the reduced theory, we can write the above maps as follows:
$$
\tilde{K}^n_{\Z/2}(X \wedge \tilde{S}^1)
\overset{j^*}{\to}
K^n_{\Z/2}(X \times \tilde{S}^1, p \times \tilde{S}^1)
\to
\tilde{K}^{n+1}_{\Z/2}(X \wedge \tilde{S}^1 \wedge S^1)
\cong
\tilde{K}^{n}_{\Z/2}(X \wedge \tilde{S}^1).
$$
The composition of the inclusion $j^*$ and the connecting homomorphism is the suspension isomorphism $\tilde{K}^n_{\Z/2}(X \wedge \tilde{S}^1) \cong \tilde{K}^{n+1}_{\Z/2}(X \wedge \tilde{S}^1 \wedge S^1)$. Since the suspension isomorphism is inverse to the Thom isomorphism for $\underline{\R}_0$, we conclude $\pi_*j^*$ is the identity. The remaining case can be shown in the same way.
\end{proof}

\subsection{A basis of $K^0_{\Z/2}(\tilde{S}^1)$}

For $i = 0, 1$, let $\underline{\C}_i = \tilde{S}^1 \times \C_i$ be the equivariant line bundle over $\tilde{S}^1$ constructed from the irreducible representation $\underline{\C}_i$. Let $L \to \tilde{S}^1$ be the equivariant line bundle such that its underlying line bundle is $L = S^1 \times \C$ and $\Z/2$ acts by $\tau(u, z) = (\bar{u}, - uz)$. By abuse of notations, we just write $\underline{\C}_0$, $\underline{\C}_1$ and $L$ to mean the elements they represent in $K^0_{\Z/2}(\tilde{S}^1)$.

We introduce a ring homomorphism
$$
F : \ K^0_{\Z/2}(\tilde{S}^1) \to K^0(S^1) \oplus R_{1} \oplus R_{-1}.
$$
In the above, $R_1 = R_{-1} = R$. The definition of $F$ is $F(x) = (f(x), x|_{1}, x|_{-1})$, where $f$ is to forget the $\Z/2$-action, and $x \mapsto x|_{\pm 1}$ is to focus on the representation of the fiber of a $\Z/2$-equivariant vector bundle at the fixed points $\pm 1 \in \tilde{S}^1$. Noting that $K^0(S^1) \cong \Z$, we can readily see that $F$ takes the following values:
$$
\begin{array}{c|c|cc}
 & K^0(\tilde{S}^1) & R_1 & R_{-1} \\
\hline
\underline{\C}_0 & 1 & 1 & 1 \\
\underline{\C}_1 & 1 & t & t \\
L & 1 & t & 1
\end{array}
$$

\begin{lem} \label{lem:basis_circle_equiv_K}
The following holds true:
\begin{itemize}
\item[(a)]
The group $K^0_{\Z/2}(\tilde{S}^1)$ is generated by $\underline{\C}_0$, $\underline{\C}_1$ and $L$. In particular, respecting the $R$-module structure, we have the following basis:
$$
K^0_{\Z/2}(\tilde{S}^1)
= 
\overbrace{\Z \underline{\C}_0 \oplus \Z \underline{\C}_1}^R \oplus 
\overbrace{\Z (\underline{\C}_0  - L)}^{R/J}.
$$

\item[(b)]
The ring homomorphism $F$ is injective.

\item[(c)]
We have the following relations in $K^0_{\Z/2}(\tilde{S^1})$:
\begin{align*}
L^2 &= \underline{\C}_0, &
\underline{\C}_1 L &= - L + \underline{\C}_0 + \underline{\C}_1.
\end{align*}

\item[(d)]
There are ring isomorphisms:
\begin{align*}
K^0_{\Z/2}(\tilde{S}^1) 
&\cong \Z[t, L]/(t^2 - 1, L^2 - 1, (1 + t)(1 - L)), \\
&\cong \Z[t, \ell]/(t^2 - 1, \ell^2 - 2\ell, (1 + t)\ell),
\end{align*}
where $L$ and $\ell$ are related by $\ell = 1 - L$.
\end{itemize}
\end{lem}

\begin{proof}
For (a), let $K$ be the free abelian group generated by $\underline{\C}_0$, $\underline{\C}_1$ and $L$, and $\iota : K \to \tilde{K}^0_{\Z/2}(\tilde{S}^1)$ the natural homomorphism. To prove that $\iota$ is bijective, we use the Mayer-Vietoris sequence: We define $U$ and $V$ as follows:
\begin{align*}
U &= \{ \exp 2\pi i \theta \in \tilde{S}^1 |\ 1/4 \le \theta \le 3/4 \}, &
V &= \{ \exp 2\pi i \theta \in \tilde{S}^1 |\ -1/4 \le \theta \le 1/4 \}.
\end{align*}
Apparently, $U$ and $V$ are equivariantly contractible, and $U \cap V \cong \partial \tilde{I}$. The Mayer-Vietoris sequence and Proposition \ref{prop:torus_general_case} provide us:
$$
0 \to 
K^0_{\Z/2}(\tilde{S}^1) \overset{i^*}{\to}
K^0_{\Z/2}(U \sqcup V) \overset{\Delta}{\to}
K^0_{\Z/2}(U \cap V) \to
0,
$$
where $\Delta : R \oplus R \to \Z$ is $\Delta(a + b t, a' + b't) = a + b - a' - b'$. We can verify $i^* : \iota(K) \to \mathrm{Ker}\Delta$ is bijective. Hence so is $\iota$, and an additive basis of $K^0_{\Z/2}(\tilde{S}^1)$ is specified. From the values of $F$, the subgroup $\tilde{K}^0_{\Z/2}(\tilde{S}^1)$ is generated by $\underline{\C}_0 - L_1$, which induces the basis respecting the $R$-module structure in (a). Now, (b) is verified directly. Using (b), we get (c), which leads to (d).
\end{proof}

\subsection{A basis of $K^0_{\Z/2}(\tilde{T}^2)$}
\label{subsec:basis_torus}

We write $\pi_i : \tilde{T}^2 \to \tilde{S}^1$ for the $i$th projection from $\tilde{T}^2 = \tilde{S}^1 \times \tilde{S}^1$. We define the equivariant line bundle $L_i \to \tilde{T}^2$ by the pull-back $L_i = \pi_i^*L$. We also define the equivariant line bundle $H' \to \tilde{T}^2$ by
$$
H' = S^1 \times \R \times \C/\sim,
$$
where $(u, x, z) \sim (u, x + n, u^nz)$ for $n \in \Z$. The $\Z/2$-action $\tau : H' \to H'$ is $\tau([u, x, z]) = [\bar{u}, -x, z]$ and the projection $\pi : H' \to \tilde{T}^2$ is $\pi([u, x, z]) = (u, \exp 2\pi i x)$. We then define $H \to \tilde{T}^2$ by $H = \underline{\C}_1L_1L_2H'$. If we forget about $\Z/2$-actions, then $c_1(H) = 1$ in $H^2(\tilde{T}^2; \Z) \cong \Z$ with respect to the standard orientation of $\tilde{T}^2$, and there is a ring isomorphism $K^0(\tilde{T}^2) \cong \Z[H]/(1 - H)^2$.

As before, we put $F(x) = (f(x), x|_{(i, j)})$ and define the ring homomorphism
$$
F : \
K^0_{\Z/2}(T^2) \longrightarrow
K^0(T^2) \oplus \bigoplus_{(i, j) = (\pm 1, \pm 1)} R_{(i, j)},
$$
where $R_{(i, j)} = R$. We have the following values of $F$:
$$
\begin{array}{c|c|cccc}
 & K^0(\tilde{T}^2) & R_{(1, 1)} & R_{(-1, 1)} & R_{(1, -1)} & R_{(-1, -1)} \\
\hline
\underline{\C}_0 & 1 & 1 & 1 & 1 & 1 \\
\underline{\C}_1 & 1 & t & t & t & t \\
H & H & t & 1 & 1 & 1 \\
\underline{\C}_1H & H & 1 & t & t & t \\
L_1 & 1 & t & 1 & t & 1 \\
L_2 & 1 & t & t & 1 & 1
\end{array}
$$

\begin{lem} \label{lem:additive_basis_circle_K}
The following holds true:
\begin{itemize}
\item[(a)]
The group $K^0_{\Z/2}(\tilde{T}^2)$ is generated by $\{ \underline{\C}_0, \underline{\C}_1, H, \underline{\C}_1H, L_1, L_2\}$. In particular, respecting the $R$-module structure, we have the following basis:
$$
\tilde{K}^0_{\Z/2}(\tilde{T}^2)
=
\overbrace{
\Z \underline{\C}_0 \oplus \Z \underline{\C}_1
}^R
\oplus
\overbrace{
\Z (\underline{\C}_0 - H) \oplus \Z \underline{\C}_1(\underline{\C}_0 - H) 
}^R
\oplus
\overbrace{\Z (\underline{\C}_0 - L_1)}^{R/J} \oplus 
\overbrace{\Z (\underline{\C}_0 - L_2)}^{R/J}.
$$

\item[(b)]
The ring homomorphism $F$ is injective.

\item[(c)]
The following relations are satisfied in $K^0_{\Z/2}(\tilde{T}^2)$.
\begin{align*}
(\underline{\C}_0 + \underline{\C}_1)
(\underline{\C}_0 - L_1) &= 0, &
(\underline{\C}_0 + \underline{\C}_1)
(\underline{\C}_0 - L_2) &= 0, \\
(\underline{\C}_0 - H) (\underline{\C}_1 - L_1) &= 0, &
(\underline{\C}_0 - H) (\underline{\C}_1 - L_2) &= 0, \\
(\underline{\C}_0 - H)(\underline{\C}_1 - H) &= 0, &
(\underline{\C}_0 - L_1)
(\underline{\C}_0 - L_2)
&=
(\underline{\C}_0 - \underline{\C}_1)(\underline{\C}_0 - H).
\end{align*}

\end{itemize}
\end{lem}

\begin{proof}
Let $K$ be the subgroup generated by $\{ \underline{\C}_0, \underline{\C}_1, H, \underline{\C}_1H, L_1, L_2\}$, and $\iota : K \to K^1_{\Z/2}(\tilde{T}^2)$ the natural inclusion. To prove that $\iota$ is bijective, we use the Mayer-Vietoris exact sequence: Let $U, V \subset \tilde{S}^1$ be as in the proof of Lemma \ref{lem:basis_circle_equiv_K}. We put $U' = \tilde{S}^1 \times U$ and $V' = \tilde{S}^1 \times V$. Then the Mayer-Vietoris sequence for $\{ U', V' \}$ and Proposition \ref{prop:torus_general_case} give us the exact sequence
$$
0 \to 
K^1_{\Z/2}(U' \cap V') \overset{\xi_{\Z/2}}{\to}
K^0_{\Z_2}(\tilde{T}^2) \overset{i^*_{\Z/2}}{\to}
K^0_{\Z/2}(U' \sqcup V') \overset{\Delta_{\Z/2}}{\to}
K^0_{\Z/2}(U' \cap V') \to
0,
$$
We can directly verify $i^*_{\Z/2}(K) = \mathrm{Ker}\Delta_{\Z/2}$. Bearing this fact in mind, we consider the following maps of exact sequences
$$
\begin{CD}
0 @>>> K' @>>> K @>>> \mathrm{Ker}\Delta_{\Z/2} @>>> 0 \\
@. @V{\iota'}VV @VV{\iota}V @| @. \\
0 @>>>
K^1_{\Z/2}(U' \cap V') @>{\xi_{\Z/2}}>> 
K^0_{\Z/2}(\tilde{T}^2) @>{i^*_{\Z/2}}>> \mathrm{Ker}\Delta_{\Z/2} @>>> 
0 \\
@. @V{f_{U' \cap V'}}VV @VV{f_{\tilde{T}^2}}V @VV{f}V @. \\
@. K^1(U' \cap V') @>\xi>> 
K^0(\tilde{T}^2) @>i^*>>
\mathrm{Ker}\Delta @>>>
0,
\end{CD}
$$
where $K'$ is the kernel of $i^*_{\Z/2}$ restricted to $K$, the injection $\iota'$ is induced from $\iota$, the third row comes from the Mayer-Vietoris sequence for usual $K$-theory, and $f$ are the maps forgetting $\Z/2$-actions. Now, the thing to be shown is $\iota'(K') = K^1_{\Z/2}(U' \cap V')$. A direct computation proves that $K'$ is generated by $\underline{\C}_1H + H - \underline{\C}_1 - \underline{\C}_0$. Using this information and the diagram above, we find 
$$
\xi(f_{U' \cap V'}(\iota'(K'))) = \xi(f_{U' \cap V'}(K^1_{\Z/2}(U' \cap V')))
= 2\Z H
$$ 
in $K^0(\tilde{T}^2) = \Z \underline{\C} \oplus \Z H$. Since $\xi \circ f_{U' \cap V'}$ is injective, we have $\iota'(K') = K^1_{\Z/2}(U' \cap V')$. Therefore $\iota$ is bijective, and $K = K^1_{\Z/2}(\tilde{T}^2)$. Because $L_1$ and $L_2$ are the pull-back of $L \to \tilde{S}^1$ under the projections $\pi_i : \tilde{T}^2 \to \tilde{S}^1$, we get the basis respecting the $R$-module structure of $K = K^1_{\Z/2}(\tilde{T}^2)$, completing the proof of (a). Once a basis is known, (b) is directly verified. Using (b), we can check (c).
\end{proof}


\begin{rem}
The idea of using the ring homomorphism $F$ comes from a work of Hajime Fujita, where finite group actions on Riemann surfaces are studied. The original proof of the injectivity of $F$ appeals to a localization formula of the equivariant index, by which we can further specify ring generators of equivariant $K$-theory.
\end{rem}

\subsection{A basis of $K^0_{\Z/2}(\tilde{T}^3)$}

Let $\pi_i : \tilde{T}^3 \to \tilde{S}^1$ be the $i$th projection from $\tilde{T}^3 = \tilde{S}^1 \times \tilde{S}^1 \times \tilde{S}^1$. We define the equivariant line bundle $L_i \to \tilde{T}^3$ by $L_i = \pi_i^*L$ for $i = 1, 2, 3$, and $H_{ij} \to \tilde{T}^3$ by $H_{ij} = (\pi_i, \pi_j)^*H$ for $i < j$. By the Atiyah-Hirzebruch spectral sequence, we find the group $K^0(\tilde{T}^3)$ is generated by $H_{12}, H_{23}, H_{13}$ and $\underline{\C}$. The multiplication in $K^0(\tilde{T}^3)$ can be computed through the Chern character: For instance, we have $(\underline{\C} - H_{12})(\underline{\C} - H_{23}) = 0$.

Setting $F(x) = (f(x), x|_{(i, j, k)})$ as before, we define the ring homomorphism
$$
F : \
K^0_{\Z/2}(\tilde{T}^3) \longrightarrow
K^0(\tilde{T}^3) \oplus 
\bigoplus_{(i, j, k) = (\pm 1, \pm 1, \pm 1)} R_{(i, j, k)},
$$
where $R_{(i, j, k)} = R$. A table of values of $F$ is in Figure \ref{fig:3_torus}.

\begin{figure}[htbp] 
$$
\begin{array}{c@{}|c|@{}c@{}c@{}c@{}c|@{}c@{}c@{}c@{}c}
K^0_{\Z/2} &
K^0 & 
\scriptstyle{(1, 1, 1)} & \scriptstyle{(-1, 1, 1)} & 
\scriptstyle{(1, -1, 1)} & \scriptstyle{(-1, -1, 1)} &
\scriptstyle{(1, 1, -1)} & \scriptstyle{(-1, 1, -1)} &
\scriptstyle{(1,-1,-1)} & \scriptstyle{(-1, -1, -1)} \\
\hline
\underline{\C}_0 & 1 & 
1 & 1 & 1 & 1 & 
1 & 1 & 1 & 1 \\
\underline{\C}_1 & 1 & 
t & t & t & t & 
t & t & t & t \\
\hline
H_{12} & H_{12} & 
t & 1 & 1 & 1 & 
t & 1 & 1 & 1 \\
H_{23} & H_{23} & 
t & t & 1 & 1 & 
1 & 1 & 1 & 1 \\
H_{13} & H_{13} & 
t & 1 & t & 1 & 
1 & 1 & 1 & 1 \\
\hline
\underline{\C}_1H_{12} & H_{12} & 
1 & t & t & t & 
1 & t & t & t \\
\underline{\C}_1H_{23} & H_{23} & 
1 & 1 & t & t & 
t & t & t & t \\
\underline{\C}_1H_{13} & H_{13} & 
1 & t & 1 & t & 
t & t & t & t \\
\hline
L_1 & 1 &
t & 1 & t & 1 &
t & 1 & t & 1 \\
L_2 & 1 &
t & t & 1 & 1 &
t & t & 1 & 1 \\
L_3 & 1 &
t & t & t & t &
1 & 1 & 1 & 1 \\
\hline
H_{12}L_3 & H_{12} &
1 & t & t & t &
t & 1 & 1 & 1
\end{array}
$$
\caption{The values of $F$ on $\tilde{T}^3$.}\label{fig:3_torus}
\end{figure}

\begin{lem} \label{lem:additive_basis_3_torus_K}
The following holds true:
\begin{itemize}
\item[(a)]
The group $K^0_{\Z/2}(\tilde{T}^3)$ is generated by:
$$
\{
\underline{\C}_0, \underline{\C}_1, 
H_{12}, H_{23}, H_{13}, 
\underline{\C}_1H_{12}, \underline{\C}_1H_{23}, \underline{\C}_1H_{13}, 
L_1, L_2, L_3, H_{12}L_3
\}.
$$
In particular, the following basis respects the $R$-module structure:
\begin{align*}
\tilde{K}^0_{\Z/2}(\tilde{T}^3)
&=
\overbrace{
\Z \underline{\C}_0 \oplus \Z \underline{\C}_1
}^R
\oplus
\overbrace{
\Z (\underline{\C}_0 - H_{12}) \oplus 
\Z \underline{\C}_1(\underline{\C}_0 - H_{12}) 
}^R \\
&
\oplus
\overbrace{
\Z (\underline{\C}_0 - H_{23}) \oplus 
\Z \underline{\C}_1(\underline{\C}_0 - H_{23}) 
}^R
\oplus
\overbrace{
\Z (\underline{\C}_0 - H_{13}) \oplus 
\Z \underline{\C}_1(\underline{\C}_0 - H_{13}) 
}^R \\
& 
\oplus
\overbrace{\Z (\underline{\C}_0 - L_1)}^{R/J} \oplus 
\overbrace{\Z (\underline{\C}_0 - L_2)}^{R/J} \oplus
\overbrace{\Z (\underline{\C}_0 - L_3)}^{R/J} \oplus
\overbrace{\Z (\underline{\C}_0 - H_{12})(\underline{\C}_0 - L_3)}^{R/J}.
\end{align*}

\item[(b)]
The ring homomorphism $F$ is injective.

\end{itemize}
\end{lem}

\begin{proof}
The proof is exactly the same as in the case of $\tilde{T}^2$, so is omitted.
\end{proof}

\subsection{The ring structure of $\mathbb{K}^*(\tilde{S}^1)$}

\begin{prop} \label{prop:ring_structure_KK_circle}
The following holds true:
\begin{itemize}
\item[(a)]
There is a ring isomorphism
$$
\mathbb{K}^*(\tilde{S}^1) 
= K^*_{\Z/2}(\tilde{S}^1) \oplus K^*_{\pm}(\tilde{S}^1)
\cong \Z[\sigma, \chi]/(\sigma^3 - 2\sigma, \chi^2 - \sigma \chi),
$$
where $\chi \in \tilde{K}^1_\pm(\tilde{S}^1) \cong R$ is the unique element whose image under the push-forward along the projection $\pi : \tilde{S}^1 \to \pt$ is $\pi_*\chi = 1$, and $\sigma \in K^1_\pm(\pt) \cong R/J$ is the generator.  

\item[(b)]
The elements $1, t, \sigma \chi \in K^0_{\Z/2}(\tilde{S}^1)$ have the following realization:
\begin{align*}
1 &= \underline{\C}_0, &
t &= \underline{\C}_1, &
\sigma \chi &= \underline{\C}_0 - L.
\end{align*}
The elements $\chi, t\chi, \sigma \in K^1_\pm(\tilde{S}^1)$ have the following realization through the injection $j^* : K^1_\pm(\tilde{S}^1) \to K^0_{\Z/2}(\tilde{T}^2)$:
\begin{align*}
j^*(\chi) &= \underline{\C}_0 - H, &
j^*(t\chi) &= \underline{\C}_1(\underline{\C}_0 - H), &
j^*(\sigma) &= \underline{\C}_0 - L_2.
\end{align*}
\end{itemize}
\end{prop}

\medskip

The proof requires a number of lemmas.

\begin{lem} \label{lem:realize_sigma}
Through the injection $j^* : K^1_\pm(\pt) \to K^0_{\Z/2}(\tilde{S}^1)$, the generator $\sigma \in K^1_\pm(\pt)$ corresponds to $\underline{\C}_0 - L \in K^0_{\Z/2}(\tilde{S}^1)$.
\end{lem}

\begin{proof}
Recall that $j^*$ is the composition of
$$
K^1_\pm(\pt) = K^0_{\Z/2}(\tilde{I}, \partial \tilde{I})
\cong K^0_{\Z/2}(\tilde{S}^1, \pt) \to K^0_{\Z/2}(\tilde{S}^1).
$$
The isomorphism in the middle is induced from $\phi : \tilde{I}/\partial \tilde{I} \cong \tilde{S}^1$, ($t \mapsto \exp \pi i t$). By a homotopy and some isomorphisms, we can directly identify the pull-back $\phi^*(\underline{\C}_0 - L)$ with the Thom class $\Phi_{\Z/2}(\underline{\R}_1) \in K^0_{\Z/2}(\tilde{I}, \partial \tilde{I})$ of $\underline{\R}_1 \to \pt$, which is realized by the triple $(\underline{\C}_0, \underline{\C}_1, \eta)$ consisting of equivariant line bundles $\underline{\C}_i = \tilde{I} \times \C_i$ on $\tilde{I}$ and the equivariant bundle map $\eta : \underline{\C}_0 \to \underline{\C}_1$ given by $\eta(t, z) = (t, tz)$.
\end{proof}

\begin{lem} \label{lem:additive_basis_circle_K_pm}
The following holds.
\begin{itemize}
\item[(a)]
The group $K^1_\pm(\tilde{S}^1)$ has the additive basis $\{ x, y, z \}$ such that
\begin{align*}
j^*x &= \underline{\C}_0 - H, &
j^*y &= \underline{\C}_1(\underline{\C}_0 - H), &
j^*z &= \underline{\C}_0 - L_2.
\end{align*}

\item[(b)]
We have $z = \sigma$.
\end{itemize}
\end{lem}

\begin{proof}
We can verify (a) using the ring monomorphism $F$ for $\tilde{T}^2$. For (b), we observe that $\underline{\C}_0 - L_2 \in K^0_{\Z/2}(\tilde{S}^1 \times \tilde{S}^1, \tilde{S}^1 \times pt)$ is the pull-back of $\underline{\C}_0 - L \in K^0_{\Z/2}(\tilde{S}^1, \pt)$ under $\pi_2 : \tilde{S}^1 \times \tilde{S}^1 \to \tilde{S^1}$. Then Lemma \ref{lem:realize_sigma} implies $z = \sigma$. 
\end{proof}

\begin{lem} \label{lem:identify_thom_class_from_2_sphere}
Let $c : \tilde{S}^1\times \tilde{S}^1 \to \tilde{S}^1 \wedge \tilde{S}^1= \C_1 \cup \{ \infty \}$ be the natural projection, and $\Phi_{\Z/2}(\underline{\C}_1) \in \tilde{K}^0_{\Z/2}(\C_1 \cup \{ \infty \})$ the Thom class of $\underline{\C}_1 \to \mathrm{pt}$. Then we have
$$
c^*\Phi_{\Z/2}(\underline{\C}_1) = \underline{\C}_0 - H
$$
in $K^0_{\Z/2}(\tilde{S}^1 \times \tilde{S}^1)$.
\end{lem}

\begin{proof}
By definition, $\Phi = \Phi_{\Z/2}(\underline{\C}_1) \in \tilde{K}^0_{\Z/2}(\C_1 \cup \{ \infty \})) \cong K^0_{\Z/2}(D(\C_1), S(\C_1))$ is realized by the triple $(\underline{\C}_0, \underline{\C}_1, \eta)$, where $\underline{\C}_i$ is the equivariant line bundles $D(\C_1) \times \C_i$ over $D(\C_1)$ and $\eta : \underline{\C}_0 \to \underline{\C}_1$ is the bundle map $(\xi, z) \mapsto (\xi, \xi z)$. By the compatibility with forgetting the group action, we have $c^*\Phi = 1 - H \in K^0(\tilde{T}^2)$, where $c_1(H) = 1$. The map $c$ carries $(1, 1) \in \tilde{T}^2$ to the point $0 \in \C_1 \cup \{ \infty \}$, while $(-1, 0), (0, -1), (-1, -1)$ to $\infty$. Thus, $F$ takes the following values for $c^*\Phi$:
$$
\begin{array}{c|c|cccc}
K^0_{\Z/2}(\tilde{T}^2) & 
K^0(\tilde{T}^2) & (1, 1) & (-1, 0) & (0, -1) & (-1, -1) \\
\hline
c^*\Phi & 1 - H & 1 - t & 0 & 0 & 0
\end{array}
$$
This implies $c^*\Phi = \underline{\C}_0 - H$.
\end{proof}

\begin{lem}
The image of $x \in K^1_\pm(\tilde{S}^1)$ under the push-forward along the projection $\pi : \tilde{S}^1 \to \pt$ is $\pi_* x = 1$ in $K^0_{\Z/2}(\mathrm{pt})$.
\end{lem}

\begin{proof}
Let $\tilde{D} \subset \C_1$ be the unit disk with the $\Z/2$-action $z \mapsto \bar{z}$, so that $\partial \tilde{D} = \tilde{S}^1$. The push-forward $\pi_* : K^1_\pm(\tilde{S}^1) \to K^0_{\Z/2}(\mathrm{pt})$ factors as follows:
\begin{align*}
K^0_{\Z/2}(\tilde{S}^1 \times \tilde{I}, 
\tilde{S}^1 \times \partial \tilde{I})
&\cong
K^0_{\Z/2}(\partial (\tilde{D} \times \tilde{I}), 
\tilde{D} \times \partial \tilde{I}) \quad
(\mbox{excision}) \\
&\to
K^0_{\Z/2}(\partial (\tilde{D} \times \tilde{I})) \quad
(\mbox{restriction})\\
&\cong
K^{-2}_{\Z/2}(\partial (\tilde{D} \times \tilde{I})) \quad
(\mbox{periodicity}) \\
&\to
K^{-1}_{\Z/2}(\tilde{D} \times \tilde{I}, 
\partial (\tilde{D} \times \tilde{I})) \quad
(\mbox{connecting})\\
&\cong
K^0_{\Z/2}(\mathrm{pt}). \quad
(\mbox{Thom isomorphism})
\end{align*}
As a result of Lemma \ref{lem:additive_basis_circle_K_pm}, we have $x \in K^1_\pm(\tilde{S}^1, \mathrm{pt}) \subset K^1_\pm(\tilde{S}^1)$, so that we restrict ourselves to consider the push-forward $\pi_* : K^1_\pm(\tilde{S}^1, \mathrm{pt}) \to K^0_{\Z/2}(\mathrm{pt})$, which is bijective. The factorization of $\pi_*$ above now reduces to:
\begin{align*}
K^0_{\Z/2}(\tilde{S}^1 \times \tilde{I}, 
\tilde{S}^1 \times \partial \tilde{I} \cup \mathrm{pt} \times \tilde{I})
&\cong
K^0_{\Z/2}(\partial (\tilde{D} \times \tilde{I}), 
\tilde{D} \times \partial \tilde{I} \cup \mathrm{pt} \times \tilde{I}) \quad
(\mbox{excision}) \\
&\to
K^0_{\Z/2}(\partial (\tilde{D} \times \tilde{I}), \mathrm{pt}) \quad
(\mbox{restriction})\\
&\cong
K^{-2}_{\Z/2}(\partial (\tilde{D} \times \tilde{I}), \mathrm{pt}) \quad
(\mbox{periodicity}) \\
&\to
K^{-1}_{\Z/2}(\tilde{D} \times \tilde{I}, 
\partial (\tilde{D} \times \tilde{I})) \quad
(\mbox{connecting})\\
&\cong
K^0_{\Z/2}(\mathrm{pt}). \quad
(\mbox{Thom isomorphism})
\end{align*}
This factorization is further identified with:
\begin{align*}
\tilde{K}^0_{\Z/2}(\tilde{S}^1 \wedge \tilde{S}^1)
&\cong
\tilde{K}^0_{\Z/2}(S^1 \wedge S^1 \wedge \tilde{S} \wedge \tilde{S}) \quad
(\mbox{periodicity}) \\
&\cong
K^0_{\Z/2}(\mathrm{pt}) \quad 
(\mbox{Thom isomorphism})
\end{align*}
Thus, $\pi_* : K^1_\pm(\tilde{S}^1, \mathrm{pt}) \to K^0_{\Z/2}(\mathrm{pt})$ turns out to be just the Thom isomorphism $\tilde{K}^0_{\Z/2}(\tilde{S}^1 \wedge \tilde{S}^1) = \tilde{K}^0_{\Z/2}(D(\C_1)/S(\C_1)) \cong K^0_{\Z/2}(\mathrm{pt})$. To relate the Thom class $\Phi = \Phi_{\Z/2}(\underline{\C}_1) \in \tilde{K}^0_{\Z/2}(\tilde{S}^1 \wedge \tilde{S}^1)$ of $\underline{\C}_1 \to \pt$ with $x \in K^1_\pm(\tilde{S}^1)$, we recall the expression $K^1_\pm(\tilde{S}) \cong K^1_{\Z/2}(\tilde{S} \times \tilde{S}, \tilde{S} \times \mathrm{pt})$ and consider the commutative diagram:
$$
\begin{CD}
K^1_\pm(\tilde{S}^1) @=
K^0_{\Z/2}(\tilde{S}^1 \times \tilde{S}^1, \tilde{S}^1 \times \mathrm{pt})
@>{j^*}>> 
K^0_{\Z/2}(\tilde{S}^1 \times \tilde{S}^1) 
\\
@AAA @AAA @AA{c^*}A \\
\tilde{K}_\pm(\tilde{S}^1) @=
K^0_{\Z/2}(\tilde{S}^1 \times \tilde{I}, 
\tilde{S}^1 \times \partial \tilde{I} \cup \mathrm{pt} \times \tilde{I})
@=
\tilde{K}^0(\tilde{S}^1 \wedge \tilde{S}^1),
\end{CD}
$$
where $c : \tilde{S}^1 \times \tilde{S}^1 \to \tilde{S}^1 \wedge \tilde{S}^1$ is the natural projection. As is seen by using the long exact sequence for a pair, $c^*$ is injective. Now, from our knowledge of $c^*\Phi$ and the definition of $x$, we get $c^*\Phi = j^*x$, so that $\pi_*x = 1$.
\end{proof}

\begin{lem} \label{lem:x_square}
In $K^*_{\Z/2}(\tilde{S}^1) \oplus K^*_{\pm}(\tilde{S}^1)$, we have:
$$
x^2 = zx = \underline{\C}_0 - L.
$$
\end{lem}

\begin{proof}
Notice that $x, tx \in \tilde{K}^1_\pm(\tilde{S}^1)$. Note also $\tilde{K^1}_\pm(\tilde{S}^1) \cong \tilde{K}^0_{\Z/2}(\tilde{S}^1 \wedge \tilde{S}^1)$. The multiplication $\tilde{K}^1_\pm(\tilde{S}^1) \times \tilde{K}^1_\pm(\tilde{S}^1) \to \tilde{K}^0_{\Z/2}(\tilde{S}^1)$ is the composition of:
$$
\tilde{K}^0_{\Z/2}(\tilde{S}^1 \wedge \tilde{S}^1) \times
\tilde{K}^0_{\Z/2}(\tilde{S}^1 \wedge \tilde{S}^1) \overset{\mu}{\to}
\tilde{K}^0_{\Z/2}(\tilde{S}^1_1 \wedge \tilde{S}^1_2 \wedge \tilde{S}^1_3) 
\overset{\beta}{\to}
\tilde{K}^0_{\Z/2}(\tilde{S}^1).
$$
In the above, $\tilde{S}^1_k = \tilde{S}^1$ for $k = 1, 2, 3$, and $\mu$ is defined as $\mu(a, b) = j_{12}^*(a) \cdot j_{13}^*(b)$ by using the map $j_{1k} : \tilde{S}^1_1 \wedge \tilde{S}^1_2 \wedge \tilde{S}^1_3 \to \tilde{S}^1 \wedge \tilde{S}^1$, $(t_1, t_2, t_3) \mapsto (t_1, t_k)$ for $k = 2, 3$. The isomorphism $\beta$ is the Thom isomorphism applied to $\tilde{S}^1_2 \wedge \tilde{S}^1_3$. We here use the commutative diagram:
$$
\begin{CD}
K^0_{\Z/2}(\tilde{S}^1 \times \tilde{S}^1) \times
K^0_{\Z/2}(\tilde{S}^1 \times \tilde{S}^1) @>{\mu}>>
K^0_{\Z/2}(\tilde{S}^1 \times \tilde{S}^1 \times \tilde{S}^1) \\
@A{c^* \times c^*}AA @AA{c^*}A \\
\tilde{K}^0_{\Z/2}(\tilde{S}^1 \wedge \tilde{S}^1) \times
\tilde{K}^0_{\Z/2}(\tilde{S}^1 \wedge \tilde{S}^1) @>{\mu}>>
\tilde{K}^0_{\Z/2}(\tilde{S}^1_1 \wedge \tilde{S}^1_2 \wedge \tilde{S}^1_3),
\end{CD}
$$
where the projections $c : \ \tilde{S}^1 \times \cdots \times \tilde{S}^1 \to \tilde{S}^1 \wedge \cdots \wedge \tilde{S}^1$ induce the vertical maps. By the long exact sequence for a pair, $c^*$ on the right turns out to be injective. Now, we use $F$ for $\tilde{T}^3$ to see
\begin{align*}
\mu(c^*x, c^*x)
&= j_{12}^*(\underline{\C}_0 - H) 
\cdot j_{13}^*(\underline{\C}_0 - H) \\
&= (\underline{\C}_0 - H_{12})
(\underline{\C}_0 - H_{13})
= (\underline{\C}_0 - L_1)(\underline{\C}_0 - H_{23}).
\end{align*}
The Thom class $\Phi_{23}$ producing $\beta : \tilde{K}^0_{\Z/2}(\tilde{S}^1_1 \wedge \tilde{S}^1_2 \wedge \tilde{S}^1_3) \to \tilde{K}^0_{\Z/2}(\tilde{S}^1)$ admits the expression $c^*\Phi_{23} = \underline{\C}_0 - H_{23}$ by Lemma \ref{lem:identify_thom_class_from_2_sphere}. Hence we get $x^2 = \underline{\C}_0 - L_1$ in $\tilde{K}^0_{\Z/2}(\tilde{S}^1)$. The computation of $xz$ proceeds in the same way. A use of $F$ shows
\begin{align*}
\mu(c^*x, c^*z)
&= j_{12}^*(\underline{\C}_0 - H) 
\cdot j_{13}^*(\underline{\C}_0 - L_2) \\
&= (\underline{\C}_0 - H_{12})
(\underline{\C}_0 - L_3)
= (\underline{\C}_0 - L_1)(\underline{\C}_0 - H_{23}),
\end{align*}
so that $xz = \underline{\C}_0 - L_1$. 
\end{proof}

\begin{proof}[The proof of Proposition \ref{prop:ring_structure_KK_circle}]
Let $x, y, z \in \tilde{K}^1_\pm(\tilde{S}^1)$ be the additive basis given in Lemma \ref{lem:additive_basis_circle_K_pm}. As is seen, $z = \sigma$ comes from $\tilde{K}^1_\pm(\pt)$, so that $x$ and $y = tx$ generate $\tilde{K}^1_\pm(\tilde{S}^1) \cong R$. Since $\pi_* : \tilde{K}^1_\pm(\tilde{S}^1) \to K^0_{\Z/2}(\pt)$ is an isomorphism $R \cong R$, the characterizing condition on $\chi$ implies $\chi = x$. By Lemma \ref{lem:x_square} and \ref{lem:additive_basis_circle_K}, we can see the ring structure of $\mathbb{K}^*(\tilde{S}^1)$ is as stated. The correspondences with geometric realizations are clear from the construction.
\end{proof}

\subsection{The ring structure of $\mathbb{K}^*(\tilde{T}^2)$}

As before, we write $\pi_i : \tilde{S}^1 \times \tilde{S}^1 \to \tilde{S}^1$ for the $i$th projection.

\begin{prop} \label{prop:ring_structure_KK_torus}
The following holds true:
\begin{itemize}
\item[(a)]
There is a ring isomorphism
$$
\mathbb{K}^*(\tilde{T}^2) 
= K^*_{\Z/2}(\tilde{T}^2) \oplus K^*_{\pm}(\tilde{T}^2)
\cong \Z[\sigma, \chi_1, \chi_2]/
(\sigma^3 - 2\sigma, \chi^2_1 - \sigma \chi_1, \chi_2^2 - \sigma \chi_2),
$$
where $\chi_i \in \tilde{K}^1_\pm(\tilde{S}^1)$ are the unique elements such that $(\pi_1)_* \chi_2 = 1$ and $(\pi_2)_*\chi_1 = 1$ in $K^0_{\Z/2}(\tilde{S}^1)$, and $\sigma \in K^1_\pm(\pt)$ is the generator.

\item[(b)]
The elements $1, t, \chi_1\chi_2, t\chi_1\chi_2, \sigma \chi_1, \sigma \chi_2 \in K^0_{\Z/2}(\tilde{T}^2)$ have the realization:
\begin{align*}
1 &= \underline{\C}_0, &
t &= \underline{\C}_1, \\
\chi_1\chi_2 &= \underline{\C}_0 - H, &
t\chi_1\chi_2 &= \underline{\C}_1(\underline{\C}_0 - H), \\
\sigma \chi_1 &= \underline{\C}_0 - L_1, &
\sigma \chi_2 &= \underline{\C}_0 - L_2.
\end{align*}
The elements $\chi_1, t\chi_1, \chi_2, t\chi_2, \sigma, \sigma \chi_1\chi_2 \in K^1_\pm(\tilde{T}^2)$ have the following realization through the injection $j^* : K^1_\pm(\tilde{T}^2) \to K^0_{\Z/2}(\tilde{T}^3)$:
\begin{align*}
j^*(\chi_1) &= \underline{\C}_0 - H_{13}, &
j^*(t\chi_1) &= \underline{\C}_1(\underline{\C}_0 - H_{13}), \\
j^*(\chi_2) &= \underline{\C}_0 - H_{23}, &
j^*(t\chi_2) &= \underline{\C}_1(\underline{\C}_0 - H_{23}), \\
j^*(\sigma) &= \underline{\C}_0 - L_3, &
j^*(\sigma\chi_1\chi_2) &= (\underline{\C}_0 - L_3)(\underline{\C}_0 - H_{12}).
\end{align*}
\end{itemize}
\end{prop}

\begin{proof}
By Lemma \ref{lem:additive_basis_circle_K}, we know that the group $K^0_{\Z/2}(\tilde{T}^2)$ has the geometric basis as described in (b). Also, by using Lemma \ref{lem:additive_basis_3_torus_K}, we can see that the group $K^1_{\pm}(\tilde{T}^2)$ has the geometric basis as described in (b). Since the push-forwards are compatible with restrictions, we have $j^*(\chi_i) = \underline{\C}_0 - H_{i3}$ for $i = 1, 2$. In the same way as the case of $\tilde{T}^2$, we have $j^*(\sigma) = \underline{\C}_0 - L_3$. The pull-back under $\pi_i^*$ is a ring monomorphism $\mathbb{K}^*(\tilde{S}^1) \to \mathbb{K}^*(\tilde{T}^2)$. This fact further allows us to identify $\sigma \chi_i = \underline{\C}_0 - L_i$. Noting that $K^1_\pm(\tilde{T}^2)$ is a $K^0_{\Z/2}(\tilde{T}^2)$-module and using Lemma \ref{lem:additive_basis_3_torus_K}, we have
$$
j^*(\sigma \chi_1 \chi_2)
= j^*((\sigma \chi_1)\chi_2)
= (\underline{\C}_0 - L_1)(\underline{\C}_0 - H_{23})
= (\underline{\C}_0 - L_3)(\underline{\C}_0 - H_{12}).
$$
From this, $\{ \chi_1, t\chi_1, \chi_2, t\chi_2, \sigma, \sigma \chi_1 \chi_2 \}$ is an additive basis of $\tilde{K}^1_\pm(\tilde{T}^2)$, and we have $h \sigma = \sigma \chi_1 \chi_2$, where we put $h = \underline{\C}_0 - H$. That $\sigma(\chi_1\chi_2 - h) = 0$ allows us the expression $\chi_1\chi_2 = h + \alpha (1 + t) h + \beta (1 + t)$ for some $\alpha, \beta \in \Z$. By using the formula $h \chi_1 = h \sigma$, which can be verified via $F$, we have
\begin{align*}
&
\chi_1(\chi_1\chi_2) 
= \sigma \chi_1 \chi_2 = \sigma h, \\
& \chi_1(h + \alpha (1 + t) h + \beta (1 + t))
= \sigma h + \beta (1 + t) \chi_1,
\end{align*}
so that $\beta = 0$ and $\chi_1\chi_2 = h + \alpha (1 + t) h$. We here consider the push-forward $(\pi_1)_* : K^0_{\Z/2}(\tilde{T}^2) \to K^1_\pm(\tilde{S}^1)$. Because $(\pi_1)_*(\sigma \chi_2) = \sigma \chi$, the push-forward restricts to give an $R$-module isomorphism $\Z h \oplus \Z th \to \Z \chi \oplus \Z t \chi$. This implies $(\pi_1)_*(h) \in \{ \pm \chi, \pm t \chi \}$. Since $(\pi_1)_*(\chi_1\chi_2) = \chi$, we conclude $\alpha = 0$ and $\chi_1\chi_2 = h$. This completes the proof of (a). From the construction, (b) is obvious.
\end{proof}

\begin{cor} \label{cor:fourier_transform}
Let $\pi_i : \tilde{S}^1 \times \tilde{S}^1 \to \tilde{S}^1$ be the $i$th projection. Then the map
\begin{align*}
T &: \ \mathbb{K}^*(\tilde{S}^1) 
\longrightarrow 
\mathbb{K}^*(\tilde{S}^1), &
& a \mapsto (\pi_2)_*((1 + t \chi_1\chi_2)\pi_1^*a)
\end{align*}
is an isomorphism of $\mathbb{K}^*(\mathrm{pt})$-modules.
\end{cor}

\begin{proof}
By direct computations, we get
\begin{align*}
T(1) &= t \chi, & 
T(t) &= \chi, & 
T(\sigma\chi) &= \sigma - (1 - t)\chi, \\
T(\chi) &= 1 - \sigma \chi, &
T(t\chi) &= t + \sigma \chi, &
T(\sigma) &= - \sigma \chi.
\end{align*}
This shows that $T$ carries a basis to another basis.
\end{proof}

\begin{rem}
$T^2 \neq 1$, $T^4 = t$ and $T^8 = 1$, as can be verified by:
\begin{align*}
T^2(1) &= t + \sigma \chi, & 
T^2(t) &= 1 - \sigma \chi, & 
T^2(\sigma\chi) &= - 1 + t + \sigma \chi, \\
T^2(\chi) &= \chi - \sigma, &
T^2(t\chi) &= t \chi + \sigma, &
T^2(\sigma) &= \chi - t\chi - \sigma.
\end{align*}
\end{rem}


\section{Topological T-duality for Real circle bundles}
\label{sec:topological_T_duality}

This section contains the proof of our main theorems: Theorem \ref{thm:main_pair} and \ref{thm:main_T_transformation}. We also give an example illustrating the main theorems, a construction of the classifying space for pairs, and finally a possible topological T-duality for future work.

\subsection{The notion of pairs}
\label{subsec:notion_of_pairs}

Following \cite{B-S}, we introduce the notion of pairs.

\begin{dfn}
Let $X$ be a space with $\Z/2$-action.
\begin{itemize}
\item[(a)]
A \textit{pair} $(E, h)$ on $X$  consists of a Real circle bundle $\pi : E \to X$ and an equivariant cohomology class $h \in H^3_{\Z/2}(E; \Z)$.

\item[(b)]
An \textit{isomorphism} $\Phi : (E, h) \to (E', h')$ of pairs on $X$ is an isomorphism of Real circle bundle $\Phi : E \to E'$ covering the identity on $X$ such that $h = \Phi^*h'$.

\item[(c)]
For a $\Z/2$-equivariant map $f : X' \to X$ from another space $X'$ with $\Z/2$-action, the \textit{pull-back} of a pair $(E, h)$ on $X$ under $f$ is the pair $f^*(E, h)$ on $X'$ consisting of the Real circle bundle $f^*E \to X'$ and $F^*h \in H^3_{\Z/2}(f^*E; \Z)$, where $F : f^*E \to E$ is the bundle map covering $f$.
\end{itemize}
\end{dfn}

\medskip

For any Real circle bundle $\pi : E \to X$, we mean by a gauge transformation an equivariant bundle map $\Phi : E \to E$ covering the identity of $X$. The group of gauge transformations of $E$ is in the obvious way isomorphic to the group of $\Z/2$-equivariant maps $\phi : X \to \tilde{S}^1$. We denote by $[X, \tilde{S}^1]_{\Z/2}$ the group of homotopy classes of $\Z/2$-equivariant maps $\phi : X \to \tilde{S}^1$. Recall from Lemma \ref{lem:cohomology_circle_with_flip} that $H^1_\pm(\tilde{S}^1) = \Z \chi \oplus \Z t^{1/2}$. We define a homomorphism $a$ as follows:
\begin{align*}
a &: [X, \tilde{S}^1]_{\Z/2} \to H^1_\pm(X), &
\phi \mapsto \phi^* \chi.
\end{align*}
As will be shown in Appendix, this is an isomorphism (Proposition \ref{prop:classifying_space_H1pm}).

\begin{lem} \label{lem:automorphism_of_pair}
Let $(E, h)$ be a pair on a space $X$ with $\Z/2$-action.
\begin{itemize}
\item[(a)]
Let $\phi : X \to \tilde{S}^1$ be an equivariant map, and $\Phi : E \to E$ the corresponding gauge transformation. Then,
$$
\Phi^*h = h + \pi^*(\pi_*h \cup a(\phi)).
$$
Thus, if $\phi$ is the constant $\phi(x) = -1$, then $\Phi^*h = h + \pi^*\pi_*h \cup t^{1/2}$.

\item[(b)]
Conversely, for any $a \in H^1_\pm(X)$, the pairs $(E, h)$ and $(E, h + \pi^*(\pi_*h \cup a))$ are isomorphic by a gauge transformation $\phi : X \to \tilde{S}^1$. Such a $\phi$ is unique up to equivariant homotopy.
\end{itemize}
\end{lem}

\begin{proof}
The following is a simple generalization of an argument in the proof of Theorem 2.16 in \cite{B-S}: We factor $\Phi : E \to E$ as follows:
$$
E \overset{(1, \pi)}{\longrightarrow}
E \times X \overset{(1, \phi)}{\longrightarrow}
E \times \tilde{S}^1 \overset{m}{\longrightarrow}
E,
$$
where $m : E \times \tilde{S}^1 \to E$, ($m(\xi, u) = \xi u$) is a bundle map covering $\pi : E \to X$:
$$
\begin{CD}
E \times \tilde{S}^1 @>m>> E \\
@V{\pi_E}VV @VV{\pi}V \\
E @>{\pi}>> X.
\end{CD}
$$
By the naturality of the push-forward, we have $(\pi_E)_*(m^* h) = \pi^*(\pi_* h)$. Then, in view of the splitting of the Gysin sequence $H^3_{\Z}(E \times \tilde{S}^1) \cong H^3_{\Z/2}(E) \oplus H^2_{\pm}(E)$, we get $m^* h = h + (\pi^*\pi_*h) \cup \chi$. This leads to $\Phi^*h = h + (\pi^*\pi_*h) \cup (\pi^*\phi^*\chi)$, so that (a) is proved. Then (b) follows from (a), because of the fact that $a : [X, \tilde{S}^1]_{\Z/2} \to H^1_\pm(X)$ is an isomorphism (Proposition \ref{prop:classifying_space_H1pm}).
\end{proof}

\subsection{Uniqueness of T-dual pair}

We prove here the uniqueness of a T-dual pair (Theorem \ref{thm:main_pair} (c)) separately. The idea of the proof is hinted by \cite{Bar2}. 

Given Real circle bundles $E$ and $\hat{E}$ on a space $X$ with $\Z/2$-action, $\pi$, $\hat{\pi}$, $p$, $\hat{p}$ and $q$ are the following projections:
$$
\begin{array}{c@{}c@{}c@{}c@{}c}
 &  & E \times_{X} \hat{E} & & \\
 & {}^p \swarrow &  & \searrow {}^{\hat{p}} & \\
E & & {}_q \downarrow \ \ & & \hat{E} \\
 & {}_\pi \searrow & & \swarrow {}_{\hat{\pi}} & \\
 &  & X. & & 
\end{array}
$$

\begin{lem} \label{lem:key_to_uniqueness}
Let $\pi : E \to X$ and $\hat{\pi} : \hat{E} \to X$ be Real circle bundles on a space $X$ with $\Z/2$-action. For any $\mu \in H^3_{\Z/2}(\hat{E})$ such that $\hat{\pi}_*\mu = 0$ and $\hat{p}^*\mu = 0$, there exists $\alpha \in H^1_\pm(X)$ such that $\mu = \hat{\pi}^*(c_1^R(E) \cup \alpha)$.
\end{lem}

\begin{proof}
We consider the following maps of fiber bundles:
$$
\begin{CD}
E \times_X \hat{E} @= E \times_X \hat{E} @>{\hat{p}}>> \hat{E} \\
@V{\hat{p}}VV @V{q}VV @V{\hat{\pi}}VV \\
\hat{E} @>{\hat{\pi}}>> X @= X.
\end{CD}
$$
Being fiber preserving, these maps induce the following maps of the Leray-Serre spectral sequences compatible with the maps of the total spaces:
$$
\begin{array}{ccc}
E_r^{p, q}(\hat{\pi}) & \Longrightarrow & H^*_{\Z/2}(\hat{E}) \\
\downarrow & & {} \ \downarrow \hat{p}^* \\
E_r^{p, q}(q) & \Longrightarrow & H^*_{\Z/2}(E \times_X \hat{E}) \\
\downarrow & & \parallel \\
E_r^{p, q}(\hat{p}) & \Longrightarrow & H^*_{\Z/2}(E \times_X \hat{E}).
\end{array}
$$
The $E_2$-terms of these spectral sequences are as follows:
\begin{align*}
E_2^{p, q}(\hat{\pi}) &= H^p_{\Z/2}(X; H^q(\tilde{S}^1)), \\
E_2^{p, q}(q) &= H^p_{\Z/2}(X; H^q(\tilde{S}^1 \times \tilde{S}^1)), \\
E_2^{p, q}(\hat{p}) &= H^p_{\Z/2}(\hat{E}; H^q(\tilde{S}^1)),
\end{align*}
where the cohomology in the coefficients are regarded as $\Z/2$-modules:
$$
\begin{array}{c|c|c|c}
 & q = 0 & q = 1 & q = 2 \\
\hline
H^q(\tilde{S}^1) & \Z(0) & \Z(1) & 0 \\
\hline
H^q(\tilde{S}^1 \times \tilde{S}^1) & \Z(0) & \Z(1) \oplus \Z(1) & \Z(0)
\end{array}
$$
Now, from the first two spectral sequences, we get the map of exact sequences:
$$
\begin{CD}
0 @>>> 
E^{3, 0}_\infty(\hat{\pi}) @>{\hat{\pi}^*}>>
H^3_{\Z/2}(\hat{E}) @>>> 
E^{2, 1}_\infty(\hat{\pi}) @>>>
0 \\
@. @VVV @VV{\hat{p}^*}V @VV{\hat{p}^*}V @. \\
0 @>>> 
E^{3, 0}_\infty(q) @>{q^*}>>
H^3_{\Z/2}(E \times_X \hat{E}) @>>> 
E^{2, 1}_\infty(q) @>>>
0.
\end{CD}
$$
The upper exact sequence is identified with the Gysin sequence for $\hat{\pi} : \hat{E} \to X$. Hence our assumption $\hat{\pi}_*\mu = 0$ implies that $\mu \in H^3_{\Z/2}(\hat{E})$ is carried to $0 \in E^{2, 1}_\infty(\hat{\pi})$, so that an element in $E^{3, 0}_\infty(\hat{\pi})$ hits $\mu$. Let $\nu \in E^{3, 0}_2(\hat{\pi}) = H^3_{\Z/2}(X)$ represent the element. Because of the other assumption $\hat{p}^*\mu = 0$ and the commutativity of the diagram, the image of $[\nu] \in E^{3, 0}_\infty(\hat{\pi})$ under $E^{3, 0}_\infty(\hat{\pi}) \to E^{3, 0}_\infty(q)$, which is again represented by $\nu \in E_2^{3,0}(q) = H^3_{\Z/2}(X)$, is trivial. Notice that 
\begin{align*}
E^{3,0}_\infty(q) &= E^{3, 0}_4(q) 
= E^{3,0}_3(q)/\mathrm{Im}[d_3: E^{0,2}_3(q) \to E^{3,0}_3(q)], \\
E^{3,0}_3(q) &= 
E^{3,0}_2(q)/\mathrm{Im}[d_2 : E^{1,1}_2(q) \to E^{3,0}_2(q)], \\
E^{0,2}_3(q) &= \mathrm{Ker}[d_2 : E_2^{0,2}(q) \to E^{2,1}_2(q)].
\end{align*}
Hence we can express $\nu \in H^3_{\Z/2}(X) = E_2^{3,0}(q)$ as
$$
\nu = d_2(\xi^{1,1}) + d_3(\xi^{0,2})
$$
by using some $\xi^{1,1} \in E_2^{1,1}(q) = H^1_\pm(X) \oplus H^1_\pm(X)$ and $\xi^{0,2} \in E_2^{0, 1}(q) \subset H^0_{\Z/2}(X)$. We write $\xi^{1,1} = (\alpha, \hat{\alpha}) \in H^1_\pm(X) \oplus H^1_\pm(X)$. Now, the map $E_2^{p, q}(q) \to E_2^{p, q}(\hat{p})$ of spectral sequences carries the expression of $\nu$ above to 
$$
\hat{\pi}^*\nu = d_2(\hat{\pi}^*\alpha),
$$
where the contribution from $d_3$ disappears, because $E_3^{0,2}(\hat{p}) = 0$. The spectral sequence $E_r^{p, q}(\hat{p})$ induces the Gysin sequence for $\hat{p} : \hat{\pi}^*E \to \hat{E}$, so that we finally get $\hat{\pi}^*\nu = d_2(\hat{\pi}^*\alpha) = \hat{\pi}^*c_1^R(E) \cup \hat{\pi}^*\alpha$. 
\end{proof}

\begin{prop} \label{prop:uniqueness}
Let $X$ be a space with $\Z/2$-action, and $(E, h)$ a pair on $X$. The isomorphism class of a pair $(\hat{E}, \hat{h})$ T-dual to $(\hat{E}, \hat{h})$ is unique.
\end{prop}

\begin{proof}
By definition, we have
\begin{align*}
c_1^R(\hat{E}) &= \pi_*h, &
c_1^R(E) &= \hat{\pi}_*\hat{h}, &
c_1^R(E)c_1^R(\hat{E}) &= 0, &
p^*h &= \hat{p}^*\hat{h}.
\end{align*}
Let $(\hat{E}', \hat{h}')$ be another pair T-dual to $(E, h)$. Because $c_1^R(\hat{E}) = \pi_*h = c_1^R(\hat{E}')$, Real circle bundles $\hat{E}$ and $\hat{E}'$ are isomorphic. Thus, we can assume $\hat{E} = \hat{E}'$. Applying Lemma \ref{lem:key_to_uniqueness} to the difference $\hat{h}' - \hat{h} \in H^3_{\Z/2}(\hat{E})$, we get $a \in H^1_\pm(X)$ such that 
$$
\hat{h}' = \hat{h} + \hat{\pi}^*(c_1^R(E) \cup a)
= \hat{h} + \hat{\pi}^*(\hat{\pi}_*\hat{h} \cup a).
$$
Now, Lemma \ref{lem:automorphism_of_pair} completes the proof.
\end{proof}

\subsection{The proof of Theorem \ref{thm:main_pair}}

For the proof, we show a few lemmas. Given Real circle bundles $\pi : E \to X$ and $\hat{\pi} : \hat{E} \to X$ on a $\Z/2$-space $X$, let $R_{E} \to X$ and $R_{\hat{E}} \to X$ be the associated Real line bundles. We denote by $\pi_S : S(R_E \oplus R_{\hat{E}}) \to X$ the sphere bundle of the Real vector bundle $R_E \oplus R_{\hat{E}}$ of rank $2$, which we may write $S = S(R_E \oplus R_{\hat{E}})$. We have equivariant inclusions $i : E \to S$ and $\hat{i} : \hat{E} \to S$.
$$
\begin{array}{c@{}r@{}c@{}l@{}c}
 &  & E \times_{X} \hat{E} & & \\
 & {}^p \swarrow &  & \searrow {}^{\hat{p}} & \\
E & \underset{i}{\to} \ & S(R_{E} \oplus R_{\hat{E}}) & \
\underset{\hat{i}}{\leftarrow} & \hat{E}.
\end{array}
$$

\begin{lem} \label{lem:homotopy}
There exists a $\Z/2$-equivariant homotopy $i \circ p \sim \hat{i} \circ \hat{p}$.
\end{lem}

\begin{proof}
The same construction as in \cite{B-S} applies: Without loss of generality, we can assume that $i$ and $\hat{i}$ preserve the Hermitian structures. We then define
\begin{gather*}
\tilde{h} \ : (E \times_{X} \hat{E}) \times [0, 1] \longrightarrow
S(R_{E} \oplus R_{\hat{E}}), \\
\tilde{h}(\xi, \hat{\xi}, t)
= ( (\cos \frac{\pi}{2}t) \xi, (\sin \frac{\pi}{2}t) \hat{\xi} ),
\end{gather*}
which apparently has the required property.
\end{proof}

Given $E$ and $\hat{E}$ as before, let $\pi_P : P(R_E \oplus R_{\hat{E}}) \to X$ be the projective space bundle. We may write $P =  P(R_E \oplus R_{\hat{E}})$ for brevity. The total space of the sphere bundle $\pi_S : S \to X$ gives rise to that of a Real circle bundle $\varpi : S \to P$. It holds $\pi_S = \pi_P \circ \varpi$. The inclusions $i$ and $\hat{i}$ induce sections $r : X \to P$ and $\hat{r} : X \to P$ of $\pi_P$, respectively, giving the commutative diagram:
$$
\begin{CD}
E @>{i}>> S(R_{E} \oplus R_{\hat{E}}) @<{\hat{i}}<< \hat{E} \\
@V{\pi}VV @VV{\varpi}V @VV{\hat{\pi}}V \\
X @>{r}>> P(R_{E} \oplus R_{\hat{E}}) @<{\hat{r}}<< X.
\end{CD}
$$
Note that the Leray-Hirsch theorem for Real vector bundles \cite{K} gives
$$
\mathbb{H}^*(P(R_{E} \oplus R_{\hat{E}}))
\cong \mathbb{H}^*(X)[\tilde{c}]/
(\tilde{c}^2 
- c_1^R(R_E \oplus R_{\hat{E}}) \tilde{c} 
+ c_2^R(R_E \oplus R_{\hat{E}})),
$$
where $\tilde{c} = c_1^R(S(R_{E} \oplus R_{\hat{E}})) \in H^2_\pm(P(R_{E} \oplus R_{\hat{E}}))$ satisfies $\pi_P^*\tilde{c} = 1$, and we regard $\mathbb{H}^*(X) \subset \mathbb{H}^*(P)$ using the pull-back $\pi_P^*$. Note also $c_1^R(E) = r^* \tilde{c}$ and $c_1^R(\hat{E}) = \hat{r}^*\tilde{c}$.

\begin{lem} \label{lem:thom_class}
Let $\pi : E \to X$ and $\hat{\pi} : \hat{E} \to X$ be Real circle bundles on a space $X$ with $\Z/2$-action such that $c_1^R(E) c_1^R(\hat{E}) = 0$.

\begin{itemize}
\item[(a)]
There exists $\Th \in H^3_{\Z/2}(S(R_E \oplus R_{\hat{E}}))$ such that:
$$
\varpi_*\Th = - \tilde{c} + \pi_P^*(c_1^R(E) + c_1^R(\hat{E})).
$$ 

\item[(b)]
If $\Th'$ and $\Th$ are as above, then $\Th' - \Th = \pi_S^*\eta$ for some $\eta \in H^3_{\Z/2}(X)$.

\item[(c)]
If $\Th$ is as in (a), then it holds that:
\begin{align*}
(\pi_S)_*\Th &= -1, &
\pi_*(i^*\Th) &= c_1^R(\hat{E}), &
\hat{\pi}_*(\hat{i}^*\Th) &= c_1^R(E), &
p^*i^*\Th &= \hat{p}^*\hat{i}^*\Th.
\end{align*}
\end{itemize}
\end{lem}

\begin{proof}
By the assumption, we have $\tilde{c}( \tilde{c} - \pi_P^*(c_1^R(E) + c_1^R(\hat{E}))) = 0$. Hence the Gysin sequence for $\varpi : S \to P$ ensures the existence of $\Th$, showing (a). For (b), there is $\tilde{\eta} \in H^3_{\Z/2}(P)$ such that $\Th' - \Th = \varpi^*\tilde{\eta}$, again by the Gysin sequence. The Leray-Hirsch theorem implies that $H^3_{\Z/2}(P) \cong H^3_{\Z/2}(X) \oplus H^1_\pm(X)$, which allows the expression $\tilde{\eta} = \pi_P^*\eta + \pi_P^*b \cup \tilde{c}$ by using some $\eta \in H^3_{\Z/2}(X)$ and $b \in H^1_\pm(X)$. Since $\varpi^*\tilde{c} = 0$, we get $\pi_S^*\eta = \varpi^*(\tilde{\eta} - \pi_P^*b \cup \tilde{c}) = \varpi^*(\tilde{\eta}) = \Th' - \Th$. For (c), we have $(\pi_S)*\Th = (\pi_P)_*\varpi_*\Th = -1$. The naturality of the Gysin sequence gives
$$
\pi_*(i^*\Th) = r^*(\varpi_*\Th) 
= - c_1^R(E) + c_1^R(E) + c_1^R(\hat{E})
= c_1^R(\hat{E}),
$$
and $\hat{\pi}_*(\hat{i}^*\Th) = c_1^R(E)$ similarly. Finally, $p^*i^*\Th = \hat{p}^*\hat{i}^*\Th$ follows from Lemma \ref{lem:homotopy} and the homotopy invariance of equivariant cohomology.
\end{proof}

\begin{lem} \label{lem:thom_class_for_given_pair}
Let $(E, h)$ be a pair on a space $X$ with $\Z/2$-action, and $\pi : \hat{E} \to X$ a Real circle bundle such that $c_1^R(\hat{E}) = \pi_*h$. Then there is $\Th \in H^3_{\Z/2}(S(R_E \oplus R_{\hat{E}}))$ such that
\begin{align*}
\varpi_*\Th &= - \tilde{c} + \pi_P^*(c_1^R(E) + c_1^R(\hat{E})), &
i^*\Th &= h.
\end{align*}
\end{lem}

\begin{proof}
By the Gysin sequence for $\pi: E \to X$, we have $c_1^R(E)c_1^R(\hat{E}) = 0$. By Lemma \ref{lem:thom_class}, we get $\Th' \in H^3_{\Z/2}(S)$ such that $\varpi_*\Th' = - \tilde{c} + \pi_P^*(c_1^R(E) + c_1^R(\hat{E}))$. Since $\pi_*(i^*\Th' - h) = 0$, there is $\eta \in H^3_{\Z/2}(X)$ such that $i^*\Th' - h = \pi^*\eta$. Then $\Th = \Th' - \pi_S^*\eta$ has the required property.
\end{proof}

\medskip

\begin{proof}[The proof of Theorem \ref{thm:main_pair}]
We construct a pair $(\hat{E}, \hat{h})$ T-dual to $(E, h)$ as follows: We define $\hat{\pi} : \hat{E} \to X$ to be a Real circle bundle such that $c_1(\hat{E}) = \pi_*h$. We have $c_1^R(E)c_1^R(\hat{E}) = 0$ by the Gysin sequence for $\pi: E \to X$, and get $\Th \in H^3_{\Z/2}(S(R_E \oplus R_{\hat{E}}))$ in Lemma \ref{lem:thom_class_for_given_pair}. If we put $\hat{h} = \hat{i}^*\Th$, then $(\hat{E}, \hat{h})$ is a pair T-dual to $(E, h)$ by Lemma \ref{lem:thom_class}. This completes (a) in Theorem \ref{thm:main_pair}, and (b) is already shown as Proposition \ref{prop:uniqueness}. For (c), let $(E', h')$ be another pair on $X$, and $(\hat{E}', \hat{h}')$ a pair T-dual to $(E', h')$. The thing we show eventually is that: if $(E, h)$ and $(E', h')$ are isomorphic, then so are $(\hat{E}, \hat{h})$ and $(\hat{E}', \hat{h}')$. Let $\Phi : (E', h') \to (E, h)$ be an isomorphism. As long as $(\hat{E}, \hat{h})$ is T-dual to $(E, h)$, its isomorphism class is unique by (b). We can say the same thing for $(\hat{E}', \hat{h}')$. Accordingly, we construct the T-dual pairs in the following way: We construct $(\hat{E}, \hat{h})$ as in the proof (a), so that $\hat{h} = \hat{i}^*\Th$ by using a class $\Th$ such as in Lemma \ref{lem:thom_class_for_given_pair}. In constructing $(\hat{E}', \hat{h}')$, we can put $\hat{E}' = \hat{E}$, because $\pi_*h = \pi'_*h'$ by the naturality of the Gysin sequence. The isomorphism $\Phi : E' \to E$ of Real circle bundles induces that of sphere bundles $\Phi_S : S(R_{E'} \oplus R_{\hat{E}}) \to S(R_{E} \oplus R_{\hat{E}})$, fitting into the commutative diagram:
$$
\begin{CD}
E' \times_X \hat{E} @>{\hat{p}'}>> 
\hat{E} @>{\hat{i}'}>> 
S(R_{E'} \oplus R_{\hat{E}}) \\
@V{(\Phi, 1)}VV @| @VV{\Phi_S}V \\
E \times_X \hat{E} @>{\hat{p}}>> 
\hat{E} @>{\hat{i}}>> 
S(R_{E} \oplus R_{\hat{E}}).
\end{CD}
$$
We put $\Th' = \Phi_S^*\Th$ and $\hat{h}' = (\hat{i}')^*\Th'$. Then $(\hat{E}, \hat{h}')$ is a pair T-dual to $(E', h')$. By construction $\hat{h}' = \hat{h}$, so that (c) is proved. We can verify (d) readily. For (e), it suffices to establish that: If $(\hat{E}, \hat{h})$ is a pair T-dual to a given pair $(E, h)$, then a pair $(E', h')$ T-dual to $(\hat{E}, \hat{h})$ is isomorphic to the original pair $(E, h)$. By the characterization of T-dual pairs, we have $c_1^R(E) = c_1^R(E')$. This allows us to assume $E = E'$. By the characterization again, we have $\pi_*(h' - h) = 0$ and further $p^*(h' - h) = 0$. Now, Lemma \ref{lem:key_to_uniqueness} and \ref{lem:automorphism_of_pair} imply that $(E, h)$ and $(E', h')$ are isomorphic as pairs. 
\end{proof}

\medskip

For later convenience, we prove:

\begin{lem} \label{lem:T_dual_is_dual}
Let $X$ be a space with $\Z/2$-action, $(E, h)$ a pair on $X$ and $(\hat{E}, \hat{h})$ a pair $T$-dual to $(E, h)$. Then there exists $\Th \in H^3_{\Z/2}(S(R_E \oplus R_{E'}))$ such that:
\begin{align*}
\varpi_*\Th &= - \tilde{c} + \pi_P^*(c_1^R(E) +  c_1^R(\hat{E}')), &
i^*\Th &= h, &
\hat{i}^*\Th &= \hat{h}.
\end{align*}
\end{lem}

\begin{proof}
By Lemma \ref{lem:thom_class_for_given_pair}, we get $\Th' \in H^3_{\Z/2}(S(R_E \oplus R_{E'}))$ such that
\begin{align*}
\varpi_*\Th &= - \tilde{c} + \pi_P^*(c_1^R(E) +  c_1^R(\hat{E}')), &
i^*\Th &= h.
\end{align*}
Because $(\hat{E}, \hat{i}^*\Th)$ is also T-dual to $(E, h)$, we use Lemma \ref{lem:key_to_uniqueness} to have $a \in H^1_\pm(X)$ such that $\hat{i}^*\Th - \hat{h} = \hat{\pi}^*(c_1^R(E) \cup a)$. Then we can verify that 
$$
\Th = \Th' - \pi_S^*(c_1^R(E) \cup a)
$$ 
has the required property.
\end{proof}

\smallskip

\begin{rem}
The cohomology class $\Th$ is our analogue of a Thom class in the sense of \cite{B-S}. The change of a sign is due to our convention.
\end{rem}

\subsection{The proof of Theorem \ref{thm:main_T_transformation}}

We begin with the case of $X = \pt$. For the trivial Real circle bundle $E = \hat{E} = \tilde{S}^1$, we have $E \times_X \hat{E} = \tilde{S}^1 \times \tilde{S}^1 = \tilde{T}^2$. We write $S = S(R_E \oplus R_{\hat{E}})$ for short. As shown in Lemma \ref{lem:homotopy}, we have the homotopy
\begin{align*}
\tilde{h} \ &: \tilde{T}^2 \times [0, 1] \longrightarrow 
S, &
\tilde{h}(u, \hat{u}, t)
&= ( (\cos \frac{\pi}{2}t) u, (\sin \frac{\pi}{2}t) \hat{u} ),
\end{align*}
Then, for any twist $\tau \in \Twist^0_{\Z/2}(S)$, Corollary \ref{cor:homotopy_construct_isomorphism} constructs the isomorphism of twists $u(\tilde{h}) : p^*i^*\tau_S \to \hat{p}^*\hat{i}^*\tau_S$. The isomorphism class of $\tau$ will be regarded as a cohomology class $[\tau] \in H^3_{\Z/2}(X; \Z) \cong \pi_0(\Twist^0_{\Z/2}(S))$.

\begin{lem} \label{lem:T_transformation_pt}
In the notation above, let $\tau \in \Twist^0_{\Z/2}(S)$ be any twist such that $(\pi_S)_*[\tau] = -1$. For any $n \in \Z$, the following compositions of maps are bijective:
\begin{align*}
&
K^{i^*\tau + n}_{\Z/2}(\tilde{S}^1) \overset{p^*}{\to}
K^{p^*i^*\tau + n}_{\Z/2}(\tilde{T}^2) \overset{u(\tilde{h})^*}{\to}
K^{\hat{p}^*\hat{i}^*\tau + n}_{\Z/2}(\tilde{T}^2) \overset{\hat{p}_*}{\to}
K^{\hat{i}^*\tau + n - 1}_\pm(\tilde{S}^1), \\
&
K^{i^*\tau + n}_{\pm}(\tilde{S}^1) \overset{p^*}{\to}
K^{p^*i^*\tau + n}_{\pm}(\tilde{T}^2) \overset{u(\tilde{h})^*}{\to}
K^{\hat{p}^*\hat{i}^*\tau + n}_{\pm}(\tilde{T}^2) \overset{\hat{p}_*}{\to}
K^{\hat{i}^*\tau + n - 1}_{\Z/2}(\tilde{S}^1).
\end{align*}
\end{lem}

\begin{proof}
Notice that $S \cong \tilde{S}^1 \wedge \tilde{S}^1 \wedge S^1$. This allows us to compute $H^3_{\Z/2}(S) \cong \Z$. Hence the isomorphism class of $\tau$ is unique. Further, the construction of $u(\tilde{h})$ in Corollary \ref{cor:homotopy_construct_isomorphism} is functorial with respect to twists. Thus, for the proof, we can choose a particular $\tau$ constructed as follows: We put $U_0 = S \backslash \hat{i}(\tilde{S}^1)$ and $U_1 = S \backslash i(\tilde{S}^1)$. In view of a handle body decomposition of $S^3$ by two solid tori, we find that $U_0$ and $U_1$ are equivariantly homotopic to $\tilde{S}^1$, and $U_0 \cap U_1$ to $\tilde{T}^2$. We identify an (isomorphism class of) equivariant line bundle on $\tilde{T}^2$ with one on $U_0 \cap U_1$. Then the equivariant open cover $\{ U_0, U_1 \}$ and an equivariant line bundle $K \to \tilde{T}^2$ constitute an equivariant twist on $S$. Note that $f : H^3_{\Z/2}(S) \to H^3(S)$ is an isomorphism. Thus, the choice $K = H^{-1}$ provides a twist $\tau \in \Twist^0_{\Z/2}(S)$ such that $(\pi_S)_*[\tau] = -1$, where $H \to \tilde{T}^2$ is the equivariant line bundle constructed in Subsection \ref{subsec:basis_torus}. Now, the twists $i^*\tau$ and $\hat{i}^*\tau$ on $\tilde{S}^1$ have the same presentation as the trivial twist ${\bf 0}$, and so do $p^*i^*\tau$ and $\hat{p}^*\hat{i}^*\tau$. Applying the construction in Corollary \ref{cor:homotopy_construct_isomorphism}, we find that $u(\tilde{h}) : {\bf 0} \to {\bf 0}$ is realized by $H^{-1} \in \mathrm{Aut}({\bf 0}) \cong H^2_{\Z/2}(\tilde{T}^2)$. In the presentation in Proposition \ref{prop:ring_structure_KK_torus}, we have $H = 1 - \chi_1\chi_2$ and $H^{-1} = 1 + t \chi_1\chi_2$. Thus, Corollary \ref{cor:fourier_transform} completes the proof.
\end{proof}

\begin{lem} \label{lem:T_transformation_Z_2}
The statement in Lemma \ref{lem:T_transformation_pt} holds, upon forgetting $\Z/2$-actions.
\end{lem}

\begin{proof}
We can recover the $K$-ring $K^*(T^2)$ by substituting $t = 1$ and $\sigma = 0$ to the presentation of $\mathbb{K}^*(\tilde{T}^2)$ in Proposition \ref{prop:ring_structure_KK_torus}. Then Corollary \ref{cor:fourier_transform} holds true, even if the $\Z/2$-actions are forgotten. Now, the argument in the proof of Lemma \ref{lem:T_transformation_pt} applies without change.
\end{proof}

Suppose that pairs $(E, h)$ and $(\hat{E}, \hat{h})$ on $X$ are T-dual to each other, and then consider the following maps
$$
\begin{array}{c@{}c@{}cc@{}c@{}c@{}c}
K^{h + n}_{\Z/2}(E) & \ \overset{p^*}{\to} \ &
K^{p^*h + n}_{\Z/2}(E \times_X \hat{E}) & \cong &
K^{\hat{p}^*\hat{h} + n}_{\Z/2}(E \times_X \hat{E}) & & \\
 & & & & \downarrow \ \hat{p}_* & & \\
 & & & &
K^{\hat{h} + \hat{\pi}^*W_3^{\Z/2}(E) + n-1}_{\pm}(\hat{E}) & \ \cong \ &
K^{\hat{h} + n - 1}_\pm(\hat{E}),
\end{array}
$$
where the isomorphism in the upper line comes from $p^*h = \hat{p}^*\hat{h}$, and that in the lower from $W_3^{\Z/2}(\hat{\pi}^*E) = \hat{\pi}^*(c_1^R(E) \cup t^{1/2}) = \hat{\pi}^*(\hat{\pi}_*\hat{h} \cup t^{1/2})$ and the equivariant automorphism $-1 : \hat{E} \to \hat{E}$. (Recall Proposition \ref{prop:Pin_c_structure} and Lemma \ref{lem:automorphism_of_pair}.) We denote the composition of these maps as
$$
T : \ K^{h + n}_{\Z/2}(E) \longrightarrow K^{\hat{h} + n - 1}_\pm(\hat{E}).
$$
In the same way, we define $T$ by exchanging $\Z/2$ and $\pm$. By construction, these $T$ are homomorphisms of $K^*_{\Z/2}(X)$-modules. Further, they are compatible with the $\mathbb{K}^*(X)$-module structures. 

Strictly speaking, $T$ above only makes sense up to automorphisms of twists realizing $h$ and $\hat{h}$: This is because the notation $K^{h + n}_{\Z/2}(E)$ by using $h \in H^3_{\Z/2}(E)$ rather than $\tau \in \Twist^0_{\Z/2}(E)$ specifies the equivalence class of the group only, and the equalities of third cohomology classes in the construction of $T$ only specify isomorphisms of twists up to automorphisms. Despite this ambiguity, we regard $T$ as a map in the following, with specifications of twists and their isomorphisms representing third cohomology classes and their equalities understood.

\begin{proof}[The proof of Theorem \ref{thm:main_T_transformation}]
Without loss of generality, we can assume that $h$ and $\hat{h}$ are such as in Lemma \ref{lem:T_dual_is_dual} for some $\Th \in H^3_{\Z/2}(S)$. Let $\tau_S$ be a twist such that $[\tau_S] = \Th$. If we put $\tau = i^*\tau_S$ and $\hat{\tau} = \hat{i}^*\tau_S$, then Corollary \ref{cor:homotopy_construct_isomorphism} provides us $u(\tilde{h}) : p^*\tau \to \hat{p}^*\hat{\tau}$. Now, it suffices to prove that the composition of 
$$
K^{i^*\tau + n}_{\Z/2}(E) \overset{p^*}{\to}
K^{p^*i^*\tau + n}_{\Z/2}(E \times_X \hat{E}) 
\overset{u(\tilde{h})^*}{\to} \!
K^{\hat{p}^*\hat{i}^*\tau + n}_{\Z/2}(E \times_X \hat{E}) 
\overset{\hat{p}_*}{\to}
K^{\hat{i}^*\tau - \tau(\hat{\pi}^*E) + n - 1}_{\Z/2}(\hat{E})
$$
is bijective, where we use the Real line bundle associated to $\hat{\pi}^*E$ to define $\tau(\hat{\pi}^*E)$. We write the composition of homomorphisms as
$$
T : \ \mathcal{K}^n(E) \longrightarrow \hat{\mathcal{K}}^{n-1}(\hat{E}),
$$
where $\mathcal{K}^*(E) = K^{\tau + *}_{\Z/2}(E)$ and $\hat{\mathcal{K}}^*(\hat{E}) = K^{\hat{\tau} - \tau(\hat{\pi}^*E) + *}_{\Z/2}(\hat{E})$ to suppress notations. 

For a finite $\Z/2$-CW complex $X$, take a sequence of subcomplexes
$$
Y_1 \subset Y_2 \subset Y_3 \subset \cdots \subset Y_n = X
$$
such that $Y_1 = \tilde{e}_1$ and $Y_{i+1} = Y_i \cup \tilde{e}_{i + 1}$, where $\tilde{e}_j$ are $\Z/2$-cells. The $\Z/2$-cell $Y_1$ is, by definition, $\Z/2$-homotopy equivalent to a point $\pt$ or $\partial \tilde{I} = \Z/2$. If $Y_1 \simeq \pt$, then $T$ is bijective by Lemma \ref{lem:T_transformation_pt}. If $Y_1 \simeq \Z/2$, then we take the quotient by the $\Z/2$-action to reduce the problem to the case where the $\Z/2$-actions are forgotten. Hence $T$ is also bijective by Lemma \ref{lem:T_transformation_Z_2}. As a hypothesis of an induction, we assume here that $T$ is bijective on $Y_{k}$. Notice that we can make sense of $T$ for the pair of spaces $(Y_{k+1}, Y_k)$. Being a composition of natural maps, $T$ induces a map of the exact sequences for pairs:
$$
\begin{array}{c@{}c@{}c@{}c@{}c@{}c@{}c@{}c@{}c}
\mathcal{K}^{-n-1}(E_k) &  \to  &
\mathcal{K}^{-n}(E_{k+1}, E_k) &  \to  &
\mathcal{K}^{-n}(E_{k+1}) &  \to  &
\mathcal{K}^{-n}(E_k) &  \to  &
\mathcal{K}^{-n+1}(E_{k+1}, E_k) \\
\downarrow & &
\downarrow & &
\downarrow & &
\downarrow & &
\downarrow \\
\hat{\mathcal{K}}^{-n}(\hat{E}_k) &  \to  &
\hat{\mathcal{K}}^{-n-1}(\hat{E}_{k+1}, \hat{E}_k) &  \to  &
\hat{\mathcal{K}}^{-n-1}(\hat{E}_{k+1}) &  \to  &
\hat{\mathcal{K}}^{-n-1}(\hat{E}_k) &  \to  &
\hat{\mathcal{K}}^{-n}(\hat{E}_{k+1}, \hat{E}_k),
\end{array}
$$
where $E_k = E|_{Y_k}$ and $\hat{E}_k = \hat{E}|_{Y_k}$. By the excision axiom, we have
\begin{align*}
\mathcal{K}^{-n}(E_{k+1}, E_k)
&\cong K^{\tau - n}_{\Z/2}
(E|_{\tilde{e}_{k+1}}, E|_{\partial \tilde{e}_{k+1}}) \\
&\cong K^{\tau - n}_{\Z/2}
(E|_{(\Z_2/H) \times e^{d_{k+1}}}, 
E|_{(\Z_2/H) \times \partial e^{d_{k+1}}}).
\end{align*}
In the above, we express the $\Z/2$-cell $\tilde{e}_{k+1}$ as $\tilde{e}_{k+1} = (\Z_2/H) \times e^{d_{k+1}}$ by using a subgroup $H \subset \Z/2$ and the usual $d_{k+1}$-dimensional cell $e^{d_{k+1}}$. Since $e^{d_{k+1}}$ is (equivariantly) contractible, the Real circle bundle on $(\Z_2/H) \times e^{d_{k+1}}$ is the pull-back of one on $\Z_2/H$. Further, because $H^3_{\Z/2}(\Z_2/H) = 0$, we get
$$
\mathcal{K}^{-n}(E_{k+1}, E_k) \cong 
K^{- n - d_{k+1}}_{\Z/2}(E|_{\Z_2/H}).
$$
Similarly, noting that $H^2_{\pm}(\Z_2/H) = 0$ implies $\tau(E|_{\Z_2/H}) \cong \tau(\underline{\R}_1)$, we get
$$
\hat{\mathcal{K}}^{-n-1}(\hat{E}_{k+1}, \hat{E}_k) \cong 
K^{\tau(\underline{\R}_1) - n -1 - d_{k+1}}_{\Z/2}(\hat{E}|_{\Z_2/H}) \cong
K^{-n-1 - d_{k+1}}_\pm(\hat{E}|_{\Z_2/H}).
$$
Through these isomorphisms, we can identify
$$
T : \ \mathcal{K}^{-n}(E_{k+1}, E_k) \longrightarrow
\hat{\mathcal{K}}^{-n-1}(\hat{E}_{k+1}, \hat{E}_k)
$$
with the map in Lemma \ref{lem:T_transformation_pt} or \ref{lem:T_transformation_Z_2}. Thus, $T$ is bijective on $Y_{k+1}$ by the five lemma, and so is on $X$ by induction. 
\end{proof}

\subsection{Example}

As a simple but non-trivial example of topological T-duality for Real circle bundles, we consider $X = S^1$, the circle with trivial $\Z/2$-action. Its cohomology is already calculated in Subsection \ref{subsec:chern_class_of_Real_vector_bundle}.

Let $\pi_0 : E_0 \to S^1$ be the trivial Real circle bundle $E_0 = S^1 \times \tilde{S}^1$, and $\pi_1 : E_1 \to S^1$ the non-trivial one whose Chern class $c = c_1^R(E_1)$ generates $H^2_\pm(S^1) = \Z/2$. A part of the Gysin sequence for $E_0$ is 
$$
\begin{CD}
H^3_{\Z/2}(S^1) @>{\pi^*_0}>{\mbox{injective}}>
H^3_{\Z/2}(E_0) @>{{\pi_0}_*}>{\mbox{surjective}}>
H^2_{\pm}(S^1), \\
@| @| @| \\
\Z_2 t^{\frac{1}{2}}c @. \Z_2 \pi_0^*(t^{\frac{1}{2}}c) \oplus \Z_2 h_0 @. 
\Z_2 c,
\end{CD}
$$
where we choose $h_0 \in H^3_{\Z/2}(E_0)$ such that ${\pi_0}_*h_0 = c_1^R(E_1)$. A part of the Gysin sequence for $E_1$ is
$$
\begin{CD}
H^3_{\Z/2}(S^1) @>{\pi^*_1}>{\mbox{trivial}}>
H^3_{\Z/2}(E_1) @>{{\pi_1}_*}>{\mbox{bijective}}>
H^2_{\pm}(S^1). \\
@| @| @| \\
\Z_2 t^{\frac{1}{2}}c @. \Z_2 h_1 @. \Z_2 c,
\end{CD}
$$
where $h_1 \in H^3_{\Z/2}(E_1)$ is such that ${\pi_1}_*h_1 = c$.

Then we find the following T-dual relations of pairs:
\begin{align*}
\mathrm{(1)} & &
(E_0, 0) &\longleftrightarrow (E_0, 0), \\
\mathrm{(2)} & &
(E_0, \pi_0^*(t^{1/2}c)) &\longleftrightarrow (E_0, \pi_0^*(t^{1/2}c)), \\
\mathrm{(3)} & &
(E_0, h_0) &\longleftrightarrow (E_1, 0), \\
\mathrm{(4)} & &
(E_0, h_0 + \pi^*(t^{1/2}c)) &\longleftrightarrow (E_1, 0), \\
\mathrm{(5)} & &
(E_1, h_1) &\longleftrightarrow (E_1, h_1).
\end{align*}
Note that the action of $t^{1/2}c \in H^3_{\Z/2}(S^1)$ relates the pairs (1) with (2), and similarly (3) with (4). Note also that $(E_0, h_0)$ and $(E_0, h_0 + \pi^*(t^{1/2}c))$ are isomorphic as pairs.

Now, computations by using the Mayer-Vietoris sequence give us:

\begin{enumerate}
\item[{}]
The $K$-groups of $E_0$:

\begin{itemize}
\item 
$h = 0$,
\begin{align*}
K^{0}_{\Z/2}(E_0)
&\cong R \oplus R/J, &
K^{1}_{\Z/2}(E_0)
&\cong R \oplus R/J, \\
K^{0}_{\pm}(E_0)
&\cong R \oplus R/J, &
K^{1}_{\pm}(E_0)
&\cong R \oplus R/J.
\end{align*}

\item
$h = \pi_0^*(t^{1/2}c)$,
\begin{align*}
K^{h + 0}_{\Z/2}(E_0)
&\cong R/I, &
K^{h + 1}_{\Z/2}(E_0)
&\cong R/J \oplus I/2I, \\
K^{h + 0}_{\pm}(E_0)
&\cong R/J \oplus I/2I, &
K^{h + 1}_{\pm}(E_0)
&\cong R/I. 
\end{align*}

\item
$h = h_0$ or $h = h_0 + \pi_0^*(t^{1/2}c)$, 
\begin{align*}
K^{h + 0}_{\Z/2}(E_0)
&\cong R/I \oplus R/J, &
K^{h + 1}_{\Z/2}(E_0)
&\cong R, \\
K^{h + 0}_{\pm}(E_0)
&\cong R/I \oplus R/J, &
K^{h + 1}_{\pm}(E_0)
&\cong R.
\end{align*}
\end{itemize}

\item[{}] 
The $K$-groups of $E_1$:

\begin{itemize}
\item
$h = 0$, 
\begin{align*}
K^0_{\Z/2}(E_1) &\cong R, &
K^1_{\Z/2}(E_1) &\cong R/I \oplus R/J, \\
K^0_\pm(E_1) &\cong R, &
K^1_\pm(E_1) &\cong R/I \oplus R/J.
\end{align*}

\item
$h = h_1$, 
\begin{align*}
K^{h + 0}_{\Z/2}(E_1) &\cong R/I, &
K^{h + 1}_{\Z/2}(E_1) &\cong R/I, \\
K^{h + 0}_\pm(E_1) &\cong R/I, &
K^{h + 1}_\pm(E_1) &\cong R/I.
\end{align*}
\end{itemize}
\end{enumerate}

\medskip

As another example, we can take the $2$-dimensional sphere $S^2 = S^1 \wedge \tilde{S}^1 = \C P^1$ with the $\Z/2$-action $\tau([z : w]) = [\bar{z}, \bar{w}]$. The Real circle bundles on this sphere, which are lens spaces, can be constructed by gluing two solid tori along the torus $\tilde{T}^2$. Their $K$-groups are computed by studying the action of the mapping class group $SL(2, \Z)$ of $\tilde{T}^2$ on $\mathbb{K}^*(\tilde{T}^2)$, but we leave the detail to the interested readers.

\subsection{Universal pair}
\label{subsec:universal_pair}

In the work of Bunke and Schick \cite{B-S}, the `universal pair' plays a key role. We here generalize their construction. This is unnecessary to the aim of the present paper logically, but would be interesting in its own right. 

\medskip

Let $E\tilde{S}^1 \to B\tilde{S}^1$ be the universal Real circle bundle. For any $\Z/2$-space $X$, we can identify $H^2_\pm(X)$ with the equivariant homotopy classes $[X, B\tilde{S}^1]_{\Z/2}$ of equivariant maps $X \to B\tilde{S}^1$, according to Proposition \ref{prop:classify_Real_circle_bundle}. We define $\H^3_{\Z/2}$ to be $\H^3_{\Z/2} = \Map(E\Z_2, K(\Z, 3))$, where $E\Z_2$ is the total space of the universal $\Z/2$-bundle, and $K(\Z, 3)$ is the Eilenberg-MacLane space. If we let $\Z/2$ act on $\H^3_{\Z/2}$ through the $\Z/2$-action on $E\Z_2$ and the trivial action on $K(\Z, 3)$, then $\H^3_{\Z/2}$ serves as a representing space of the equivariant cohomology $H^3_{\Z/2}$:
\begin{align*}
H^3_{\Z/2}(X)
&= H^3(E\Z_2 \times_{\Z_2} X)
\cong [E\Z_2 \times_{\Z_2} X, K(\Z, 3)] \\
&= [E\Z_2 \times X, K(\Z, 3)]_{\Z/2} 
\cong [X, \Map(E\Z_2, K(\Z, n))]_{\Z/2}
\cong [X, \H^3_{\Z/2}]_{\Z/2}.
\end{align*}
We denote by $\tilde{L}\H^3_{\Z/2} = \Map(\tilde{S}^1, \H^3_{\Z/2})$ the free loop space. We let $\Z/2$ act on a loop $\gamma : \tilde{S}^1 \to \H^3_{\Z/2}$ by $(\tau \gamma)[u] = \tau(\gamma[u^{-1}])$. Note that $\tilde{S}^1$ also acts on $\tilde{L}\H^3_{\Z/2}$ by the reparametrization of loops: $(R_\zeta \gamma)[u] = \gamma[\zeta u]$ for $\zeta \in \tilde{S}^1$. Now, we let $\zeta \in \tilde{S}^1$ act on $B\tilde{S}^1 \times \tilde{L}\H^3_{\Z/2}$ by $(\xi, \gamma) \mapsto (\xi \zeta^{-1}, R_\zeta \gamma)$, and take the quotient to define
$$
\RR = (E\tilde{S}^1 \times \tilde{L}\H^3_{\Z/2})/\tilde{S}^1.
$$
We let $\pi : \E \to \RR$ be the pull-back of the universal Real circle bundle under the map $\RR \to B\tilde{S}^1$ induced from the projection $B\pi : E\tilde{S}^1 \to B\tilde{S}^1$. We then construct ${\bf h} : \E \to \H^3_{\Z/2}$ as follows: An element in $\bf{E}$ is described as $([\xi, \gamma], \eta)$, where $[\xi, \gamma] \in \RR$ and $\eta \in E\tilde{S}^1$ are such that $B\pi(\xi) = B\pi(\eta)$. Using the unique element $u \in \tilde{S}^1$ such that $\xi u = \eta$, we put ${ \bf h}([\xi, \gamma], \eta) = \gamma[u]$. This map is well-defined and $\Z/2$-equivariant. By the representing property of $\H^3_{\Z/2}$, the equivariant map ${\bf h}$ represents a class in $H^3_{\Z/2}(\E)$, which we again denote by ${\bf h}$.

Now, a straight generalization of argument in \cite{B-S} proves that: 

\begin{prop}
For any pair $(E, h)$ on a space $X$ with $\Z/2$-action, there is an equivariant map $\phi : X \to \RR$ such that $\phi^*(E, h)$ is isomorphic to $(E, h)$. Such an equivariant map is unique up to equivariant homotopy.
\end{prop}

Other results in \cite{B-S} can be generalized: For instance, $c = c_1^R(\E)$ and $\hat{c} = \pi_*{\bf h}$ constitute a basis of $H^2_\pm(\RR)$ such that $c \hat{c} = 0$. Moreover, using the basis, we get the following bases of $\mathbb{H}^*(\RR)$ and $\mathbb{H}^*(\E)$ respecting the ring structures:
$$
\begin{array}{|c|c|c|c|c|c|}
\hline
 & 
n = 0 & n = 1 & n = 2 & n = 3 & n = 4 \\
\hline
H^n_{\Z/2}(\RR) &
\Z & 0 & \Z_2 t & 
\Z_2 t^{\frac{1}{2}}c \oplus \Z_2 t^{\frac{1}{2}}\hat{c} & 
\Z c^2 \oplus \Z \hat{c}^2 \oplus \Z_2 t^2 \\
\hline
H^n(\RR) &
\Z & 0 & \Z f'(c) \oplus \Z f'(\hat{c}) & 0 & 
\Z f(c^2) \oplus \Z f(\hat{c}^2) \\
\hline
H^n_{\pm}(\RR) &
0 & \Z_2 t^{\frac{1}{2}} & \Z c \oplus \Z \hat{c} & \Z_2 t^{\frac{3}{2}} & 
\Z_2 t c \oplus \Z_2 t \hat{c} \\
\hline
\end{array}
$$
$$
\begin{array}{|c|c|c|c|c|c|}
\hline
 & 
n = 0 & n = 1 & n = 2 & n = 3 & n = 4 \\
\hline
H^n_{\Z/2}(\E) &
\Z & 0 & \Z_2 t & 
\Z {\bf h} \oplus \Z_2 t^{\frac{1}{2}} \pi^*(\hat{c}) &
\Z \pi^*(\hat{c})^2 \oplus \Z_2 t^2 \\
\hline
H^n(\E) &
\Z & 0 & \Z f'(\pi^*(c)) & \Z f({\bf h}) &
\Z f(\pi^*(\hat{c})^2) \\
\hline
H^n_{\pm}(\E) &
0 & \Z_2 t^{\frac{1}{2}} & \Z \pi^*(\hat{c}) & \Z_2 t^{\frac{3}{2}} &
\Z_2 \pi^*(t\hat{c}) \oplus \Z_2 t^{\frac{1}{2}}{\bf h} \\
\hline
\end{array}
$$

Though will not be detailed anymore, the further generalization of the argument in \cite{B-S} provides us another approach to Theorem \ref{thm:main_pair}.


\subsection{A possible topological T-duality}

To close, we discuss a possible T-duality from a purely mathematical viewpoint: In our main theorem, a Real circle bundle appears as a T-dual of another Real circle bundle. There would exist a topological T-duality such that an equivariant circle bundle appears as a T-dual of another equivariant circle bundle and equivariant cohomology is involved only. (Our proof would be directly generalized to, at least, the $\Z/2$-equivariant case.) In these dualities, the `Real' world and the `equivariant' world are parallel. A question is whether a duality mixing them is possible or not. A trial to keep track with the consistency with the Gysin sequences leads us to:

\begin{conj} \label{conj:T_dual_pair}
For any $\Z/2$-space $X$, the following would hold true.
\begin{itemize}
\item[(a)]
For a pair $(E, h)$ consisting of a $\Z/2$-equivariant circle bundle $\pi : E \to X$ and $h \in H^3_\pm(E)$, there exists a pair $(\hat{E}, \hat{h})$ consisting of a Real circle bundle $\hat{\pi} : \hat{E} \to X$ and $\hat{h} \in H^3_\pm(\hat{E})$ such that
\begin{align*}
c_1^R(\hat{E}) &= \pi_*h, &
c_1^{\Z/2}(E) &= \hat{\pi}_*\hat{h}, &
c_1^{\Z/2}(E) \cup c_1^R(\hat{E}) &= 0, &
p^*h &= \hat{p}^*\hat{h},
\end{align*}
where $c_1^{\Z/2}(E) \in H^2_{\Z/2}(X)$ is the equivariant Chern class of $E$.

\item[(b)]
Conversely, for a pair $(\hat{E}, \hat{h})$ consisting of a Real circle bundle $\hat{\pi} : \hat{E} \to X$ and $\hat{h} \in H^3_\pm(\hat{E})$, there exists a pair $(E, h)$ consisting of a $\Z/2$-equivariant circle bundle $\pi : E \to X$ and $h \in H^3_\pm(E)$ satisfying the relations in (a).

\item[(c)]
The assignments $(E, h) \mapsto (\hat{E}, \hat{h})$ in (a) and $(\hat{E}, \hat{h}) \mapsto (E, h)$ in (b) induce maps on the isomorphism classes of pairs, inverse to each other.
\end{itemize}
\end{conj}

We call $(E, h)$ and $(\hat{E}, \hat{h})$ as above \textit{T-dual pairs}.

\medskip

In view of the conjecture above, $K$-theories are to be twisted by a class in $H^3_\pm$. An example of a $K$-theory admitting such a twist is $\KR$-theory. Actually, in a recent work of Moutuou \cite{Mou}, the notion of \textit{Rg (Real graded) $S^1$-central extension}, or equivalently \textit{Rg bundle gerbes} is developed as a twisting of $\KR$-theory. A particular type of these geometric objects on a space $X$ with $\Z/2$-action is classified by $H^3_\pm(X) \cong H^3_{\Z/2}(X; \Z(1))$. Note that an Rg bundle gerbe on a $\Z/2$-space is essentially equivalent to a bundle gerbe with \textit{Jandl structure} \cite{SSW}, whose role as a twist is studied in \cite{GSW}. Twisted $\KR$-theory is also discussed in \cite{B-Sj,DFM,DMR}.

Now, writing $\KR^{h+n}(X)$ for the $\KR$-theory twisted by $h \in H^3_\pm(X)$, we propose:

\begin{conj} \label{conj:T_transformation}
If $(E, h)$ and $(\hat{E}, \hat{h})$ are T-dual pairs, then we would have:
$$
\KR^{h + n}(E) \cong \KR^{\hat{h} + n - 1}(\hat{E}).
$$
\end{conj}

If $\Z/2$ acts on $X$ freely, then Conjecture \ref{conj:T_dual_pair} holds true, since the pairs of our interest give rise to \textit{T-duality backgrounds} formulated in a paper of Baraglia \cite{Bar1}. Note that, in this case, Conjecture \ref{conj:T_transformation} is compatible with the conjecture for T-dual backgrounds stated in the paper.

For any space $X$ with $\Z/2$-action, we have the trivial equivariant circle bundle $E = X \times S^1$, and the trivial Real circle bundle $\hat{E} = X \times \tilde{S}^1$. Here $S^1$ is the circle with the trivial $\Z/2$-action, and $\tilde{S}^1$ is the circle $S^1 \subset \C$ with the $\Z/2$-action $u \mapsto \bar{u} = u^{-1}$. Clearly, $(E, 0)$ and $(\hat{E}, 0)$ are T-dual. For these pairs, Conjecture \ref{conj:T_transformation} holds true, since we have (cf.\ \cite{Hori}):
\begin{align*}
\KR^n(E) &\cong \KR^n(X) \oplus \KR^{n-1}(X), &
\KR^n(\hat{E}) &\cong \KR^n(X) \oplus \KR^{n+1}(X).
\end{align*}

A generalization of our methods would have a potential to prove the conjectures, which should be the theme of a future work.


\appendix

\section{Classification of Real circle bundles}
\label{sec:appendix}

We here give a proof Proposition \ref{prop:classify_Real_circle_bundle} supplying some details to Kahn's original proof \cite{K}. The way to the classification is parallel to that of equivariant line bundles given in \cite{Bry}. Since Real line bundles and Real circle bundles are essentially the same, we are to classify Real circle bundles.


\subsection{Simplicial space}

A \textit{simplicial space} $X_\bullet$ is a sequence of spaces $\{ X_n \}_{n = 0, 1, \ldots}$ equipped with the `face maps' $\partial_i : X_n \to X_{n-1}$, ($i = 0, \ldots, n)$ and the `degeneracy maps' $s_i : X_{n-1} \to X_n$, ($i = 0, \ldots, n)$ obeying the `simplicial relations':
\begin{align*}
\partial_i \partial_j &= \partial_{j-1}\partial_i \ \ (i < j), &
\partial_is_j &= s_{j-1}\partial_i \ \ (i < j), &
\partial_js_j &= 1 = \partial_{j+1}s_j, & & \\
\partial_is_j &= s_j\partial_{i-1} \ \ (i > j+1), &
s_is_j &= s_{j+1}s_i \ \ (i \le j).
\end{align*}

The simplicial space $\Z_2^\bullet \times X$ we are to concern below is one associated to a space $X$ with $\Z/2$-action. This is defined as a sequence of spaces $\{ (\Z_2)^n \times X \}_{n = 0, 1, \ldots}$. The degeneracy and the face maps are defined by
\begin{align*}
\partial_i(g_1, \ldots, g_n, x)
&= 
\left\{
\begin{array}{ll}
(g_2, \ldots, g_n, x), & (i = 0) \\
(g_1, \ldots, g_ig_{i+1}, \ldots, g_n, x), & (i = 1, \ldots, n-1) \\
(g_1, \ldots, g_{n-1}, g_nx), & (i = n)
\end{array}
\right. \\
s_i(g_1, \ldots, g_{n}, x)
&= (g_1, \ldots, g_{i-1}, 1, g_{i}, \ldots, g_{n}, x),
\end{align*}
where we write $gx \in X$ for the action of $g \in \Z_2 = \{ \pm 1 \}$ on $x \in X$.


\subsection{Sheaf on simplicial space}

A sheaf $\mathcal{S}$ (of abelian groups) on a simplicial space $X_\bullet$ consists of a sequence of sheaves $\mathcal{S}_n$ on $X_n$, ($n = 0, 1, 2, \ldots$) together with homomorphisms $\delta_i : \partial_i^{-1} \mathcal{S}_{n-1} \to \mathcal{S}_n$, ($i = 0, \ldots, n$) and $\sigma_i : s_i^{-1} \mathcal{S}_{n+1} \to \mathcal{S}_{n}$ compatible with the simplicial relations. 

For example, an abelian group $A$ defines the sheaf of $A$-valued functions $\underline{A} = \underline{A}_{X_n}$ on each $X_n$. The inverse image sheaves $\partial_i^{-1}\underline{A}_{X_{n-1}}$ and $s_i^{-1}\underline{A}_{X_{n+1}}$ are isomorphic to $\underline{A}_{X_n}$ in a canonical way. If we set $\delta_i$ and $\sigma_i$ to be the identity map under the identifications, then $\underline{A}_{\bullet} = \{ \underline{A}_{X_n} \}$ is a sheaf on $X_\bullet$. In the similar way, we get the constant sheaf $A_\bullet$ by considering locally constant functions.

\medskip

In the case of the simplicial space $\Z_2^\bullet \times X$ associated to an action of $\Z/2$ on $X$, we can `twist' the sheaf $\underline{A}_\bullet$ by means of an example of a \textit{twisting function} \cite{May2}. For $n \ge 1$, we define a map $\epsilon_n : \Z_2^n \times X \to \Z_2$ by
$$
\epsilon_n(g_1, \ldots, g_n, x) = g_1.
$$
These maps satisfy the relations
\begin{align*}
\epsilon_n \cdot \partial_0^*\epsilon_{n-1} &= \partial_1^* \epsilon_{n-1}, &
\partial^*_i \epsilon_{n-1} &= \epsilon_n, \quad (i \ge 2) \\
s_0^*\epsilon_n &= 1, &
s_i^* \epsilon_n &= \epsilon_{n-1}. \quad (i \ge 1)
\end{align*}
Now, we put $\underline{\tilde{A}}_{\Z_2^n \times X} = \underline{A}_{\Z_2^n \times X}$. Regarding $-1 \in \Z_2 = \{ \pm 1\}$ as the automorphism $-1 : A \to A$ of taking the additive inverse, we define $\tilde{\delta}_i : \partial_i^{-1}\underline{\tilde{A}}_{\Z_2^{n-1} \times X} \to \underline{\tilde{A}}_{\Z_2^n \times X}$ and $\tilde{\sigma}_i : \partial_i^{-1}\underline{\tilde{A}}_{\Z_2^{n+1} \times X} \to \underline{\tilde{A}}_{\Z_2^n \times X}$ to be
\begin{align*}
\tilde{\delta}_i
&= 
\left\{
\begin{array}{lc}
\epsilon_n \circ \delta_i, & (i = 0) \\
\delta_i. & (i \neq 0)
\end{array}
\right. &
\tilde{\sigma}_i
&= 
\left\{
\begin{array}{lc}
\epsilon_n \circ \sigma_i, & (i = 0) \\
\sigma_i. & (i \neq 0)
\end{array}
\right.
\end{align*} 
Explicitly, under the identification $\partial_i^{-1}\underline{\tilde{A}}_{\Z_2^{n-1} \times X} \cong \underline{\tilde{A}}_{\Z_2^n \times X}$, we have
$$
(\tilde{\delta}_1 f)(g_1, \ldots, g_n, x)
= g_1 \cdot f(g_1, \ldots, g_n, x)
$$
for any function $f \in \underline{\tilde{A}}_{\Z_2^n \times X}(U)$ on an open set $U \subset \Z_2^n \times X$. The homomorphisms $\tilde{\delta}_i$ and $\tilde{\sigma}_i$ are compatible with the simplicial relations, and hence we get a new sheaf $\underline{\tilde{A}}_{\bullet}$ on $\Z_2^\bullet \times X$. The same construction applies to $A_\bullet$, yielding $\tilde{A}_\bullet$.


\subsection{Sheaf Cohomology}

To a sheaf $\mathcal{S}_\bullet$ on a simplicial space $X_\bullet$, we can associate the cohomology of $X_\bullet$ with coefficients in $\mathcal{S}_\bullet$. A formulation uses an injective resolution of $\mathcal{S}_\bullet$, but we adapt the \v{C}ech cohomology construction. 

An open cover $\U^\bullet$ of a simplicial space $X_\bullet$ consists of a sequence of open covers $\U^{(n)} = \{ U^{(n)}_\alpha \}_{\alpha \in \mathfrak{A}^{(n)}}$ of $X_n$ such that: (i) the index sets $\mathfrak{A}^{(n)}$ form a simplicial set; and (ii) the following relations of inclusions hold true:
\begin{align*}
\partial_i(U^{(n+1)}_\alpha) &\subset U^{(n)}_{\partial_i(\alpha)}, &
s_i(U^{(n-1)}_\alpha) &\subset U^{(n)}_{s_i(\alpha)}.
\end{align*}

Given an open cover $\U^\bullet$ and a sheaf $\mathcal{S}_\bullet$, we have a double complex:
$$
K^{i, j} = C^j(\U^{(i)}, \mathcal{S}_i)
= \prod_{\alpha_1, \ldots, \alpha_j \in \mathfrak{A}^{(i)}}
\Gamma(U^{(i)}_{\alpha_1} \cap \cdots \cap U^{(i)}_{\alpha_j}, 
\mathcal{S}_i).
$$
A differential on $K^{*, *}$ is the usual \v{C}ech coboundary $\delta : K^{i, j} \to K^{i, j+1}$. The other differential $\partial : K^{i, j} \to K^{i+1, j}$ is given by $\partial = \sum_{k = 0}^{i+1}(-1)^k \partial_k^*$, where $\partial_k^*$ is the abbreviation of the composition of the natural map $\mathcal{S}_i \to \partial_k^{-1}\mathcal{S}_i$ and $\tilde{\partial}_k : \partial_k^{-1} \mathcal{S}_i \to \mathcal{S}_{i+1}$. These two differential combine to give a differential $D$ on the total complex of $C^n(\U^\bullet; \mathcal{S}_\bullet) = \oplus_{i + j = n} K^{i, j}$. The convention of this paper is that $D$ acts on the component $K^{i, j}$ by $D = \partial + (-1)^i \delta$. We denote the cohomology of the complex above as $H^*(\U^\bullet; \mathcal{S}_\bullet)$. 

Now, the cohomology of $X_\bullet$ with coefficients in $\mathcal{S}_\bullet$ is defined by
$$
H^n(X_\bullet; \mathcal{S}_\bullet)
= \varinjlim H^n(\U^\bullet; \mathcal{S}_\bullet),
$$
where, as usual, the limit is taken with respect to the direct system formed by open covers $\U^\bullet$ of $X_\bullet$ and their refinements.

\medskip

Now, by modifying an argument in \cite{K}, we can find a relation between a sheaf cohomology on $\Z_2^\bullet \times X$ and one on $E\Z_2 \times_{\Z_2} X$ as follows: A \textit{$\Z/2$-equivariant sheaf} $\mathcal{S}$ on $X$ can be defined as a sheaf $\mathcal{S}$ on $X$ equipped with a lift $\tau^{-1}\mathcal{S} \to \mathcal{S}$ of the $\Z/2$-action \cite{Gro}. To such an equivariant sheaf, we can associate a sheaf $\mathcal{S}_\bullet$ on $\Z_2^\bullet \times X$. Then, since $E\Z_2$ is contractible, the second projection $\pi : E\Z_2 \times X \to X$ induces an isomorphism
$$
H^*(\Z_2^\bullet \times X; \mathcal{S}_\bullet) \cong 
H^*(\Z_2^\bullet \times (E\Z_2 \times X); \pi^*\mathcal{S}_\bullet).
$$
Further, since $\Z/2$ acts on $E\Z_2 \times X$ freely, we get an isomorphism
$$
H^*(\Z_2^\bullet \times (E\Z_2 \times X); \pi^*\mathcal{S}_\bullet)
\cong
H^*(E\Z_2 \times_{\Z_2} X; \mathcal{S}^*).
$$
In the above, the sheaf $\mathcal{S}^*$ is associated to $U \mapsto \mathcal{S}^*(U) = \pi^*\mathcal{S}(V)$ for an open set $U \subset E\Z_2 \times_{\Z_2} X$, where $V \subset X$ is an open set such that $p^{-1}(U) = \bigsqcup_{g \in \Z_2}g(V)$ and $p : E\Z_2 \times X \to E\Z_2 \times_{\Z_2} X$ is the projection. 

As a particular case, we consider the equivariant sheaf $\Z(n)$ on $X$, where the lift $\tau^{-1}\Z(n) \to \Z(n)$ is to multiply $(-1)^n$. Upon the identification of sheaf cohomology and singular cohomology on $E\Z_2 \times_{\Z_2} X$, we eventually get the natural isomorphisms
\begin{align*}
H^*(\Z_2^\bullet \times X; \Z_\bullet) &\cong H^*_{\Z/2}(X; \Z(0)), &
H^*(\Z_2^\bullet \times X; \tilde{\Z}_\bullet) & \cong H^*_{\Z/2}(X; \Z(1)).
\end{align*}


\subsection{Vanishing}

\begin{lem} \label{lem:vanising}
For any integer $n \ge 1$, the following holds true:
\begin{itemize}
\item[(a)]
$H^n(\Z_2^\bullet \times X; \underline{\tilde{\R}}_\bullet) = 0$.

\item[(b)]
$H^n(\Z_2^\bullet \times X; \underline{\tilde{S}}^1_\bullet) \cong H^{n+1}(\Z_2^\bullet \times X; \tilde{\Z}_\bullet)$.
\end{itemize}
\end{lem}

\begin{proof}
For (a), a filtration of the double complex defining $H^*(\Z_2^\bullet \times X; \underline{\tilde{\R}}_\bullet)$ induces a spectral sequence
$$
E_2^{p, q}
= H^p_{\mathrm{group}}(\Z_2; H^q(X; \underline{\tilde{\R}}))
\Longrightarrow 
H^*(\Z_2^\bullet \times X; \underline{\tilde{\R}}_\bullet),
$$
Since $H^q(X; \underline{\R}) = 0$ for $q \ge 1$, the spectral sequence degenerates at $E_2$, and we get
$$
H^p_{\mathrm{group}}(\Z_2; \Map(X, \tilde{\R}))
\cong H^p(\Z_2^\bullet \times X; \underline{\tilde{\R}}_\bullet).
$$
For $n \ge 1$ and $f \in C^n_{\mathrm{group}}(\Z_2; \Map(X, \tilde{\R}))$ we define $\tilde{f} \in C^{n-1}_{\mathrm{group}}(\Z_2; \Map(X, \tilde{\R}))$ by
$$
\tilde{f}(g_1, \ldots, g_{n-1})
= \frac{1}{2}\{ f(1, g_1, \ldots, g_{n-1}) - f(-1, g_1, \ldots, g_{n-1}) \}.
$$
If $f$ is a group cocycle, i.e.\ $\partial f = 0$, then $\partial \tilde{f} = f$, which completes the proof of (a). Now, the long exact sequence associated to the short exact sequence of sheaves
$$
0 \longrightarrow 
\tilde{\Z}_\bullet \longrightarrow 
\underline{\tilde{\R}}_\bullet \overset{\exp 2\pi i(\cdot)}{\longrightarrow} 
\underline{\tilde{S}}^1_\bullet \longrightarrow 
0
$$
leads to the natural isomorphism in (b).
\end{proof}

\medskip

As a simple application of the lemma above, we here prove:

\begin{prop} \label{prop:classifying_space_H1pm}
The homomorphism $a : [X, \tilde{S}^1]_{\Z/2} \to H^1_\pm(X)$ introduced in Subsection \ref{subsec:notion_of_pairs} is bijective.
\end{prop}

\begin{proof}
Since $H^1(\Z_2^\bullet \times X; \tilde{\Z}_\bullet) \cong H^1_\pm(X)$, we prove that the homomorphism
$$
a : [X, \tilde{S}^1]_{\Z/2} \longrightarrow
H^1(\Z_2^\bullet \times_{\Z_2} X; \tilde{\Z}_\bullet).
$$
constructed in the same way as in Subsection \ref{subsec:notion_of_pairs} is bijective. For this aim, we notice the exact sequence of sheaves
$$
0 \longrightarrow 
\tilde{\Z}_\bullet \longrightarrow 
\underline{\tilde{\R}}_\bullet \longrightarrow
\underline{\tilde{S}}^1_\bullet \longrightarrow 
0.
$$
By the help of Lemma \ref{lem:vanising}, we get the exact sequence
$$
\Map(X, \tilde{\R})_{\Z/2} \longrightarrow
\Map(X, \tilde{S}^1)_{\Z/2} \longrightarrow
H^1(\Z_2^\bullet \times X; \tilde{\Z}_\bullet) \longrightarrow
0,
$$
where the subscripts $\Z/2$ mean that we are considering $\Z/2$-equivariant maps. Upon taking the equivariant homotopy, the exact sequence above further induces
$$
[X, \tilde{\R}]_{\Z/2} \longrightarrow
[X, \tilde{S}^1]_{\Z/2} \overset{\alpha}{\longrightarrow}
H^1(\Z_2^\bullet \times X; \tilde{\Z}_\bullet) \longrightarrow
0.
$$
An inspection by constructing a basis of $H^1(\Z_2^\bullet \times \tilde{S}^1; \tilde{\Z}_\bullet)$ shows that $\alpha$ above agrees with $a$, so that $a$ is surjective. Because $[X, \tilde{\R}]_{\Z/2} = 0$ clearly, $a$ is injective.
\end{proof}


\subsection{Classification}

\begin{lem} \label{lem:reformulate_Real_circle_bundle}
The following notions are equivalent.
\begin{enumerate}
\item[(a)]
A Real circle bundle $\pi : P \to X$,

\item[(b)]
A principal $S^1$-bundle $\pi : P \to X$ equipped with a fiber preserving map $\phi : \partial_0^*P \to \partial_1^*P$ on $\Z_2 \times X$ such that:
\begin{itemize}
\item
$\phi(p u) = \phi(p) u^{\epsilon_1(g, x)}$ for $(g, x) \in \Z_2 \times X$, $p \in \partial_0^*P|_{(g, x)}$ and $u \in S^1$.

\item
The following diagram is commutative:
$$
\begin{CD}
\partial_0^* \partial_0^* P @>{\partial_0^*\phi}>>
\partial_0^* \partial_1^* P @>{\cong}>>
\partial_2^* \partial_0^* P \\
@V{\cong}VV @. @VV{\partial_2^*\phi}V \\
\partial_1^* \partial_0^* P @>{\partial_1^*\phi}>>
\partial_1^* \partial_1^* P @>{\cong}>>
\partial_2^* \partial_1^* P.
\end{CD}
$$
\end{itemize}
\end{enumerate}
\end{lem}

\begin{proof}
Given a Real circle bundle $P$, we express $\partial_0^*P$ and $\partial_1^*P$ explicitly:
\begin{align*}
\partial_0^*P
&= \{ (g, x, p) \in \Z_2 \times X \times P |\ x = \pi(p) \}, \\
\partial_1^*P
&= \{ (g, x, p) \in \Z_2 \times X \times P |\ gx = \pi(p) \}.
\end{align*}
Using the $\Z/2$-action $\Z_2 \times P \to P$, ($(g, p) \mapsto gp$), we define $\phi : \partial_0^*P \to \partial_1^*P$ by $\phi(g, x, p) = (g, x, gp)$. Then $\phi$ satisfies the properties in (b). Performing the construction above in the converse way, we can recover from $\phi$ in (b) a $\Z/2$-action on $P$ so as to be a Real circle bundle.
\end{proof}

\begin{prop}
For any space $X$ with $\Z/2$-action, the group of isomorphism classes of Real circle bundles on $X$ is isomorphic to $H^1(\Z_2^\bullet \times X; \underline{\tilde{S}}^1_\bullet)$.
\end{prop}

\begin{proof}
Suppose that a Real circle bundle $\pi : P \to X$ is given. In the same way as in \cite{A}, such a bundle is of locally trivial. Hence we can choose an open cover $\U^\bullet$ of $\Z_2^\bullet \times X$ and local sections $s_\alpha : U^{(0)}_\alpha \to P|_{U^{(0)}}$. We define $g^{(0)}_{\alpha_0\alpha_1}$ and $g^{(1)}_{\alpha_0}$ by
\begin{align*}
g^{(0)}_{\alpha_0 \alpha_1} &: 
U^{(0)}_{\alpha_0} \cap U^{(0)}_{\alpha_1} \to S^1, &
s_{\alpha_1} &= s_{\alpha_0} g^{(0)}_{\alpha_0 \alpha_1}, \\
g^{(1)}_{\alpha_0} &: U^{(1)}_{\alpha_0} \to S^1, &
\phi \circ (\partial_0^*s)_{\alpha_0} &= 
(\partial_1^*s)_{\alpha_0} g^{(1)}_{\alpha_0},
\end{align*}
where $\phi$ is as in Lemma \ref{lem:reformulate_Real_circle_bundle}. We can verify $(g^{(0)}_{\alpha_0\alpha_1}, g^{(1)}_{\alpha_0}) \in C^1(\U^\bullet; \underline{\tilde{S}}^1_\bullet)$ is a cocycle. Its cohomology class $c(P)$ in $H^1(\Z_2 \times X; \underline{\tilde{S}}^1_\bullet)$ is seen to be independent of the choice of local sections $s_\alpha$ and the open cover $\U^\bullet$. Further, the assignment $P \to c(P)$ induces a homomorphism $c$ from the group of isomorphism classes of Real circle bundles. Now, it is easy to see that $c(P) = 0$ for the trivial Real circle bundle, i.e.\ $P = X \times \tilde{S}^1$ with the $\Z/2$-action $(x, u) \mapsto (\tau(x), u^{-1})$. Hence $c$ is injective. Conversely, given a cohomology class $c$ in $H^1(\Z_2 \times X; \underline{\tilde{S}}^1_\bullet)$, we represent it as a \v{C}ech cocycle $(g^{(0)}_{\alpha_0\alpha_1}, g^{(1)}_{\alpha_0}) \in C^1(\U^\bullet; \underline{\tilde{S}}^1_\bullet)$. As usual, we use $g^{(0)}_{\alpha_0 \alpha_1}$ as a transition function to construct a principal $S^1$-bundle $\pi : P \to X$. We then use $g^{(1)}_{\alpha_0}$ to construct a bundle map $\phi : P \to P$ so that $P$ gives rise to a Real circle bundle. By construction, we have $c(P) = c$, so that $c$ is surjective.
\end{proof}

Recalling Lemma \ref{lem:vanising}, we finally get:

\begin{cor} 
For any space $X$ with $\Z/2$-action, the group of isomorphism classes of Real circle bundles on $X$ is isomorphic to $H^2_{\Z/2}(X; \Z(1)) \cong H^2_\pm(X)$.
\end{cor}

Notice that the cohomology class $c(P) \in H^1(\Z_2^\bullet \times X; \underline{\tilde{S}}^1_\bullet)$ classifying a Real circle bundle $P \to X$ corresponds to $c_1^R(P) \in H^2_\pm(X)$. This is because $f'(c_1^R(P) - c(P)) = 0$ generally and $f' : H^2_\pm(\C P^\infty) \to H^2(\C P^\infty)$ is bijective.


\medskip

\begin{acknowledgment}
I would like to thank D. Tamaki for his interest in my work, valuable discussions and particularly the way to construct a classifying space for equivariant cohomology. 
The author's research is supported by 
the Grant-in-Aid for Young Scientists (B 23740051), JSPS.
\end{acknowledgment}


\end{document}